\documentclass[11pt,reqno,a4paper]{amsart}
\usepackage[utf8]{inputenc}
\usepackage[headings]{fullpage} 
 
\usepackage[english]{babel}
\usepackage[T1]{fontenc}         
\usepackage{lmodern}  
\usepackage{amsmath,amsthm,amsfonts,amssymb,bm,thmtools}           
\usepackage{mathrsfs}

\usepackage{graphicx}
\usepackage{tikz-cd}
\usepackage{rotating}

\usepackage{todonotes}

\definecolor{blue-grey}{HTML}{4A90E2}

\usepackage{url}
\usepackage[colorlinks,backref=page,allcolors=blue-grey]{hyperref}
\usepackage[alphabetic,backrefs,lite]{amsrefs}
\usepackage[capitalize,nameinlink,noabbrev]{cleveref} 
\crefname{page}{page}{pages}

\DeclareMathOperator{\Sym}{\mathbf{Sym}}

\DeclareMathOperator{\CEff}{\mathrm{Eff}} 
\DeclareMathOperator{\Hom}{\mathrm{Hom}} 
\DeclareMathOperator{\HHom}{\mathbf{Hom}} 
\DeclareMathOperator{\Spec}{\mathrm{Spec}} 

\DeclareMathOperator{\Pic}{\mathrm{Pic}} 
\DeclareMathOperator{\PPic}{\mathbf{Pic}} 
\DeclareMathOperator{\Div}{\mathrm{Div}} 
\DeclareMathOperator{\PDiv}{\mathrm{PDiv}} 
\DeclareMathOperator{\PDDiv}{\mathbf{PDiv}} 

\DeclareMathOperator{\pr}{\mathrm{pr}} 
\DeclareMathOperator{\supp}{\mathrm{sup}} 
\DeclareMathOperator{\rk}{\mathrm{rk}}  
\DeclareMathOperator{\rg}{\mathrm{rk}}  
\DeclareMathOperator{\LCM}{\mathrm{lcm}} 

\DeclareMathOperator{\four}{\mathscr{F}}
\DeclareMathOperator{\ev}{\mathrm{ev}}

\newcommand{\KVar}[1]{\mathrm{K_0} \mathbf{Var}_{#1} } 

\DeclareMathOperator{\Fil}{\mathrm{Fil}}

\newcommand{\KCharVar}[2]{\mathrm{K_0} \mathbf{Char}_{#1} \mathbf{Var}_{#2}} 
\newcommand{\CharM}[2]{\mathcal{C}\mathit{har}_{#1}\mathscr{M}_{#2}}
\newcommand{\Char}{\mathcal{C}\mathit{har}}

\newcommand{\PL}{\mathrm{PL}}

\DeclareMathOperator{\degg}{\mathbf{deg}}  
\DeclareMathOperator{\divv}{\mathrm{div}}  

\newcommand{\1}{\mathbf 1}

\newcommand{\aA}{\mathbf{a}} 
\newcommand{\bB}{\mathbf{b}}
  
\newcommand{\dd}{{\bm{d}}}

\newcommand{\tT}{\mathbf{t}}

\newcommand{\CC}{\mathbf{C}}

\newcommand{\KK}{\mathbf{K}}
\newcommand{\PP}{\mathbf{P}}
\newcommand{\RR}{\mathbf{R}}
\newcommand{\ZZ}{\mathbf{Z}}
\newcommand{\NN}{\mathbf{N}}
\newcommand{\TT}{\mathbf{T}}
\newcommand{\LL}{\mathbf{L}}
\newcommand{\GG}{\mathbf{G}}

\newcommand{\OO}{\mathcal O}

\newcommand{\CCC}{\mathscr{C}}

\newcommand{\FFF}{\mathscr{F}}

\newcommand{\MMM}{\mathscr{M}}
\newcommand{\OOO}{\mathscr{O}}

\newcommand{\Kapr}{\mathrm{Kapr}}

\newcommand{\mdeg}{\delta}
\newcommand{\scdot}{\,\cdot \,}
\newcommand{\kb}{k} 

\newcommand{\Fun}{\mathrm{Fun}}

\theoremstyle{plain}
\newtheorem{theorem}{Theorem}[section]
\newtheorem{lemma}[theorem]{Lemma}
\newtheorem{proposition}[theorem]{Proposition}
\newtheorem{cor}[theorem]{Corollary}
\newtheorem{theoremintro}{Theorem}

\theoremstyle{definition}
\newtheorem*{defintro}{Definition}
\newtheorem{definition}[theorem]{Definition}
\newtheorem{notation}[theorem]{Notation}

\newtheorem{example}[theorem]{Example}

\theoremstyle{remark}
\newtheorem{remark}[theorem]{Remark}

\usepackage{marginnote}

\DeclareSymbolFont{yhlargesymbols}{OMX}{yhex}{m}{n} \DeclareMathAccent{\yhwidehat}{\mathord}{yhlargesymbols}{"62}

\title[A motivic Poisson formula for algebraic tori]{A motivic Poisson formula for split algebraic tori \\ with an application to motivic height zeta functions}

\author[M. Bilu]{Margaret Bilu}
\address{Centre de Mathématiques Laurent Schwartz, École polytechnique, France}
\email{margaret.bilu@polytechnique.edu}

\author[L. Faisant]{Loïs Faisant}
\address{KU Leuven, Department of Mathematics, B-3001 Leuven, Belgium}
\email{lois.faisant@kuleuven.be}

\begin{document}

\maketitle


\begin{abstract}
    We prove a motivic version of the Poisson formula on the adelic points of a split algebraic torus
    and apply it to the study of the motivic height zeta function of split projective toric varieties,
    in the context of the motivic Manin-Peyre principle. 
\end{abstract}

\setcounter{tocdepth}{1}
\tableofcontents 


\section*{Introduction}

Harmonic analysis stands as one of the foundational tools in the number theorist’s arsenal for addressing counting problems. Within the framework of the Batyrev-Manin-Peyre conjectures—which predict the distribution of rational points on Fano varieties over global fields—this method has been instrumental since the field’s inception.

Classically,  the counting problem is encoded into a
height zeta function on which harmonic analysis is performed. 
In that context, an important role is played by the Poisson formula. 

The goal of the present work is to state and prove a geometric-motivic analogue of the Poisson formula for split algebraic tori 
and apply it to the study of the motivic height zeta function of toric varieties
over function fields of curves. 
In particular, we recover a motivic stabilisation result 
for the class of the moduli space of high degree morphisms 
from a smooth and geometrically irreducible projective curve 
to a smooth split projective toric variety
that was previously obtained by the second author 
using a different approach  \cite{faisant2025motivic-distribution,faisant2025universaltorsors}.
This work also extends residue-type results of Bourqui \cite{bourqui2009produit}
and Bilu--Das--Howe \cite{bilu-das-howe2022zeta}. 


\subsection*{The classical setting}
Let us recall the classical Poisson formula 
in its form used by Batyrev and Tschinkel in 
their proof of Manin's conjecture 
for toric varieties over number fields
\cites{batyrev-tschinkel1995anisotropic-toric,batyrev-tschinkel1998manin-toric}
and function fields of curves over finite fields \cite{bourqui2011MAMS}.
Let $\mathscr G$ be a locally compact topological commutative group, endowed with a certain Haar measure $ \mathrm d x$,
and let $\mathscr H \subset \mathscr G$ be a cocompact discrete subgroup of $\mathscr G$.
We give ourselves a function $f : \mathscr H \to \mathbf R$ 
and assume that $f$ can be extended to an $\mathrm L^1$-function on the whole of $\mathscr G$.
Finally, 
up to normalising $\mathrm d x$, one can assume that $\mathscr H $ has covolume $1$ in $\mathscr G$. 

Assume that the Fourier transform $\hat f$
of $f$ with respect to $\mathrm d x$ is an $\mathrm L^1$-function
on the group $\mathscr H^\perp$ of characters of $\mathscr G$ which are trivial on $\mathscr H$.
Then 
\begin{equation}\label{equ:classical_poisson_summation}
\sum_{x \in \mathscr H}
f ( x ) 
=
\sum_{\chi \in \mathscr H ^\perp}
\hat f ( \chi ) .
\end{equation}
Given a global field $K$, 
Batyrev--Tschinkel and Bourqui
apply this formula to the adelic points
\[
\mathscr G = U ( \mathbb A_K )
\]
of an algebraic torus $U$ defined over $K$,
with $\mathscr H = U ( K )$ being its set of $K$-rational points
diagonally embedded in $U ( \mathbb A_K )$.


\subsection*{Moduli spaces of curves and motivic Batyrev-Manin-Peyre principle}

Assume now that $K$ is the function field of a geometrically irreducible smooth projective curve $\mathscr C$ defined over an absolute base field $k$ of characteristic zero
and let $X$ be a projective variety defined over $k$.
By the valuative criterion of properness, the set of $K$-rational points of $X$
can be identified with the set of $k$-morphisms ${\mathscr C \to X}$. 
Such morphisms are parametrised by a $k$-scheme $\Hom_k ( \mathscr C , X ) $. 
Moreover, 
any morphism $f : \mathscr C \to X$
induces a linear form $\deg ( f ) \in \Pic ( X )^\vee $,
called the \emph{multidegree},
which sends the class of a line bundle to the degree of its pull-back to the curve $\mathscr C$. 
For any $\delta \in \Pic ( X )^\vee$,
the condition $\deg ( f ) = \alpha$ defines a subscheme 
$\Hom_k^\delta ( \mathscr C , X ) $ of finite type. 
Finally, if $U$ is a non-empty open subset of $X$,
we denote by $\Hom_k^\delta ( \mathscr C , X )_U$ 
the submoduli space whose points correspond to morphisms 
sending the generic point of $\mathscr C$ into $U$. 

Assume that $X$ is a Fano-like variety in the sense of \cite[Definition 1]{faisant2025motivic-distribution} (for example, a Fano variety or a split smooth projective toric variety). 
The motivic Batyrev-Manin-Peyre principle 
is a set of predictions concerning the behaviour 
of the class of $\Hom_k^\delta ( \mathscr C , X )_U$ in a suitable localisation of the Grothendieck ring of $k$-varieties $\KVar k $. 
As a group, $\KVar k $ is defined by generators and relations:
generators are isomorphism classes $[Y]$ of $k$-varieties 
and relations are of the form
\[
[ Y ] - [ Z ] - [ Y - Z  ]
\]
whenever $Z$ is a closed subscheme of a variety $Y$. 
The ring structure of $\KVar k $ 
is given by taking Cartesian products of varieties:
\[
[Y_1 ] [ Y_2 ]
=
[Y_1 \times_k Y_2] . 
\]
The class of the affine line plays a special role and is denoted by $\mathbf L_k$. It is convenient to invert it and work in the localisation 
\[
\mathscr M_k = 
\KVar k [ \mathbf L_k^{-1} ] . 
\]
The ring $\mathscr M_k$
admits a filtration by the virtual dimension
and a corresponding completion $\widehat{\mathscr M_k }^{\dim}$
which allows one to define sums of absolutely convergent series. 

The most basic version of the motivic Batyrev-Manin-Peyre principle 
asks whether the normalised class 
\begin{equation}\label{equation:BMP}
\left [
\Hom_k^\delta ( \mathscr C , X )_U
\right ]
\mathbf L_k^{- \delta \scdot \omega_X^{-1}} 
\end{equation}
stabilises (if necessary, in a well-chosen completion of $\mathscr M_k$) as $\delta \in \Pic ( X )^\vee$
lies inside the movable cone of $X$
and
goes arbitrarily far away from its boundaries. 
This question is closely related to the study of the following series.

\begin{defintro}
Let $\mathbf{T}$ be a family of formal variables indexed by a basis of $\Pic ( X )^{\vee}$. 
The multivariate motivic height 
zeta function is the formal series
\[
\zeta_H^{\mathrm{mot}} ( \TT )
=
\sum_{\delta \in \Pic ( X )^\vee}
\left 
[
\Hom^\delta_\kb ( \CCC , X )_U
\right 
]
\TT^\delta . 
\]
\end{defintro}


\subsection*{Split toric varieties and main results}

In this article we focus on the case of smooth, projective and split toric varieties.

Let $U \simeq \mathbf G_m^n$ be a split torus 
and $X = X_\Sigma$ a split smooth projective variety 
over a field $k$
compactifying $U$, defined by a certain fan $\Sigma$
of the lattice of cocharacters 
$\mathcal X_* ( U ) = \Hom_\mathrm{gp} ( \mathbf G_m , U )$
of $U$. Let $r$ be the Picard rank of $X_\Sigma$ and $\omega_{X_\Sigma}$ its canonical line bundle.

\begin{theoremintro}[Meromorphic continuation]
\label{thm-intro-residue}
There exists an $\eta >0$ 
and an integer $a_{X_\Sigma}$
such that 
the formal series 
\[
( 1 -  ( \LL T )^{^{a_{X_\Sigma}}} )^
{\rk ( \Pic ( X_\Sigma ) )}
\zeta_{H}^{\mathrm{mot}}
\left ( 
        T^{\omega_{X_\Sigma}^{-1}} 
\right )  
\] 
converges for $| T | < \LL_\kb^{-1 + \eta }$,
with its value at $\LL^{-1}$
being a non-zero element of $\widehat{\MMM_\kb}^{\dim}$.
\end{theoremintro}
We refer to \cref{theorem:final-residue-type-result} for a more detailed version. 
The following second statement is compatible with the predictions 
from \cite{faisant2025motivic-distribution}. 

\begin{theoremintro}[Multi-height motivic stabilisation]
\label{thm-intro-stabilisation-multi-height}
As the class
\[
\mdeg \in \Pic ( X_\Sigma )^\vee \cap \CEff ( X_\Sigma )^\vee
\]
goes arbitrarily far away from the boundary of $\CEff ( X_\Sigma )^\vee$, the normalised class 
\[
\left 
[
\Hom^\delta_\kb ( \CCC , X_\Sigma )_U
\right 
]
\LL_{\kb}^{-\, \delta \scdot \omega_{X_\Sigma}^{-1}}
\]
tends to the non-zero motivic Euler product 
\[
\LL_\kb^{( 1 - g ( \CCC ) ) \dim ( X_\Sigma )  } 
\left ( 
\frac{ \left  [ \Pic^0 ( \CCC ) \right ]\LL_\kb^{1-g ( \CCC ) } }{ 
\LL_\kb - 1 
}
\right )^{\rg (\Pic ( X_\Sigma ))}
 \prod_{p\in \CCC} 
\left  (1-\LL_{\kappa (p)}^{-1} \right )^{\rg (\Pic (X_\Sigma))}  \frac{[X_{\kappa ( p )}]}{\LL_{\kappa ( p )}^{\dim ( X_\Sigma )}} .  
\]
\end{theoremintro}


\subsection*{Relation with earlier works} 

\subsubsection*{The motivic Poisson formula for additive groups}\label{subsubsection:additive_Poisson_formula}
The first occurrence of harmonic analysis in the motivic setting was the motivic Poisson formula of Hrushovski and Kazhdan \cite{hrushovski-kazhdan}. While originally formulated in the language of model theory of valued fields, it was translated in \cite{chambert-loir-loeser2016motivic} into a more geometric context, where it corresponds to a motivic analogue of the formula (\ref{equ:classical_poisson_summation}) in the specific case when $\mathscr{G} = \mathbf{G}^n_a(\mathbf{A}_F)$  are the adelic points of the additive group scheme $\mathbf{G}^n_a$ for $F = k(\CCC)$ the function field of a curve, and $\mathscr{H} = \mathbf{G}^n_a(F)$, diagonally embedded into $\mathscr{G}$. In the motivic setting there is no local compactness, and the classical theory of integration is replaced by a form of motivic integration, with the motivic Poisson formula being an equality between two classes in the so-called \emph{Grothendieck ring of varieties with exponentials}. The first ingredient of the construction is the definition of \textit{motivic Schwartz-Bruhat functions}, which are motivic analogues of locally constant functions with compact support, and which in this setting are elements of appropriate relative Grothendieck rings with exponentials. Integration and Fourier transformation operators are defined as pushforwards between such Grothendieck rings.  Thanks to the self-duality properties of the additive group, the objects appearing on both sides of the Poisson formula are of the same nature, given as the image of a summation operator applied to a motivic Schwartz-Bruhat function and its Fourier transform, respectively. The Poisson formula itself is proved by first reducing to (analogues of) characteristic functions of balls, where it follows from the Riemann-Roch theorem and Serre duality of the curve $\CCC$. 
This Poisson formula was used successively in \cite{chambert-loir-loeser2016motivic}, \cite{bilu2023MAMS} and \cite{faisant-additive-groups} to study moduli spaces of curves on equivariant compactifications of vector groups. 

\subsubsection*{Previous work for toric varieties in the motivic setting} The first one to study Manin's problem for split toric varieties in the motivic setting was Bourqui \cite{bourqui2009produit}, obtaining a convergence result for the motivic height zeta function in the Grothendieck ring of Chow motives. Combining Bourqui's work with the notion of motivic Euler product from \cite{bilu2023MAMS} led to a lift of Bourqui's result to the Grothendieck ring of varieties (and even its generalisation to so-called Hadamard convergence) in \cite{bilu-das-howe2022zeta}, and to the full motivic  Batyrev-Manin-Peyre principle in \cite{faisant2025motivic-distribution} (for rational curves) and \cite{faisant2025universaltorsors} (in arbitrary genus). All of these results were based on the universal torsor method, and the present paper is the first one to produce a functioning Fourier theory for algebraic tori in the motivic setting. 


\subsection*{Main ideas}

As for the motivic Poisson formula in the additive setting, the first step is an appropriate variant of the Grothendieck ring of varieties. 
To explain our choice, let us first describe the objects that we will need to work with. 

Let $F = k(\CCC)$ be the function field of a smooth projective curve $\CCC/k$, and for every closed point $v\in \CCC$, denote by $F_v$ the corresponding completion. For the sake of simplicity, we specialize in this paragraph to the case where our torus $U$ is one-dimensional: $U = \GG_m$, so that the local height functions which we want to build motivic analogues of are functions $f:F_v^{\times}\to \CC$ which are invariant modulo the group of units $\OO_v^{\times}$, and thus may be thought of as functions on $F_v^{\times}/\OO_v^\times\simeq \ZZ$ (where the isomorphism is given by the valuation), which also turn out to be finitely supported. Thus, our families of local functions will naturally arise as elements of a relative Grothendieck ring over the constant group scheme $\underline{\ZZ}_{\CCC}.$ For a more general torus $U$, $\ZZ$ will be replaced by the lattice of cocharacters $\mathcal{X}_*(U).$

As opposed to the additive setting described above, there is no self-duality in the multiplicative setting, and Fourier transforms are \emph{a priori} objects of a different nature: they are motivic functions over the Cartier dual $D(\underline{\ZZ}_{\CCC})$, which is isomorphic to $\GG_{m,\CCC}$. 
We introduce in a \textit{Grothendieck ring of varieties with characters}  where characters are encoded into a separate component. More precisely, given an abstract commutative group $M$ that is finitely generated, a typical generator of the ring $\KCharVar{M}{\CCC}$ is of the form
\[
[X\times_\CCC D(\underline{L}_{\CCC}), \chi\mapsto \chi( g ( m ) )]
\]
where $X$ is a variety over $\CCC$, 
$g : M\to L$ is a morphism of groups, 
$m \in M$ and the quotient by cut-and-paste relations is only on the first factor. This ring comes with a corresponding ``integration over characters'' $\int_{D(M)}$ operation which kills the second component when $ g ( m ) \neq 0$, compatibly with the usual orthogonality relations. It is interesting to note that this setup bears some similarities with the one of \cite{cluckers-motivic-mellin}, though a precise comparison is yet to be made. 

We start by building a local Fourier theory using this framework, before achieving the passage from local to global via the notion of motivic Euler product as in \cite{bilu2023MAMS}. The role of the idèles in this setting is played by the group of divisors $\Div(\CCC)$ on the curve $\CCC$, with the multiplicative group of the field being represented by the subgroup of principal divisors $\PDiv(\CCC).$ Both of these groups are obtained as limits, as $\mathsf{S}$ varies, of the corresponding subgroups $\Div_{\mathsf{S}}(\CCC)$ (resp. $\PDiv_{\mathsf{S}}(\CCC)$) with support restricted to a finite set of places $\mathsf{S}$. In this context, our local Fourier theory applies, giving local integration operators $\int_{D(\Div_{\mathsf{S}}(\CCC)/\PDiv_{\mathsf{S}}(\CCC))}$ satisfying a Poisson formula.  From these local integration operators, we then build a global integration operator $\int_{\PDiv^{\perp}}$ which corresponds to integrating over characters trivial on $\PDiv(\CCC)$ and prove a global Poisson formula. The latter may be applied to the motivic height zeta function, rewriting it analogously to the proofs of Batyrev-Tschinkel and Bourqui. 


\subsection*{Organisation of the paper}

In \cref{section:rings-of-varieties},
we start with recalling the classical definitions of rings of varieties before introducing our new Grothendieck ring of varieties with characters. 

We proceed the same way in \cref{section:motivic-Euler-products} where we first recall the definition of motivic Euler products from \cite{bilu2023MAMS} before extending the notion to motivic functions on constant group schemes and on their Cartier duals. 

We define Fourier transforms of motivic functions on constant group schemes in \cref{section:multiplicative-motivic-harmonic-analysis}, where we obtain Fourier inversion and Poisson formulas. This framework is adapted to the toric setting in  \cref{section:toric-harmonic-analysis}, resulting in a global Poisson formula for split tori.
 

We eventually apply our theory in the two final sections:
in \cref{section:application-motivic-height-zeta-function}
we apply the motivic Poisson formula to the motivic height zeta function 
and the final analysis of the contributions of the Fourier transforms is performed in \cref{section:final-analysis}.

\subsection*{Notations and conventions}

To avoid confusion, in this article
$U$ will always be a split $n$-dimensional torus over the absolute base field $k$,
while we will use a capital $T$ for indeterminates and its bold version
$\mathbf T$ for sets of indeterminates. On the other hand, lowercase $t$ will usually denote a local parameter.  

The $k$-curve $\mathscr C$ is assumed to be smooth, projective and geometrically irreducible. 
Moreover, we assume that it admits a $k$-divisor of degree one (this additional assumption is automatically satisfied, for example, if $k = \mathbf F_q$ or if $k$ is algebraically closed).
Its function field is denoted by $F$ and its set $| \mathscr C |$ of closed points is identified with the set of discrete valuations of $F$. For each $v\in |\CCC|$, we denote by $F_v$ the corresponding completion of $F$, with ring of integers $\OO_v$. The adèle ring $\mathbf{A}_F$ is the restricted product of the $F_v$ with respect to the $\OO_v$, and the idèle ring $\mathbf{A}_F^{\times}$ is the restricted product of the multiplicative groups $F_v^{\times}$ with respect to their subgroups $\OO_v^{\times}$. 

We use bold letters to differentiate schemes and group schemes from respectively sets and abstract groups, with the following exception: if $M$ is an abstract group and $S$ a base scheme,
the constant group scheme above $S$ will be denoted by $\underline{M}_S$.

As an abstract group, the group of divisors
$
\Div ( \CCC ) 
$
on the curve $\CCC$ is filtered by the subgroups 
$
\Div_\mathsf S ( \CCC ) 
$
of divisors with support in finite subsets $\mathsf S$ of the set $|\CCC|$ of closed points of $\CCC$, giving 
\[
\Div ( \CCC ) 
= 
\varinjlim_\mathsf S
\Div_\mathsf S ( \CCC ) 
=
\varinjlim_\mathsf S 
\bigoplus_{x \in \mathsf S} \ZZ_{\kappa ( x )}
=
\bigoplus_{x \in | \CCC |} \ZZ_{\kappa ( x )}. 
\]
This filtration restricts to degree-zero divisors
\[
\Div^0 ( \CCC ) 
= 
\varinjlim_\mathsf S
\Div_\mathsf S^0 ( \CCC )
\]
and to principal divisors 
\[
\PDiv ( \CCC ) 
= 
\varinjlim_\mathsf S
\PDiv_\mathsf S ( \CCC ) .
\]
Let $I$ be a set. Elements of the free abelian monoid 
\[
\mathbf{N}^{(I)} = \{ (n_i)_{i\in I}\in \mathbf{N},\ n_i = 0\ \text{for almost all}\ i\}
\]
are called \emph{generalised partitions}. If $I$ is of the form $I_0\setminus \{0\}$ where $I_0$ is a commutative monoid, and $\pi=(n_i)_{i\in I}\in \mathbf{N}^{(I)}$, we may define 
\[
\sum \pi = \sum_{i\in I}i n_i \in I_0.
\]
If $i\in I_0$ is such that $\sum \pi = i$, we say that $\pi$ is a partition of $i$
and write $\pi\vdash i$.


\subsection*{Acknowledgments} 
MB was partially supported by the ANR projects Cyclades (ANR-23-CE40-0011) and AdAnAr (ANR-24-CE40-6184). 
LF acknowledges partial funding from the European Union’s Horizon 2020 research and innovation programme under the Marie Skłodowska-Curie grant agreement No 101034413,
as well as 
partial support by the KU Leuven IF grant C16/23/010. 

The authors thank A.~Chambert-Loir who first proposed this project to MB, as well as D.~Bourqui, R.~Cluckers, M.~Raibaut and E.~Peyre for useful conversations.


\renewcommand{\theequation}{\thesection.\arabic{equation}}
\counterwithin*{equation}{section}


\section{Grothendieck rings}

\label{section:rings-of-varieties}

This section aims to define a new variant 
of the ring of varieties that will allow us to handle certain families of multiplicative characters. 
We first start with recalling the classical definition of ring of varieties over an arbitrary base scheme. 


\subsection{Rings of varieties}

Let $S$ be a scheme. 
An $S$-variety is a finitely presented scheme morphism $X \to S$.

\begin{definition}
The \textit{Grothendieck group of $S$-varieties }
\[
    \KVar S
\] 
is defined as the 
    abelian group generated by the isomorphism classes of \textit{$S$-varieties}, with relations given by the \textit{cut-and-paste relations}
\[
X - Y - U
\] 
whenever $X$ is an $S$-variety, 
$Y $ is a closed subscheme of $X$ 
and $U$ is its open complement in $X$. 
The class of an $S$-variety $X$ in this ring is denoted by $[X \to S]$,
$[X]_S$ or simply $[X]$
if the structural morphism is clear from the context.
Let $\LL_S$ be the class of the affine line. 

A ring structure on $\KVar S $ is defined by  
\[
[X][Y]=[X\times_S Y]
\]
for which $[ \mathrm{id} : S \to S ]$ is the unitary element. 
\end{definition}

A constructible subset
of an $S$-variety
is a finite union of locally closed subsets. 
By \cite[p. 59]{chambert-loir-nicaise-sebag2018motivic},
any constructible subset $X$ of an $S$-variety admits a class $[X]$ in $\KVar S$.

\begin{lemma}[{\cite[Lemma 1.1.8]{chambert-loir-loeser2016motivic}
	or 
	\cite[Theorem 1]{cluckers-halupczok2022evaluation}}]
 \label{lemma-motivic-functions-are-defined-on-points}
	Let $\varphi $ be a motivic function on $S$,
	that is to say an element of $\KVar{S}$.
	If $s^* \varphi = 0 $ 
	in $\KVar{\kappa ( s )}$
	for all point $s \in S$ then $\varphi = 0$ in $\KVar S$.  
\end{lemma}

\begin{definition}
    The localised Grothendieck ring of varieties $\mathscr M_S$ , sometimes called \emph{ring of motivic integration}
 is the localisation of $\KVar S $  at $\LL_S$.
\end{definition}

\begin{definition}

The ring $\mathscr M_S$ admits a decreasing filtration by the virtual dimension: 
for $m\in \ZZ$, let $\mathcal F^m \MMM_S $ be the subgroup of $\MMM_S$ generated by elements of the form
\[
[ X ]\LL_S^{-i} 
\]
where $X$ is an $S$-variety and $i$ an integer such that $\dim_S ( X ) - i \leqslant - m$.
The completion of $\MMM_S$ with respect to this decreasing dimensional filtration is the projective limit
\[
\widehat{\MMM_S}^{\dim } = 
\underset{\longleftarrow}{\lim } \, \MMM_S / \mathcal F^m \MMM_S .
\] 
It comes with a morphism $\MMM_S \to \widehat{\MMM_S}^{\dim }$. 

\end{definition}

\label{paragraph-ring-of-varieties}


\subsection{Biduality for constant commutative group schemes}
A reference for this subsection is \cite[Exp.~VIII]{SGA3-II}.
Let $S$ be a scheme. 
Recall that the Cartier dual of a group scheme $G$ above $S$ 
    is the $S$-group scheme
    \[
    D ( G ) 
    = 
    \Hom_{S-\mathrm{gp}} ( G , \GG_{m,S} )
    \]
    of multiplicative characters of $G$. 
    It is a closed subscheme of $ \Hom_{S-\mathrm{sch}} ( G , \GG_{m,S} )$
    and comes with a canonical pairing of group schemes 
    \[
    \ev : \left \{ 
    \begin{array}{rll}
       G \times_S D ( G )  & \longrightarrow & \GG_m \\
      ( g , \chi ) & \longmapsto & \chi ( g ) .
    \end{array}
    \right.
    \]
A constant commutative group scheme $\Gamma$ above $S$
is an $S$-group scheme of the form 
\[
    \Gamma = \underline M_S = \coprod_M S
\]
where $M$ is an abstract commutative group.
A group scheme $G \to S$ is said to be \emph{diagonalisable}
if it is of the form $G = D ( \underline M_S )$
where $M$ is an abstract commutative group,
and \emph{locally diagonalisable}
if every point of $S$ admits an open neighborhood $U$ such that $G_{|U}$ is diagonalisable.

Let $\Gamma$ be a constant commutative group scheme above $S$.
Then the canonical morphism
    \[
    \Gamma 
    \longrightarrow
    D ( D ( \Gamma ) ) 
    \]
    is an isomorphism \cite[Exp.~VIII, Théorème 1.2]{SGA3-II}. 

\begin{proposition}[{\cite[Exp.~VIII, Prop. 2.1]{SGA3-II}}] 
    Let $G = D ( \underline M_S ) $ be a diagonalisable group scheme over $S$. 
    Then $G$ is faithfully flat and affine over $S$.
    Moreover,
    \begin{itemize}
        \item $M$ is of finite type as an abstract group if and only if $G \to S$ is of finite type ;
        \item $M$ is finite if and only if $G \to S$ is finite, and $M$ is torsion if and only if $G \to S$ is integral ;
        \item $M$ is of finite type, with torsion of order coprime to the residual characteristics of $S$, if and only if $G \to S$ is smooth. 
    \end{itemize}
\end{proposition}

\begin{proposition}\label{prop:Cartier_dual_of_product} 
Let $M, L,L'$ be abstract commutative groups such that there exist group morphisms $M\to L$ and $M\to L'$. We denote by $L\times_ML'$ the corresponding pushout in the category of abelian groups, i.e. the universal object fitting into the diagram
\[
\begin{tikzcd}
 M \arrow[r]\arrow[d] & L\arrow[d]\\
 L' \arrow[r] & L\times_ML'.
\end{tikzcd}
\]
Given a scheme $S$, the canonical morphisms $D(\underline{L\times_ML'}_S)\to D(\underline{L}_S) $ and $D(\underline{L\times_ML'}_{S})\to D(\underline{L'}_S)$ induce an isomorphism
\[ D(\underline{L\times_ML'}_S) \to D(\underline{L}_S)\times_{D(\underline{M}_S)}D(\underline{L'}_S)\]
of $D(\underline{M}_S)$-schemes, explicitly given by $\chi\mapsto (\chi_{|L},\chi_{|L'})$. Moreover, for all $l\in L, l'\in L'$, this isomorphism makes the diagram
\[
\begin{tikzcd}[sep = large]
D(\underline{L\times_ML'}_S) \ar[d,"{\ev((l,l'),\cdot)}"] \ar[r]& D(\underline{L}_S)\times_{D(\underline{M}_S)}D(\underline{L'}_S) \ar[d,"{(\ev(l,\cdot),\ev(l',\cdot))}" ] \\
\GG_m & \GG^2_m \ar[l, "\mathrm{group\ law}"]
\end{tikzcd}
\]
commute.
\end{proposition}


\subsection{Ring of varieties with multiplicative characters}

Let $S$ be a scheme.
\begin{remark}
By \cite[Exp.~VIII, Corollaire 1.3]{SGA3-II},
any $S$-group morphism 
\[
D ( \underline L_S ) \to \mathbf G_{m,S}
\]
is defined by a unique section of $\underline L_S$, that is to say a locally constant map $S \to L$. 
In particular, if $S$ is connected, $D ( \underline L_S ) \to \mathbf G_{m,S}$ is uniquely determined by an element of $L$. 

More generally, if $M$ and $N$ are abstract commutative groups,
the natural morphism 
\[
\Hom_{\text{gp}} ( M , N ) 
\longrightarrow 
\Hom_{S-\text{gp}} ( D ( \underline N_S ) , D ( \underline M_S ) )  
\]
is an isomorphism. 
If moreover $M$ is of finite type, 
the morphism of $S$-functors
\[
\underline{\Hom_\text{gp} ( M , N )}_S
\longrightarrow
\HHom_{S-\text{gp}} ( \underline M_S  ,  \underline N_S ) 
\]
is an isomorphism,
from which one deduces that  
$\underline{\Hom_\text{gp} ( M , N )}_S$,
represents the functor 
\[
\HHom_{{S-\text{gp}}} ( D ( \underline N_S ) , D ( \underline M_S ) ).
\]
Again, if moreover $S$ is connected, 
$\Hom_{S-\text{gp}} ( D ( \underline N_S ) , D ( \underline M_S ) ) $ is isomorphic to 
\[
\Hom_\text{gp} ( M , N ). 
\]
\end{remark}

\begin{definition}

\label{def:new-char-var} Let $M$ be an abstract commutative group. 
The Grothendieck group
of $S$-varieties with   characters on $M$
is the $\ZZ$-module 
\[
\KCharVar{M}{S}
\]
generated by isomorphism classes 
\[
[X \times_S D ( \underline L_S ), h ]_{S \times D ( \underline M_S )}
\]
of diagrams
    \begin{equation}\label{eq:generator-diagram}
    ( X \times_S D ( \underline L_S ) ,  h :  D ( \underline L_X ) \to \GG_m ) = 
\begin{tikzcd}
    X \times_S D ( \underline L_S ) \arrow[d,"{(f, g )}"] \arrow[r,equal] & \rar D ( \underline L_X ) \rar["h"] & \GG_m \\
    S \times_S D ( \underline M_S ) & & 
\end{tikzcd}
\end{equation}
where
\begin{itemize}
    \item $f :X \to S$ is an $S$-variety,
    \item $g$ corresponds by Cartier duality to a 
    morphism $\underline M_S \to \underline L_S$ of constant $S$-groups,
    \item $h \in \Hom_{\mathrm{gp}} ( D ( \underline L_X ) , \mathbf G_{m,X} ) $ is given by the evaluation at an element belonging to the image of $\underline M_S \to \underline L_S$.
\end{itemize}
\[
\left [
\begin{tikzcd}
     X \times_S D ( \underline L_S ) \arrow[d,"{(f,g)}"] \arrow[r,equal] & D ( \underline L_X ) \arrow[r,"h"] & \GG_m \\
    S \times_S D ( \underline M_S ) & &
\end{tikzcd}
\right ] 
\]
\[
=
\left [
\begin{tikzcd}
    U \times_S D ( \underline L_S ) \arrow[d,"{(f_{|U},g)}"] \arrow[r,equal] & D ( \underline L_U ) \arrow[r,"h_{|U}"] & \GG_m \\
    S \times_S D ( \underline M_S ) & &
\end{tikzcd}
\right ] 
+
\left [
\begin{tikzcd}
    Z \times_S D ( \underline L_S ) \arrow[d,"{(f_{|Z},g )}"] \arrow[r,equal] & D ( \underline L_Z ) \arrow[r,"h_{|Z}"] & \GG_m \\
    S \times_S D ( \underline M_S ) & & 
\end{tikzcd}
\right ]
\]
for every Zariski-closed subscheme $Z$ of $X$ with open complement $U$. 
A product law is defined relatively to $S \times_S D ( \underline M_S )$: 
\[
\left [\begin{tikzcd}
    X_1 \times_S D ( \underline{L_1}_S ) \dar  \rar["h_1"] & \GG_m  \\
    S \times_S D ( \underline M_S )   & 
\end{tikzcd}
\right ]
\times
\left [
\begin{tikzcd}
      X_2 \times_S D ( \underline{L_2}_S ) \dar \rar["h_2"] & \GG_m  \\
    S \times_S D ( \underline M_S )   & 
\end{tikzcd}
\right ]
\]
\[
= 
\left [
\begin{tikzcd}[column sep=huge]
    ( X_1 \times_S X_2 ) \times_S ( D ( \underline{L_1}_S ) \times_{D ( \underline{M}_S )} D ( \underline{L_2}_S ) ) \dar \arrow[r, "{( h_1 , h_2 )}"] & \GG_m^2 \rar["\text{group law}"] & \GG_m \\
    S \times_S D ( \underline M_S ) & 
\end{tikzcd}
\right ] 
\]
where the last diagram is a well-defined element of $\KCharVar{M}{S}$ by \cref{prop:Cartier_dual_of_product}. In particular, this endows $ \KCharVar{M}{S}$ with a structure of $\KVar S$-module. 
\end{definition}

\begin{remark} When $M$ is the trivial group, then all of the diagrams of the form (\ref{eq:generator-diagram}) for fixed $X$ over $S$ become isomorphic, so that $\KCharVar{M}{S} =\KVar{S}$. 
\end{remark}

\begin{notation}[Pullback and pushforward on the first factor] 

Let $p:S\to S'$ be a morphism of schemes. Then, as usual with Grothendieck rings, $p$ induces a ring morphism:
\[p^*:  \KCharVar{M}{S'}\to  \KCharVar{M}{S} \]
given by
\[
\left [
\begin{tikzcd}
     X \times_{S'} D ( \underline L_{S'} ) \arrow[d,"{(f,g)}"] \arrow[r,equal] & D ( \underline L_X ) \arrow[r,"h"] & \GG_m \\
    S' \times_{S'} D ( \underline M_{S'} ) & &
\end{tikzcd}
\right ] \mapsto \] \[ \left [
\begin{tikzcd}
     (X\times_{S'}S)\times_S D ( \underline L_S ) \arrow[d,"{(p^{*}f,p^{*}g)}"] \arrow[r,equal] & D ( \underline L_{X\times_{S'}S}) \arrow[r,"p^*h"] & \GG_m \\
    S \times_S D ( \underline M_S ) & &
\end{tikzcd}
\right ]
\]
and, if $p$ is of finite presentation, a group morphism $p_!$:
\[p_!:  \KCharVar{M}{S}\to \KCharVar{M}{S'} \]
given by \[\left [
\begin{tikzcd}
     X \times_S D ( \underline L_S ) =  X \times_{S}(S\times_{S'} D ( \underline L_{S'} ))  \arrow[d,"{(f,g)}"] \arrow[r,equal] & D ( \underline L_X ) \arrow[r,"h"] & \GG_m \\
    S \times_S D ( \underline M_S ) & &
\end{tikzcd}
\right ] \mapsto\] \[\left [
\begin{tikzcd}
     X \times_{S'} D ( \underline L_{S'} ) \arrow[d,"{(p\circ f,g_{S'})}"] \arrow[r,equal] & D ( \underline L_{X_{S'}} ) \arrow[r,"h_{S'}"] & \GG_m \\
    S' \times_{S'} D ( \underline M_{S'} ) & &
\end{tikzcd}
\right ]\]
\end{notation}
 \begin{remark}
 More generally, if $p:S\to S'$
 is a morphism of schemes and $\mathfrak{a}\in \KCharVar{M}{S}$ which is supported over a subscheme of $S$ of finite presentation over $S'$, we may make sense of $p_!\mathfrak{a}\in  \KCharVar{M}{S'}$. We will use this repeatedly in the special case where $p$ is the projection $\underline{M}_S\to S$ and $\mathfrak{a}$ has finite support. \end{remark}

 \begin{notation}[Evaluation map]\label{notation:evaluation_map}
For any $m \in M$,
we write
\[
\ev ( m , \cdot ) 
:=
\left [
S \times_S D ( \underline M_S ) , \chi \mapsto \chi ( m ) 
\right ]_{S \times_S D ( \underline M_S ) }\in \KCharVar{M}{S} .
\]

Note that by definition of the product in $\KCharVar{M}{S}$, we have $\ev(m,\cdot)\ev(n,\cdot) = \ev(m+n,\cdot)$ for all $m,n\in M$. 

We may also define the evaluation pairing
\[\ev:= [\underline{M}_S\times_SD(\underline{M}_S), (m,\chi)\mapsto\chi(m)]\in \mathrm{K_0} \mathbf{Char}_M \mathbf{Var}_{\underline{M}_S}. \]

\end{notation}

Orthogonality relations are encoded by the following morphism of integration with respect to the character variables.

\begin{definition}[Integration with respect to  the second factor]
    \label{def:loc-integration-characters}
    There exists a unique $\KVar S$-linear group morphism
    \[
    \int_{ D ( \underline M_S ) } : 
    \KCharVar{M}{S}
    \to 
    \KVar{S}
    \]
    sending a class $  [ X \times_S D ( \underline L_S ) , \chi \mapsto \chi ( l ) ]$ to $[X]_S$ if $l = 0$ and zero otherwise. 
\end{definition}
For $\psi\in \KCharVar{M}{S}$ we write the image via this morphism either as $\int_{ D ( \underline M_S ) } \psi$ or
\[
\int_{ D ( \underline M_S ) }\psi(\chi)\mathrm{d} \chi
\]
to highlight the interpretation as an integral over characters. 
\begin{remark}
In particular, 
we get the formal identity 
\begin{equation}\label{equation:orthogonality}
\int_{ D ( \underline M_S )}
\ev ( m , \chi  )
\mathrm d 
\mathbf \chi 
= 
\begin{cases}
    1 & \text{ if } m  = 0 \\
    0 & \text{ otherwise}
\end{cases}
\end{equation}
for all $m \in \underline M_S$.

\end{remark}

Let us now consider the following situation: let $N\subset M$ be commutative abstract groups. We identify $D ( M / N )$ with the kernel of the restriction morphism
$\mathrm{res}_{|N} : D ( M ) \to D ( N ) $
so that we have a short exact sequence
\[
0\to D( M / N ) \xrightarrow{i} D(M) \xrightarrow{j} D(N) \to 0.
\]
and a pushforward morphism 
\[
i_{!}: \KCharVar{ M / N }{S}\to \KCharVar{M}{S}.
\]

\begin{definition}\label{def:restricted-local-integration}
We define an integration morphism $\int_{D(\underline{M/N}_{S})}: \KCharVar{M}{S} \to \KVar{S}$ by 
\[
 \int_{ D ( \underline{ M / N }_S ) } 
 =
  \int_{ D ( \underline{M}_S ) } \mathbf 1_{D ( \underline{ M / N }_S )}
\]
where 
\[
\mathbf 1_{D ( \underline{ M / N }_S )} = [ D ( \underline{ M / N }_S ) \hookrightarrow D ( \underline{M}_S  ) , 1 ]_{D(\underline{M}_S)}.
\]
\end{definition}
In other words, given a generator $[X\times_SD(\underline{L}_S),h]$
\[
 \int_{ D ( \underline{ M / N }_S ) } [X\times_SD(\underline{L}_S),h] 
 = \int_{ D ( \underline{M}_S )} [X\times_S(D(\underline{L}_S)\times_{D ( \underline{M}_S )} D ( \underline{ M / N }_S)),h]. 
\]

\begin{remark}
It is straightforward to check that the operator $ \int_{ D ( \underline{ M / N }_S ) } $ is compatible with the operator 
\[
  \int_{ D ( \underline{ M / N }_S ) }: \KCharVar{M/N}{S} \to \KVar{S}
\] defined in \cref{def:loc-integration-characters}, in the sense that one recovers the latter by composing the former with the pushforward map $i_!$ on the second factor. 
\end{remark}

\begin{remark}
In particular, we can check that we get the formal identity 
\begin{equation}\label{equation:orthogonality_with_quotient}
\int_{ D ( \underline{M/N}_S )}
\ev ( m , \chi  )
\mathrm d 
\mathbf \chi 
= 
\begin{cases}
    1 & \text{ if } m \in N \\
    0 & \text{ otherwise}
\end{cases}
\end{equation}
for all $m \in \underline M_S$.
\end{remark}

\begin{definition}\label{def:partial_integration} Assume that $M$ splits as $M = M'\oplus M''$, so that $D(M) = D(M') \times D(M'').$ We define a \textit{partial integration operator}
\[
\int_{D(M')}:\KCharVar{M}{S} \to \KCharVar{M''}{S} 
\]    
in the following way. Given a generator
\[
[X\times_SD(\underline{L}_S), \chi \mapsto \chi(l)]_{S\times_SD(\underline{M'\oplus M''}_S)},
\]
with $g: M'\oplus M''\to L$ and $l = g(m',m'')$ in the image of $g$, we define 
\[\int_{D(M')}[X\times_SD(\underline{L}_S), \chi \mapsto \chi(l)] =
\begin{cases} 0 & \text{if}\ m'\neq 0 \\
[X\times_SD(\underline{L}_S),\chi \mapsto \chi(l)]_{S\times_S D(\underline{M''}_S)} & \text{otherwise,}
\end{cases}
\]
where in the latter case the morphism $D(\underline{L}_S)\to D(\underline{M''}_S)$ comes from $g_{|M''}:M''\to L.$
If moreover we are given an subgroup $N\subset M'$, we may more generally define
\[
\int_{D(M'/N)}:\KCharVar{M}{S} \to \KCharVar{M''}{S} 
\]
given by
\[\int_{D(M'/N)}[X\times_SD(\underline{L}_S), \chi \mapsto \chi(l)] =
\begin{cases} 0 & \text{if}\ m'\not\in N \\
[X\times_SD(\underline{L}_S),\chi \mapsto \chi(l)]_{S\times_S D(\underline{M''}_S)} & \text{otherwise.}
\end{cases}
\]
\end{definition}
\begin{remark}\label{remark:composition_of_partial_integration} Let $M = M'\oplus M''$ with $N$ a subgroup of $M'$ as in \cref{def:partial_integration}. Then it is straightforward to check that the operator
\[
\int_{D(M/N)}: \KCharVar{M}{S}\to \KVar{S}
\]
is equal to the composition $\int_{D(M'')}\circ \int_{D(M'/N)}.$
\end{remark}

\begin{definition} 
We define the localisation 
 \[
\CharM{M}{S} 
=
   \KCharVar{M}{S}
  \left [
  \mathbf L_S^{-1} 
  \right ] . 
 \]
\end{definition}


\subsection{A local integration operator} \label{sect:local_integration_operator} Let $(M_\mathsf{S})_{\mathsf{S}}$ be a directed system of finite type abstract commutative groups, i.e., for every $\mathsf{S}\subset \mathsf{S}'$ there is a group morphism 
\[
i_{\mathsf{S},\mathsf{S'}}: M_{\mathsf{S}} \to M_{\mathsf{S}'},
\]
which we additionally assume to be injective. We also fix directed systems $(M'_\mathsf{S})_{\mathsf{S}}$ and $(M''_\mathsf{S})_{\mathsf{S}}$, assuming that for every $\mathsf{S}$
\[
M_{\mathsf{S}} = M'_{\mathsf{S}} \oplus M''_{\mathsf{S}}
\]
and that $M''_{\mathsf{S}} = M''$ is constant in $\mathsf{S}$. Finally, we fix a family $(N_{\mathsf{S}})_{\mathsf{S}}$ of groups such that for every $\mathsf{S}$, $N_\mathsf{S}$ is a subgroup of $M'_\mathsf{S}$, and such that for every $\mathsf{S}\subset \mathsf{S}'$, $i_{\mathsf{S},\mathsf{S'}}(N_\mathsf{S})\subset N_{\mathsf{S'}}$, so that the restrictions of the $i_{\mathsf{S},\mathsf{S'}}$ make the family $(N_\mathsf{S})_{\mathsf{S}}$ into a directed system. We define $M = \varinjlim_{\mathsf{S}} M_{\mathsf{S}}$ and $M' = \varinjlim_{\mathsf{S}} M'_{\mathsf{S}}$.

We wish to define an operator
\[
\int_{D(M'/N)}: \KCharVar{M}{X} \to \KCharVar{M''}{X}.
\]
For this, let 
$[Y\times_X D(\underline{L}_{X}), \chi\mapsto \chi(l)]$ be a generator, with a group morphism $g:M\to L$ and $l\in L$ in the image of $g$, i.e. $l = g(m)$ where $m = m'+ m''\in M = M'\oplus M''.$ We pick $\mathsf{S}$ such that $m\in M_\mathsf{S}$ and define

\[
\int_{D(M'/N)} [Y\times_X D(\underline{L}_{X}), \chi\mapsto \chi(l)]:= \int_{D(M'_\mathsf{S}/N_\mathsf{S})}[Y\times_X D(\underline{L}_{X}), \chi\mapsto \chi(l)] \in \KCharVar{M''}{X}
\]
in the sense of \cref{def:partial_integration}. We claim that this definition does not depend on the choice of such an $\mathsf{S}$. Take $\mathsf{S}\subset \mathsf{S'}$. We have
\[
\int_{D(M'_\mathsf{S}/N_\mathsf{S})}[Y\times_X D(\underline{L}_{X}), \chi\mapsto \chi(l)] = 
\begin{cases}
    0 & \text{ if } m' \not\in N_{\mathsf{S}} \\
    [Y\times_X D(\underline{L}_{X}), \chi\mapsto \chi(l)]_{D(M'')} & \text{ otherwise}.
\end{cases}
\]
in $\KCharVar{M''}{X}$. Now, viewing $m'$ as an element of $M'_{\mathsf{S'}}$, we have  $m'\in N_{\mathsf{S}'}$ if and only if $m'\in N_{\mathsf{S}'}\cap M'_{\mathsf{S}} = N_{\mathsf{S}}$. This shows that 
\[\int_{D(M'_\mathsf{S}/N_\mathsf{S})} [Y\times_X D(\underline{L}_{X}), \chi\mapsto \chi(l)] = \int_{D(M'_{\mathsf{S}'}/N_{\mathsf{S}'})} [Y\times_X D(\underline{L}_{X}), \chi\mapsto \chi(l)], 
\]
proving the claim. 

\subsection{Dimensional filtration and convergence of series}

\begin{definition}\label{def:dim}  
For any integer $i$, we define $\Fil^i \mathscr M_S$ as the subgroup of $\mathscr M_S$ generated by classes of the form $[X]_S \mathbf{L}_{S}^{-m}$ where $X$ is a variety over $S$ and $m$ is an integer such that $\dim_{S}X -m \leq -i$. This defines a decreasing filtration $(\Fil^i \mathscr M_S)_{i\in \ZZ}$ on $\mathscr M_S$, endowing $\mathscr M_S$ with the structure of a filtered ring. 

For any class $\mathfrak{a}\in \mathscr M_S$, we define
\[\dim_{S}(\mathfrak{a}) = \inf\{i\in \ZZ, \mathfrak{a}\in \Fil^i \mathscr M_S\}\]
We define the completed Grothendieck ring
\[
\widehat{\mathscr M_S}
= 
\lim_{\longleftarrow} \mathscr M_S/ \Fil^{i}\mathscr M_S.
\]
The $\dim_S$ function extends naturally to $\widehat{\mathscr M_S}$. 
\end{definition}

\begin{definition} Let $F(T) = \sum_{i\geq 0}A_{i} T^i\in \widehat{\mathscr M_S}[[T]]$ be a power series with coefficients in $\widehat{\mathscr M_S}$. The radius of convergence of $F$ is defined to be
$$\sigma_F = \limsup_{i\geq 1}\frac{\dim_S(A_i)}{i}.$$
We say that $F$ converges for $|T|< \LL^{-r}$ if $r\geq \sigma_F$. 
\end{definition}
If $F$ converges for $|T|< \LL^{-r}$, then for every element $\mathfrak{a}\in \widehat{\mathscr M_S}$ such that $\dim_S(\mathfrak{a}) < -r$, the value $F(\mathfrak{a})$ exists as an element of $\widehat{\mathscr M_S}$. 

 We now turn to convergence of multivariate series.  The place ourselves in the following general setting: let $N$ be a finitely generated free $\ZZ$-module, and consider
 \[\langle \cdot , \cdot \rangle : N^{\vee}\times N \to \mathbf{Z} 
 \]
 the natural pairing with the dual $N^\vee = \Hom(N,\mathbf{Z}).$ For a rational polyhedral cone $\Lambda\subset N_{\mathbf{R}}$, the intersection $\Lambda\cap N$ is a finitely generated monoid to which we may associate the corresponding completed monoid algebra $\widehat{\MMM_S}[[\Lambda\cap N]]$ with coefficients in $\widehat{\MMM_S}$. The elements of this algebra are viewed as power series with coefficients in $\widehat{\MMM_S}$ in a set of indeterminates $\mathbf{T} = (T_i)_{i\in I}$ indexed by a basis of $\Lambda\cap N$. The relative interior of a cone $\Lambda$ is denoted by $\overset{\circ}{\Lambda}.$
\begin{definition} Let $\Lambda\subset N_{\mathbf{R}}$ be a cone. Let $F(\mathbf{T}) = \sum_{\mathbf{m}\in \Lambda\cap N}A_{\mathbf{m}} \mathbf{T}^{\mathbf{m}} \in \widehat{\mathscr M_S}[[\Lambda\cap N]]$ be a power series with coefficients in $\widehat{\mathscr M_S}$. Let $\mathbf{\lambda}$ be an element of  $N^{\vee}\cap \overset{\circ}{\Lambda^{\vee}}$. The radius of convergence of $F$ in the direction $\lambda$ is defined to be
\[
\sigma_{F,\lambda} = \limsup_{\mathbf{m}\in \Lambda}\frac{\dim_S(A_\mathbf{m})}{\langle  \mathbf{\lambda}, \mathbf{m}\rangle}.
\]
We say that $F$ converges for $|\mathbf{T}|< \LL^{-s\lambda}$ (or, equivalently, for $|T_i| < \LL^{-s\lambda_i}$ for all $i\in I$) if $s\geq \sigma_{F,\lambda}$. 

If $F$ converges for $|\mathbf{T}|< \LL^{-s\lambda}$, then for every tuple of elements $\mathfrak{a} = (\mathfrak{a}_i)_{i\in I}\in \widehat{\MMM_{S}}^I$ such that $\dim_S(\mathfrak{a})< -s\lambda_i$, the value $F(\mathfrak{a})$ exists as an element of $\widehat{\MMM_S}.$
\end{definition}

\subsection{Motivic stabilisation}
The following is a motivic variant of \cite[Lemme 12]{bourqui2011quadriquesintrinseques}.
\begin{lemma}
\label{lemma-from-convergent-series-to-stabilisation-of-coeffs}
Let 
$( \mathfrak a_\dd )_{\dd \in \NN^r}$
be a family of classes in $\MMM_S$
and 
    \[
F ( \TT )
= 
\sum_{\dd \in \NN^r}
\mathfrak a_\dd 
\TT^\dd .
    \]
Assume that $F ( \TT )$ converges absolutely for $| T_i | < \LL^{-\rho_i + \varepsilon}$
and 
consider the power series 
\[
\frac{F ( \TT ) }{
\prod_{i=1}^r ( 1 - \LL^{\rho_i} T_i )
}
=
\sum_{\dd \in \NN^r} \mathfrak b_d \TT^\dd .
\]
Then for all $ 0 < \eta < \varepsilon$
\[
\mathfrak b_\dd \LL^{ - \langle \bm \rho , \dd \rangle } - 
F ( \LL^{- \bm \rho } ) 
\]
has virtual dimension bounded above by 
\[
- \eta \min ( \rho_i d_i ) .
\]
\end{lemma}

\begin{proof}
It goes as for \cite[Lemme 12]{bourqui2011quadriquesintrinseques}. 
The difference $\mathfrak b_\dd - 
F ( \LL^{- \bm \rho } ) \LL^{ \langle \bm \rho , \dd \rangle }$ 
is the coefficient of degree $\bm d \in \mathbf N^r$
of $ \frac{
        F ( \TT ) - F ( \LL^{- \bm \rho} )
    }
    {
       ( 1 - \LL^{\rho_1} T_1 )
       \cdots 
       ( 1 - \LL^{\rho_r} T_r )
    }$.
We can write 
    \[
    \frac{
        F ( \TT ) - F ( \LL^{- \bm \rho} )
    }
    {
        \prod_{i=1}^r ( 1 - \LL^{\rho_i} T_i )
    }
=
    \sum_{j=1}^r 
    \underbrace{\frac{
        F ( T_1 , ... , T_j , \LL^{-\rho_{j+1} } , ... , \LL^{-\rho_r} ) - F (  T_1 , ... , T_{j-1} , \LL^{-\rho_j } , ... , \LL^{-\rho_r}  )
    }
    {
        \prod_{i=1}^r ( 1 - \LL^{\rho_i} T_i )
    }
    }_{=: G_j ( \mathbf T ) }
    \]
where the degree $\bm d \in \mathbf N^r$ coefficient of $ G_j ( \mathbf T ) $ is
\[
\mathfrak b_{j,\bm d} 
= 
-
\sum_{\substack{
    0 \leqslant \delta_1 \leqslant d_1 
    \\
    \dots 
    \\
    0 \leqslant \delta_{j-1} \leqslant d_{j-1} 
    }
    }
\sum_{ \delta \in \mathbf N}
\sum_{
    \substack{
    \delta_{j+1} \in \mathbf N 
    \\
    \dots 
    \\
    \delta_r \in \mathbf N 
    }
    }
\mathfrak a_{
    \delta_1 , ... , \delta_{j-1} , d_j + \delta + 1 , \delta_{j+1} , ... \delta_r} 
\mathbf L^{
    - ( \delta + 1 )\rho_j + \sum_{\ell \neq j} ( d_\ell - \delta_\ell )\rho_\ell 
} . 
\]
By assumption,
we have 
\[
\dim ( \mathfrak a_{
    \delta_1 , ... , \delta_{j-1} , d_j + \delta + 1 , \delta_{j+1} , ... \delta_r}  ) 
< \langle 
( \delta_1 , ... , \delta_{j-1} , d_j + \delta + 1 , \delta_{j+1} , ... \delta_r ) 
 , ( 1 - \eta ) \bm \rho \rangle 
\]
for $0 < \eta < \varepsilon$ and $\bm d$ large enough,
hence the claim. 
\end{proof}


\section{Symmetric products and motivic Euler products}

\label{section:motivic-Euler-products}

The aim of this section is to recall the definitions of symmetric products and motivic Euler products from \cite{bilu2023MAMS}, and explain how one can interpret them as functions on the scheme of effective divisors of the curve $\mathscr C$.


\subsection{Symmetric products}
Following for example \cite[\S 6.1]{bilu-howe2021motivic-statistics},
we start with a few basic definitions concerning symmetric products of varieties.
These products are 
the first building blocks leading to the notion of motivic Euler product. 

\subsubsection*{Symmetric powers}
Let $S$ be a scheme, and let $X$ be a quasi-projective variety over $S$. If $n$ is a non-negative integer,
the $n$-th symmetric power of the variety $X  \to S$
is the quotient 
\[
\Sym_{/S}^n ( X ) 
= 
\underbrace{X \times_S ... \times_S X}_{n \text{ times}}
/ \mathfrak S_n.
\]

If $Y \to X$ is a variety over $X$, quasi-projective over~$S$, the symmetric power $\Sym^{n}_{/S}(Y)$ naturally comes with a morphism
\[
\Sym^n_{/S} ( Y )
\longrightarrow  
\Sym^n_{/S} ( X ).
\]

\subsubsection*{The incidence algebra}
We now explain how to define symmetric products of classes in $\KVar{X}$.
The product 
\[
\KVar{\Sym^\bullet_{/S} ( X ) }
:=
\prod_{n \geq 0}
\KVar{\Sym^n_{/S} ( X )}
\]
has the structure of a $\KVar{S}$-algebra, where addition is termwise, and the product is given by the Cauchy product 
\[
\left 
( 
(ab)_n
\right )_{n\in \NN}
=
\left 
( 
\sum_{i=0}^n a_i b_{n-i}
\right )_{n\in \NN}
\]
where for every $n\in \NN$ and $i \in \{ 0 , ... , n \}$,
the product comes from
the exterior product
\[
 a_i \boxtimes b_{n-i} \in \KVar{\Sym^i_{/S} ( X ) 
\times_\kb
\Sym^{n-i}_{/S} ( X ) }
\]
composed with the pushforward via the natural arrows 
\[
\Sym^i_{/S} ( X ) 
\times_\kb
\Sym^{n-i}_{/S} ( X ) 
\longrightarrow 
\Sym^n_{/S} ( X ) . 
\]
We denote by $\KVar{\Sym^{\bullet}_{/S}(X)}^1\subset \KVar{\Sym^{\bullet}_{/S}(X)}$ the subset of families $(a_n)_{n\geq 0}$ with ``constant term'' $a_0$ equal to 1. It is a subgroup of the multiplicative units in $\KVar{\Sym^{\bullet}_{/S}(X)}.$

Then there is a unique group morphism (see \cite[Lemma 3.5.1.2]{bilu2023MAMS})
\[
\KVar{X}\to \KVar{\Sym^\bullet_{/S} ( X )}^{1} 
\]
which to the class $[Y\to X]$ of an $X$-variety $Y$ associates the family 
\[
([\Sym^{i}_{/S}(Y)\to \Sym^i_{/S}(X)])_{i\geq 1}
\]
of classes of its symmetric powers; for every $a\in \KVar{X}$ we then denote by $\Sym^i_{X/S}(a)$ the $i$-th coordinate of its image via this morphism.

\subsubsection*{Generalized symmetric products} 

Let $I$ be a set and let $\pi\in \mathbf{N}^{(I)}$ be a generalized partition. Let $X$ be a quasi-projective variety over $S$. As before, all products are taken relatively to $S$. We denote

The product of permutation groups $\prod_{i\in I} \mathfrak{S}_{n_i}$ acts on the product $\prod_{i\in I} X^{n_i}$, inducing a quotient morphism
\begin{equation}\label{morphism:quotient_by_permutations}
\prod_{i\in I} X^{n_i} \to \prod_{i\in I} \Sym_{/S}^{n_i}X
\end{equation}
We denote by 
\[
\left( 
\prod_{i\in I} X^{n_i}
\right)_{*,X/S}
\]
the complement of the big diagonal in $\prod_{i\in I} X^{n_i}$, that is, the open subset consisting of points with coordinates all distinct. 

\begin{definition}
    The symmetric product $\Sym^{\pi}_{/S}(X)_*$ of $X$ relatively to $S$, corresponding to the partition $\pi$, is defined as 
    \[
    \left( \prod_{i\in I} \Sym^{n_i} X \right)_{*,X/S},
    \]
    that is, the image of the open subset $\left( 
\prod_{i\in I} X^{n_i}
\right)_{*,X/S}$ via the morphism (\ref{morphism:quotient_by_permutations}). 
\end{definition}

\begin{definition} Let now $\mathscr{X} = (X_i)_{i\in I}$ be a family of $X$-varieties, quasi-projective over $S$. We denote by $\Sym^{\pi}_{X/S} (\mathscr{X})_*$ the variety fitting into the top left corner of the following diagram:
\[
\begin{tikzcd}
\Sym^{\pi}_{X/S}(\mathscr{X})_* \ar[r] \ar[d]& \prod_{i\in I} \Sym^{n_i}(X_i)\ar[d] \\
\Sym^{\pi}_{/S}(X)_* \ar[r] & \prod_{i\in I} \Sym_{/S}^{n_i}X
\end{tikzcd}
\]
where the bottom horizontal arrow is the natural inclusion, and the rightmost vertical arrow is induced by the given morphisms $X_i\to X$. 
    
\end{definition}

\subsubsection*{Generalized symmetric product classes} We now extend the above construction to general classes in $\KVar{X}.$
Let $\mathscr A = ( a_i )_{i\in I}$
be a family of classes in $\KVar{X}$
and $\pi = ( n_i )_{i\in I} \in \mathbf N^{(I)}$.  We set 
\[
\Sym_{X/S}^\pi ( \mathscr A )_*
= 
\left(\boxtimes_{i\in I} \Sym_{X/S}^{n_i} ( a_i ) \right)_{*,X/S}
\]
in $\KVar{\Sym_{/S}^\pi ( X )_* }$, where $\boxtimes_{i\in I}$ is the exterior product and the subscript ``$*,X/S$'' indicates application of the restriction morphism
\[
\KVar{\prod_{i\in I} \Sym^{n_i}_{/S}(X)} \to \KVar{\Sym^{\pi}_{/S}(X)_*}
\]
induced by the inclusion $\Sym^{\pi}_{/S}(X)_*\to \prod_{i\in I} \Sym^{n_i}_{/S}(X).$

\begin{remark}[Points of symmetric products] We will repeatedly use the following description of points of symmetric products, explained in detail in \cite[Section 3.2.2.]{bilu2023MAMS}. 

Let $I$ be a set, and $\pi = (n_i)_{i\in I}$ a partition. Following in \cite[Section 3.2.2.]{bilu2023MAMS}, a schematic point $E\in \Sym^{\pi}_{/S}(X)_*$ can be written as an effective zero-cycle 
\[ 
\sum_{i\in I} i (v_{i,1} + \ldots + v_{i,n_i})
\]
where the $v_{i,j}$ are distinct geometric points of $X$ (lying in the same fibre above $S$). More generally, given a family $\mathscr{A} = (A_i)_{i\in I}$ of varieties over $X$, a point of $\Sym^{\pi}(\mathscr{A})$ lying above $E$ is written in the form
\[
\sum_{i\in I} i (a_{i,1} + \ldots + a_{i,n_i})
\]
where for every $i$, the $a_{i,j}$ are geometric points of $A_i$ such that $a_{i,j}$ maps to $v_{i,j}$. 
\end{remark}

\subsection{Mixed symmetric products}
\label{subsection:mixed-sym-prod}
We also need the more general notion of mixed symmetric products, introduced in \cite[Section 3.3.2]{bilu2023MAMS}. As before, let $X$ be a variety over $S$. All products of varieties in this sections are taken relatively to $S$. 

Let $p\geq 1$ be an integer, and let $I$ be a set. For every $j\in \{1,\ldots,p\}$, we fix a generalised partition ${\pi_{j} = (n_{i,j})_{i\in I}}$. We may define the variety $\Sym_{/S}^{\pi_1,\ldots,\pi_p}(X)_*$ to be the image of the complement of the big diagonal in
\[
\prod_{i\in I }\Sym_{/S}^{n_{i,1}}(X)\times \ldots \times\Sym_{/S}^{n_{i,p}}(X) . 
\]

More generally, let $\mathscr Y_1  = (Y_{i,1})_{i\in I}, ... , \mathscr Y_p = (Y_{i,p})_{i\in I}$ be 
families of varieties over $X$, all quasi-projective over $S$. 

For every $p$-tuple 
$ 
( \pi_1 , ... , \pi_p ) 
= ( ( n_{i,1} )_{i\in I} , ... , ( n_{i,p} )_{i\in I} ) 
 $ of generalised partitions of $I$,  
we consider the product
\[
\prod_{i\in I} Y_{i,1}^{n_{i,1}}\times \ldots \times Y_{i,p}^{n_{i,p}}  \to \prod_{i\in I }\Sym_{/S}^{n_{i,1}}(X)\times \ldots \times\Sym_{/S}^{n_{i,p}}(X).
\]
and restrict to $\Sym_{/S}^{\pi_1,\ldots,\pi_p}(X)_*$, to get the \emph{mixed symmetric product} 
\[
\Sym_{X/S}^{\pi_1 , ... , \pi_p} ( \mathscr Y_1 , ... , \mathscr Y_p ) _* \to \Sym^{\pi_1,\ldots,\pi_p}_{/S}(X)_*
\]
of the families $\mathscr{Y}_1,\ldots \mathscr{Y}_p$. 

Now let, for every $j \in \{1,\ldots,p\}$, $\mathscr{A}_j = (a_{i,j})_{i\in I}$ be a family of classes in $\KVar{X}$ indexed by the set~$I_j$.
The class
\[
\Sym_{X/S}^{\pi_1,\ldots,\pi_p}(\mathscr{A}_1,\ldots,\mathscr{A}_p)_*\in \KVar{\Sym^{\pi_1,\ldots,\pi_p}_{/S}(X)_*}
\]
is defined as the image of the class
\[
\prod_{i\in I} \Sym_{X/S}^{n_{i,1}}(a_{i,1})\times \ldots \times \Sym_{X/S}^{n_{i,p}}(a_{i,p})
\]
via the restriction morphism
\[
\KVar{\prod_{i\in I}\Sym_{/S}^{n_{i,1}}(X)\times \ldots \times\Sym_{/S}^{n_{i,p}}(X)}\to \KVar{\Sym^{\pi_1,\ldots,\pi_p}_{/S}(X)_*}.
\]

\begin{remark}\label{remark:mixed-sym-prod-with-extra-structure}
We will be in a situation where there is the following extra structure: the set $I$ is partitioned into subsets:
\[ 
I = \bigsqcup_{j=1}^pI_j
\]
and in fact every family $\mathscr{A}_j$ is supported above $I_j$. Then each partition $\pi_j$ may in fact be viewed as an element of $\NN^{(I_j)}.$
\end{remark}

\subsection{Motivic Euler products}\label{subsection:motivic_euler_products}

Let us now recall the definition 
of motivic Euler products introduced in \cite[Chap.~3]{bilu2023MAMS}. We will consider the following set-up, which will fit our purposes. 

Let $I$ be a set which is of the form $I_0\setminus \{0\}$ where $I_0$ is a commutative monoid. 
Given $\iota_1,\iota_2\in I$, the natural morphism
\[
\Sym^{\iota_1}_{/S}(X) \times_S \Sym^{\iota_2}_{/S}(X)\to \Sym^{\iota_1 +\iota_2}_{/S}(X)
\]
induces, by composition with the exterior product, a morphism
\begin{equation}\label{multiplication_for_incidence_alg}
\KVar{\Sym^{\iota_1}_{/S}(X)}\times \KVar{\Sym^{\iota_2}_{/S}(X)}\to \KVar{\Sym^{\iota_1 + \iota_2}_{/S}(X)}.
\end{equation}
We define 
\[
\KVar{\Sym^{\bullet}_{/S}X} := \prod_{\iota\in I_0}\KVar{\Sym^{\iota}_{/S}(X)}
\]
the $\KVar{S}$-algebra with addition given by coordinatewise addition, and multiplication defined using the morphisms (\ref{multiplication_for_incidence_alg}). We also define $\KVar{\Sym^{\bullet}_{/S}X}^1\subset \KVar{\Sym^{\bullet}_{/S}X}$ the subset of families $(a_{\iota})_{\iota\in I_0}$ with $a_0 = 1$. It is a subgroup of the group of multiplicative units of $\KVar{\Sym^{\bullet}_{/S}X}$. An element of $\KVar{\Sym^{\bullet}_{/S}X}^1$ may be viewed as a power series 
\[ 
\sum_{\iota\in I}A_{\iota}T_{\iota}
\]
with $A_0 = 1$, where each $A_\iota$ is viewed as an element of $\KVar{\Sym^{\iota}_{/S}(X)}$ and where the indeterminates $T_i$ satisfy relations $T_{\iota_1}T_{\iota_2} = T_{\iota_1 + \iota_2}$

Our motivic Euler products will be defined as such power series.

\begin{definition}
    \label{definition-motivic-euler-product}
    Let    $ \mathscr A = ( a_i )_{i\in I} $ 
    be a family of classes above $X$.
    The formal motivic Euler product 
    \[
    \prod_{x\in X/S}
    \left ( 
        1 + \sum_{i\in I} a_{i,x} T_i 
    \right ) 
    \]
    associated to $\mathscr A $
    is by definition the formal series
    \[
    \sum_{i\in I} \left( \sum_{\pi \vdash i}
        \Sym_{X/S}^\pi ( \mathscr A ) _* \right) T_i \in \KVar{\Sym^{\bullet}_{/S}(X)} . 
    \]
 \end{definition}
 
 We now define a refined setting. Assume additionally that $I = \bigsqcup_{j=1}^pI_j$ for some commutative semigroups $I_j$. We pick sets of indeterminates $(T_{i,j})_{i\in I_j,1\leq j \leq p}$ such that $T_{i_1,j}T_{i_2,j} = T_{i_1 + i_2,j}$ for all $j$ and all $i_1,i_2\in I_j$. 
 
 We denote for every $j$, $\bar{I}_{j} = I_j\cup\{0\}$ and define
 \[
 \KVar{\Sym^{\bullet}_{/S}(X)} = \prod_{(i_1,\ldots,i_p)\in \bar{I}_1\times \ldots \times\bar{I}_{j}}
 \KVar{\Sym^{i_1,\ldots,i_p}_{/S}(X)}
 \]
 again with coordinatewise addition, and appropriate multiplication operations. 
 
 \begin{definition}
    \label{definition-mixed-motivic-euler-product}
    For every $j\in \{1,\ldots,p\}$ let $ \mathscr A_j = ( a_{i,j} )_{i\in I} $ 
    be a family of classes in $\KVar{X}$
    The formal motivic Euler product 
    \[
     \prod_{x\in X/S}
    \left ( 
        1 + 
        \sum_{j=1}^p 
        \sum_{i\in I_j}
           a_{i,j,x}  T_{i,j}
    \right ) 
    \]
    is by definition the formal series
    \[
    \sum_{(i_1,\ldots,i_p)\in \bar{I}_1\times \ldots \times\bar{I}_{j}}
    \left(
    \sum_{\pi_1\vdash i_1 , ... , \pi_p \vdash i_j}
    \Sym_{X/S}^{\pi_1 , ... , \pi_p} ( \mathscr A_1 , ... , \mathscr A_p ) _* 
    \right)
    T_{i_1,1}\ldots T_{i_p,p} . 
    \]
 \end{definition}
 
\begin{remark}
All of these definitions work analogously for localised Grothendieck rings. 
\end{remark}

 \subsection{Symmetric products of varieties with characters} \label{subsection:symproducts_characters}

In this section we extend the notion of motivic Euler product from \cite{bilu2023MAMS} to the Grothendieck ring of varieties with multiplicative characters.
In particular, this will allow us to make sense of products of the specific form
\[
\prod_{v \in \CCC} 
\left ( 1
+
\sum_{m \in \ZZ_{\neq 0}^n} X_{m,v}  \chi_v ( m ) \mathbf T^m 
\right ) 
\]
appearing later as certain Fourier transforms. 
The approach is completely analogous
to how one defines motivic Euler products for varieties with exponentials in \cite[Section 3.6.1]{bilu2023MAMS}. 

Let $I$ be a set. Let $S$ be a scheme, $X$ a variety over $S$, 
$( M_i )_{i\in I}$
a family of 
commutative groups,
and 
\[
( \mathscr{Y}, h ) 
= ( Y_i \times_X D ( \underline{ L_i }_X ) , h_i : D ( \underline{L_i}_{Y_i} ) \to \GG_m )
\]
a family of quasiprojective varieties with characters over $X$. 
Let $\pi = (n_i)_{i\in I}\in \NN^{(I)}$ be a generalised partition. 

We consider the product 
\[
\prod_{i\in I} ( Y_i \times_X D ( \underline{ L_i }_X ) )^{n_i}
\]
(taken relatively to the base scheme $S$), endowed with the map
\[
h^{\pi}: \prod_{i\in I} D ( \underline{L_i}_{Y_i} )^{n_i} \to \GG_{m}
\] 
above $\prod_{i\in I} Y_i^{n_i}$
given by 
\[
(\chi_{i,1},\ldots, \chi_{i,n_i})
\mapsto 
\prod_{i\in I}
\prod_{j=1}^{n_i} h_i ( \chi_{i,j} ) .
\]
The aforementioned product comes with a morphism 
\[
\prod_{i\in I} Y_i^{n_i} \times_X D ( \underline{ L_i }_X )^{n_i} \to \prod_{i\in I} X^{n_i}
\] 
(all products are taken relatively to $S$) and we denote by 
\[
\left( \prod_{i\in I} Y_i^{n_i} \times_X D ( \underline{ L_i }_X )^{n_i} \right)_{*,X/S}
\]
the inverse image of the complement of the big diagonal in $\prod_{i\in I}X^{n_i}$. The map $h^{\pi}$ restricts to $\left( \prod_{i\in I} Y_i^{n_i} \times_X D ( \underline{ L_i }_X )^{n_i} \right)_{*,X/S}$, and is invariant modulo the permutation action of $\prod_{i\in I} \mathfrak{S}_{n_i}$, giving us a diagram
\[
\begin{tikzcd}
\Sym^{\pi} ( Y_i\times_X D( \underline{ L_i }_X) ) _{i\in I} \ar[r,"\Sym^{\pi}(h_i)_{i\in I}"] \ar[d]& \GG_m\\
\Sym^{\pi}_{X/S}(Y_i)_{i\in I}
\end{tikzcd}
\]
denoted $\Sym^{\pi}(\mathscr{Y},h)$.


\subsection{Higher order Grothendieck rings}

\begin{definition}[Grothendieck ring of varieties with characters of level $\pi$] \label{def:higher-order-grothendieck-ring} We define the Grothendieck group $\KCharVar{\pi,(M_i)_{i\in I}}{\Sym^{\pi}_{/S}(X)_*}$ as the free abelian group 
generated by isomorphism classes 
\[
[Y \times_{\Sym^{\pi}_{/S}(X)_*} \Sym^{\pi}_{X/S}((D(L_i))_{i\in I})_*, h ]_{\Sym^{\pi}_{/S}(X)_*\times_{\Sym^{\pi}_{/S}(X)_*}\Sym^{\pi}_{X/S}(D(\underline{M_i}_X))_{i\in I} }
\]
of diagrams
   \[ 
\begin{tikzcd}
    Y \times_{\Sym^{\pi}_{/S}(X)_*}\Sym^{\pi}_{X/S}((D(L_i))_{i\in I})_*\arrow[d]  \arrow[r,"h"] & \GG_m \\
    \Sym^{\pi}_{/S}(X)_*\times_{\Sym^{\pi}_{/S}(X)_*}\Sym^{\pi}_{X/S}D(\underline{M}_X) & & 
\end{tikzcd}
\]
such that $h$ is of the form $\Sym^{\pi}(h_i)_{i\in I}$ as in \cref{subsection:symproducts_characters}, modulo cut-and-paste relations on the first factor. When all the $M_i$ are equal to the same $M$, we simply write $\KCharVar{\pi,M}{\Sym^{\pi}_{/S}(X)_*}.$

\end{definition}

As for classical Grothendieck rings, it will be convenient to combine Grothendieck rings with characters into incidence algebras, which will be natural receptacles for the notion of motivic Euler product with characters. 

Let $I$ be a set which is of the form $I_0\setminus \{0\}$ where $I_0$ is a commutative monoid.

\begin{definition} 
    We denote by $\KCharVar{\iota,M}{\Sym^{\iota}_{/S}(X)}$ the free abelian group 
generated by isomorphism classes 
\[
[Y \times_{\Sym^{\iota}_{/S}(X)} \Sym^{\iota}_{X/S}(D(L_i))_{i\in I}, h ]_{ \Sym^{\iota}_{/S}(X)\times_{\Sym^{\iota}_{/S}(X)}\Sym^{\iota}_{X/S}D(\underline{M}_X) }
\]
of diagrams
   \[ 
\begin{tikzcd}
    Y \times_{\Sym^{\iota}_{/S}(X)}\Sym^{\iota}_{X/S}(D(L_i))_{i\in I}\arrow[d]  \arrow[r,"h"] & \GG_m \\
    \Sym^{\iota}_{/S}(X)\times_{\Sym^{\iota}_{/S}(X)}\Sym^{\iota}_{X/S}D(\underline{M}_X) & & 
\end{tikzcd}
\]
such that the restriction of $h$ to $\Sym^{\pi}_{X/S}((D(L_i))_{i\in I})_*$ for every $\pi\vdash \iota$ is as in \cref{def:higher-order-grothendieck-ring}, modulo cut-and-paste relations on the first factor.    

\end{definition}

For every partition $\pi$ of $\iota$, the inclusion $i_{\pi,\iota}:\Sym^{\pi}_{X/S}(\underline{M}_X) \to \Sym^{\iota}_{X/S}(\underline{M}_X)$ induces a restriction morphism
\[
\KCharVar{\iota,M}{\Sym^{\iota}_{/S}(X)} \to \KCharVar{\pi,M}{\Sym^{\pi}_{/S}(X)_*}.
\]

\begin{lemma}
These restriction morphisms induce an isomorphism
\[
\KCharVar{\iota,M}{\Sym^{\iota}_{/S}(X)}\simeq \prod_{\pi\vdash\iota}\KCharVar{\pi,M}{\Sym^{\pi}_{/S}(X)_*}.
\]
\end{lemma}
\begin{proof}
Let $([Y_{\pi}\times_{\Sym^{\pi}_{X/S}(\underline{M}_X)}\Sym^{\pi}_{X/S}(D(L_{\pi,i})_{i\in I})_*,h^{\pi}])_{\pi\vdash \iota}$ be a generator on the right-hand side (beware that different partitions $\pi$ here a priori correspond to different families $(L_{\pi,i})_{i\in I}$ of groups with morphisms $M\to L_{\pi,i}$). We define $Y_{\pi,\iota}$ to be the variety $Y_{\pi}$ viewed above $\Sym^{\iota}_{/S}(X)$ via composition with the inclusion $\Sym^{\pi}_{/S}(X)_* \to \Sym^{\iota}_{/S}(X).$ Then we send our tuple to
\[
\sum_{\pi\vdash \iota}([Y_{\pi,\iota}\times_{\Sym^{\iota}_{X/S}(\underline{M}_X)}\Sym^{\iota}_{X/S}(D(L_{\pi,i})_{i\in I}),h^{\pi,\iota}])_{\pi\vdash \iota}
\]
where $h^{\pi,\iota}$ is defined to be $h^{\pi}$ extended by 1 outside of $\Sym^{\pi}_{X/S}(D(L_{\pi,i})_{i\in I}).$
\end{proof}

Given $\iota_1,\iota_2\in I$, the natural morphisms
\[
\Sym^{\iota_1}_{/S}(X) \times_S \Sym^{\iota_2}_{/S}(X)\to \Sym^{\iota_1 +\iota_2}_{/S}(X)
\]
and 
\[
\Sym^{\iota_1}_{/S}D(\underline{M}_X) \times_S \Sym^{\iota_2}_{/S}D(\underline{M}_X) \to \Sym^{\iota_1 + \iota_2}_{/S}D(\underline{M}_X)
\]
induce maps
\[
\KCharVar{\iota_1,M}{\Sym^{\iota_1}_{/S}(X)}
\times 
\KCharVar{\iota_2,M}{\Sym^{\iota_2}_{/S}(X)} 
\to 
\KCharVar{\iota_1 + \iota_2,M}{\Sym^{\iota_1 + \iota_2}_{/S}(X)}.
\]
These in turn allow to define a multiplication map on 
\[
\prod_{\iota}\KCharVar{\iota,M}{\Sym^{\iota}_{/S}(X)}.
\]

\begin{definition}
    \label{def:incidence-algebra-with-characters}
    We denote by 
\[
\KCharVar{\Sym^{\bullet}(M)}{\Sym_{/S}^{\bullet}(X)} =\prod_{\iota}\KCharVar{\iota,M}{\Sym^{\iota}_{/S}(X)}.
\]
the $\KVar{S}$-algebra with addition given by coefficientwise addition, and with multiplication given by the above multiplication map. 
\end{definition}

We can also do the same construction with localised Grothendieck rings, giving us the following definition.
\begin{definition} 
\label{def:incidence-algebra-with-characters-localised}
     We denote by 
\[
\CharM{\Sym^{\bullet}(M)}{\Sym_{/S}^{\bullet}(X)} =\prod_{\iota}\CharM{\iota,M}{\Sym^{\iota}_{/S}(X)}.
\]
the $\MMM_S$-algebra with addition given by coefficientwise addition, and with multiplication given by the analogue for localised Grothendieck rings of the above multiplication map.  
\end{definition}

Now that we have access to incidence algebras also in the setting of Grothendieck rings of varieties with characters, we may define, in the same way as for classical Grothendieck rings, (mixed) symmetric products of classes in the Grothendieck ring of varieties with characters. 

 \subsection{Motivic Euler products and characters} 

The definition of motivic Euler products in the setting of the Grothendieck rings of varieties with characters then works exactly as in \ref{subsection:motivic_euler_products}. Again, we give ourselves a set $I = I_0\setminus \{0\}$ where $I_0$ is a commutative monoid and indeterminates $(T_i)_{i\in I}$ satisfying $T_{i_1}T_{i_2} = T_{i_1 + i_2}.$
\begin{definition} Let $S$ be a scheme, $X$ a quasi-projective variety over $S$, and $M$ an abstract commutative group. Let $\mathscr{A} = (a_i)_{i\in I}$ be a family of classes in $\KCharVar{M}{X}$.
The motivic Euler product 
\[
\prod_{x \in X / S}
\left ( 
    1 + \sum_{i\in I}a_{i,x} T_i 
\right ) 
\]
is a notation for the series
\[
\sum_{i\in I}
\left(\sum_{\pi\vdash i}\Sym^{\pi}_{X/S}(\mathscr{A})_*\right) T_i \in \KCharVar{\Sym^{\bullet}(M)}{\Sym^{\bullet}_{/S}(X)}.
\]

\end{definition}

This notion of motivic Euler product with characters in particular allows us to define a notion of \textit{motivic $L$-function}.

\begin{definition}[Motivic $L$-function] \label{def:motivic-L-function}
The (universal) motivic $L$-function is the formal series
\[
L (\cdot, T )
=
\prod_{v\in \CCC} \left ( 
\sum_{n\geqslant 0} \ev_n T^n 
\right )
\in 
\KCharVar{\Sym^\bullet \ZZ}{\Sym^\bullet \CCC} 
\]
which should be understood as the motivic function 
sending $\chi$
to the formal series
\[
\prod_{v\in \CCC} ( 1 - \chi_v T )^{-1} . 
\]
\end{definition}


\subsection{Symmetric products and Hilbert 90}

We generalise \cite[Proposition 3.4.0.1]{bilu2023MAMS} to algebraic groups satisfying Hilbert 90. We say that an algebraic group $G$ satisfies Hilbert 90 if a $G$-torsor which is locally trivial for the étale topology is also Zariski-locally trivial. By \cite[Theorem 11.4]{milne}, the multiplicative group $\GG_m$ (and, thus, more generally, split tori) satisfies Hilbert 90. 

\begin{proposition}
    Let $X$ be a variety over $S$,
    $\mathscr Y = ( Y_i )_{i\in I}$ a family of varieties above $X$ and 
    $ \mathscr G = ( G_i )_{i\in I}$ a family of algebraic groups over $S$ satisfying Hilbert 90.
    Let $\mathscr Y \times \mathscr G$ be the family $ ( Y_i \times G_i )_{i\in I}$, each element of this family being viewed above $X$ via the first projection.

    Then for any generalised partition $\pi = ( n_i )_{i\in I}$, 
    the symmetric product 
    \[
    \Sym_{X/S}^\pi ( \mathscr Y \times \mathscr G ) _*
    \]
    is endowed with the structure of a $\prod_i G_i^{n_i}$-torsor 
    (for the Zariski topology) 
    above $\Sym_{X/S}^\pi ( \mathscr Y ) _*$. 
\end{proposition}

\begin{proof}
    The argument used to prove \cite[Proposition 3.4.0.1]{bilu2023MAMS} works here almost verbatim. 
    Since 
    \[
    \left ( 
            \prod_{i\in I} Y_i^{n_i}
             \times \prod_{i\in I} G_i^{n_i}
        \right )_*
    \simeq 
    \left ( 
            \prod_{i\in I} Y_i^{n_i}
             \times \prod_{i\in I} G_i^{n_i}
        \right )
        \times_{\prod_{i\in I} X^{n_i} }
          \left ( 
            \prod_{i\in I} X_i^{n_i}
        \right )_*
    \simeq 
    \left ( 
            \prod_{i\in I} Y_i^{n_i}
        \right )_* \times \prod_{i\in I} G_i^{n_i}
    \]
    we have the Cartesian square:
    \[
    \begin{tikzcd}
        \left ( 
            \prod_{i\in I} Y_i^{n_i}
        \right )_* \times \prod_{i\in I} G_i^{n_i}
        \rar \dar["q'"] 
        & \left ( 
            \prod_{i\in I} Y_i^{n_i}
        \right )_* \dar["q"] \\
        \Sym_{X/S}^\pi ( \mathscr Y \times \mathscr G ) _*
        \rar["p"]
        &
         \Sym_{X/S}^\pi ( \mathscr Y  ) _*
    \end{tikzcd}
    \]
    where $q'$ is an étale-local trivialisation of $p$ and thus the claim follows by Hilbert 90 for $G_i$, $ i\in I$. 
\end{proof}

\begin{example}
    \label{example:familie-of-char-is-just-a-torsor}
    Let $S$ be the spectrum of a field $k$.
    Take $\mathscr Y$ to be the constant family $\mathscr C \to \mathscr C$
    and $(G_i)_{i\in I}$ to be the constant family $\mathbf G_{m,k}$. 
    Then the previous proposition tells us that
    \[
    \Sym_{\mathscr C / k}^\pi
    \left ( 
   \mathscr C \times_k   \mathbf G_{m} 
   \right )_*
   =
   \Sym_{\mathscr C / k}^\pi
    \left ( 
   \mathbf G_{m,\mathscr C} 
   \right )_*
    \]
    is a $\GG_m^{\sum_{i\in I}n_i}$-torsor above 
    $ \Sym_{/ k}^\pi
    \left ( 
   \mathscr C 
   \right )_*$,
   a Zariski trivialisation being canonically given by the Cartesian square 
   \[
   \begin{tikzcd}
        \left ( 
           \prod_i \mathscr C^{n_i}
        \right )_* \times  \mathbf G_m^{ \sum_{i\in I} n_i}
        \rar["{\pr_1}"] \dar["q'"] 
        & \left ( 
            \prod_i
            \mathscr C^{n_i}
        \right )_* \dar["q"] \\
       \Sym_{\mathscr C / k}^\pi
    \left ( 
   \mathbf G_{m,\mathscr C} 
   \right )_* 
        \rar["p"]
        &
         \Sym_{/ k}^\pi  \left (   \mathscr C  \right )_*.
    \end{tikzcd}
   \]
\end{example}


\subsection{Symmetric products of functions on constant group schemes}\label{subsection:sym-prod-constant-gp-schemes} 

In what follows, we will need to work with symmetric products in the following setting: let $X$ be a quasi-projective variety over $S$, let $(M_i)_{i\in I}$ be sets indexed by a set $I$ and consider, for every $i\in I$,
$Y_i = \underline{M_i}_X$ the corresponding constant $X$-scheme
(which, by definition, is the coproduct $\bigsqcup_{m\in M_i}X$, and thus is locally algebraic). In this section, products are taken relatively to~$S$. 

\subsubsection{Domain of definition}
Let $\pi\in \mathbf{N}^{(I)}$ be a partition. The product of symmetric groups $\prod_{i\in I}\mathfrak{S}_{n_i}$ acts on the product 
\[
\prod_{i\in I} (\underline{M_i}_X)^{n_i},
\] 
and we want to check that the corresponding quotient exists as a scheme. For this, by
\cite[Exposé V, Proposition 1.8]{SGA1} 
it suffices to check  that the  orbit of every point is contained in an affine open. Let $x\in \prod_{i\in I} (\underline{M_i}_X)^{n_i}$. Then the orbit of $x$ is contained in a finite union of products $X\times_S\ldots \times_S X$ of copies of $X$. Thus, for the sake of proving our statement we may assume that the $M_i$ are finite, and then it follows from quasi-projectivity of $X$ over $S$.  

Thus, the quotient $\prod_{i\in I}\Sym^{n_i}_{/S}(\underline{M_i}_X)$ exists as a (locally algebraic) scheme, and it is by construction endowed with a morphism to $\prod_{i\in I}\Sym^{n_i}_{/S}(X)$. By pulling back via the inclusion
\[
\Sym^{\pi}_{X/S}(X)_*\to \prod_{i\in I}\Sym^{n_i}_{/S}(X),
\] 
we also get a scheme $\Sym^{\pi}_{X/S}((\underline{M_i}_X)_{i\in I})_*$.

Assume that each $M_i$ is endowed with a commutative group structure. Then the group scheme structures $\underline{M_i}_X$ over $X$ canonically induce a group scheme structure
on 
\[
\prod_{i\in I} (\underline{M_i}_X)^{n_i},
\] which passes to the quotient and gives $\prod_{i\in I}\Sym^{n_i}_{X/S}(\underline{M_i}_X)$, 
respectively $\Sym^{\pi}_{X/S}((\underline{M}_X)_{i\in I})_*$, 
a group scheme structure over $\prod_{i\in I}\Sym^{\pi}_{/S}(X)$, resp. $\Sym^{\pi}_{/S}(X)_*$. The construction also extends to mixed symmetric products in the obvious way.

\subsubsection{Symmetric products of functions}
\label{subsection:symmetric-products-of-functions}
Let again $X \to S$ be a quasi-projective variety above a base scheme $S$.

Let $I$ and $( M_i )_{i\in I}$ be sets
and $( X_{m_i} )_{( m_i ) \in \prod_{i\in I} M_i}$ a family of varieties above $X$,
which we see as a motivic function 
on the constant group $\prod_{i\in I} \underline{M_i}_X$ above $X$. 
We consider the family 
\[
( \varphi_i )_{i\in I}
=
\left ( 
    \coprod_{m_i \in M_i} (  X_{m_i} \to X ) 
\right )_{i\in I}  . 
\]
For any generalised partion $\pi = ( n_i )_{i\in I}$,
we have canonical isomorphisms
\begin{align*}
    \prod_{i\in I}
        \left ( 
            \coprod_{m_i \in M_i} X_{m_i}
        \right )^{n_i}
  &   \simeq
    \prod_{i\in I}  
        \left ( 
            \coprod_{m_{i,1} \in M_i} X_{m_{i,1}}
        \right )
    \times_S
        \cdots 
    \times_S
        \left ( 
            \coprod_{m_{i,n_i} \in M_i} X_{m_{{i,n_i}}}
        \right ) \\
    & \simeq 
    \coprod_{ (m_{i,j} ) \in \prod_{i\in I} M_i^{n_i} } 
            \prod_{i\in I} 
            \left ( 
                X_{m_{i,1}} \times_S ... \times_S X_{m_{i,n_i}} 
            \right )
         . 
\end{align*}
For every $i\in I$,
the action of $\sigma_i \in \mathfrak S_{n_i}$ 
on $X^{n_i}$ permuting the factors 
induces an action on 
$ \left ( 
            \coprod_{m_{i,1} \in M_i} X_{m_{i,1}}
        \right )
    \times_S
        \cdots 
    \times_S
        \left ( 
            \coprod_{m_{i,n_i} \in M_i} X_{m_{{i,n_i}}}
        \right )
$
which does the following: for $( m_{i,j} )_{j=1}^{n_i} \in M_i^{n_i}$
identifies the component
$  X_{m_{i,1}} \times_S ... \times_S X_{m_{i,n_i}} $
with the component
$  X_{m_{i, \sigma (1)}} \times_S ... \times_S X_{m_{i,\sigma_i ( n_{i} )}}$
given by $(m_{i,\sigma ( j )} )_{j=1}^{n_i} \in M_i^{n_i}$
(the factors are the same, they are simply reordered)
relatively to $X^{n_i}$. 
Taking quotients with respect to this action leads to the commutative diagram
\[
\begin{tikzcd}
    \displaystyle \coprod_{ (m_{i,j} ) \in \prod_{i\in I} M_i^{n_i} } 
        \prod_{i\in I} X_{m_{i,1}} 
            \times_X ... \times_X X_{m_{i,n_i}} \rar \arrow[d,"{/\prod_{i\in I} \mathfrak S_{n_i}}"] &
        \displaystyle \prod_{i\in I} X^{n_i} \arrow[d,"{/\prod_{i\in I} \mathfrak S_{n_i}}"] \\
    \displaystyle \Sym^\pi_{X/S} ( ( \varphi_i )_{i\in I} )  \rar & \displaystyle \Sym^\pi_{/S} ( X ) 
\end{tikzcd}
\]
and one restricts to $\Sym^\pi_{/S} ( X )_* $ to get $\Sym^\pi_{X/S} ( ( \varphi_i )_{i\in I} )_*$. 

\begin{example}
Take $M_i = \mathbf Z$ and $X_i = X$ for every $i\in I$. Then 
\[
\left ( 
    \coprod_{m_i \in M_i} X_{m_i}
\right )_{i\in I} 
=
\left ( 
    \coprod_{\mathbf Z} X
\right )_{i\in I} 
=
\left ( 
    \underline{\mathbf Z}_X
\right )_{i\in I} 
\]
and the previous diagram becomes 
\[
\begin{tikzcd}
    \displaystyle \prod_{i\in I} \underbrace{\underline{\ZZ}_X \times_S ... \times_S \underline{\ZZ}_X}_{n_i \text{ times}} \arrow[r,"\simeq"] & 
    \displaystyle \coprod_{\prod_{i\in I} \ZZ^{n_i} } \prod_{i\in I} X^{n_i}
        \rar \arrow[d,"{/\prod_{i\in I} \mathfrak S_{n_i}}"] &
        \displaystyle \prod_{i\in I} X^{n_i} \arrow[d,"{/\prod_{i\in I} \mathfrak S_{n_i}}"] \\
    &  \Sym^\pi_{X/S} ( ( \underline{\mathbf Z}_X )_{i \in I} )  \rar & \displaystyle \Sym^\pi_{/S} ( X ) .
\end{tikzcd}
\]
The fibre above a point $D = \sum_{i\in I} D_i $
where $D_i = [p_{i,1} ] + ... + [p_{i,n_i}]$
is an unordered product 
\[
\oplus_{i\in I} \oplus_{i=1}^{n_i} \underline{\mathbf Z}_{\kappa ( p_i )} . 
\]
\end{example}


\section{Multiplicative motivic harmonic analysis}

\label{section:multiplicative-motivic-harmonic-analysis}

In this section we develop an elementary version of motivic harmonic analysis 
which in practice works well with abstract commutative groups and
their Cartier duals,
classically called \emph{diagonalisable algebraic groups}. 
The setting we adopt is sometimes more general than needed and we included alternative simpler proofs that work in the specific situation we consider in our application to the height zeta function.


\subsection{Fourier transforms and inversion formulas}
\label{subsect:general_local_fourier_transforms}
Let $S$ be a scheme and $M$ be an abstract commutative group. In this paragraph, by \emph{motivic function} we will simply mean an element of the relative Grothendieck ring $\CharM{M}{\underline{M}_S}$, which will be interpreted as a function on $\underline{M}_S$. 

Recall from Notation \ref{notation:evaluation_map} that we have the class 
\[
\ev 
= 
[\underline{M}_{S}\times_{S} D(\underline{M}_S), (m,\chi)\mapsto \chi(m)]
    \in \Char_M\MMM_{\underline{M}_S} . 
\] 

\begin{definition}\label{def:motivic-Fourier-transform} 
Let $M$ be an abstract commutative group and let $\psi\in \MMM_{\underline{M}_S}$ be a motivic function supported over finitely many points of~$\underline M_S$. The motivic Fourier transform of $\psi$ is defined as the class
\[
\four \psi 
= \psi\cdot \ev
\in \Char_M\MMM_{S}
\]
where the product is taken in $\Char_M\MMM_{\underline{M}_S}$ and then viewed in $\Char_M\MMM_{S}.$
We can see this as the motivic function associating 
\[
( \four \psi )
( \chi ) 
=
\sum_{m\in M} \psi ( m ) \chi( m )
\]
to any $\chi \in D(\underline{M}_S)$. 
\end{definition}

Assume $\psi$ is of the form $[X\to \underline{M}_S, 1].$ Then, explicitly, 
\[\four \psi  = [X\times_{\underline{M}_S}\underline{M}_S\times_SD(\underline{M}_S), (x,m,\chi)\mapsto \chi(m)]_{S\times_SD(\underline{M}_S)}\]

\begin{equation}\label{diagram-local-fourier-transform}
\begin{tikzcd}
\arrow[d]  X    &  \underline{M}_S \times_S D ( \underline{M}_S ) \drar["\pr_2"] \dlar["\pr_1"] \arrow[r,"\ev"] & \GG_{m,S}        \\
\underline{M}_S      \drar    &                                                              & D ( \underline{M}_S ) \dlar  \\ 
                &   S                                            &  \\   
\end{tikzcd}    
\end{equation}

\begin{example}\label{example_fourier_of_char_function}
Taking $\psi = [\{m\}\times S\subset \underline{M}_S,1]$ for some $m\in M$ (interpreted as the characteristic function $\mathbf{1}_{\{m\}}$ of $m$), we obtain 
\[ 
\four \psi  = [S\times_SD(\underline{M}_S),\chi\mapsto \chi(m)] = \ev(m,\cdot).
\]
\end{example}

\begin{proposition}[Motivic Fourier inversion formula]
\label{proposition-motivic-inversion-formula}
Let 
$\psi \in \MMM_{\underline M_S}$
be a motivic function supported over finitely many points of $M$.
Then the motivic Fourier transform of $\psi$ is well-defined and 
    \[
    \psi ( x )
    =
    \int_{D ( \underline M_s ) }
    ( \four \psi ) ( \chi )
    \cdot 
    \ev(-x,\chi)
    \]
for any $x \in \underline M_S$. 
\end{proposition}

\begin{proof}
    By \cref{lemma-motivic-functions-are-defined-on-points} we can assume $S = \Spec ( k ) $.
    The function $\psi$ can be written as a finite linear combination of indicator functions of points of $M = \underline M_{\Spec ( k ) }$,
    so it is enough to prove the statement for an elementary function 
    of the form
    \[
    \mathbf 1_{\{ x_0 \}}
    = 
    \left [
        \{ x_0 \} \hookrightarrow M
    \right ]_M 
    \]
    for some point $x_0$. 
    Then by \cref{example_fourier_of_char_function},
    \[
    \four \mathbf 1_{\{ x_0 \}} 
    =  
    \ev (  x_0, \scdot )
     \mathrm d
    \chi
    \]
    and 
    \[
     \int_{D ( \underline M ) } \four \mathbf 1_{\{ x_0 \}} ( \chi ) 
      \cdot   \ev(-x,\chi)
    = 
     \int_{D ( \underline M ) }
    \ev( x_0 - x,\chi) 
    \mathrm d
    \chi
    = 
    \begin{cases}
       1 & \text{ if } x = x_0 \\
        0 & \text{ otherwise }
    \end{cases}
    \]
 by (\ref{equation:orthogonality}),    hence the identity. 
\end{proof}


\subsection{Local motivic Poisson formula} 

Our
\cref{proposition-motivic-inversion-formula}
is actually a special case of the following. 

\begin{theorem}
[Local motivic Poisson formula]

\label{thm:local-Poisson-formula}
Let $ H\subset G$ be commutative groups. 
We identify $D ( G / H )$ with the kernel of the restriction morphism
$\mathrm{res}_{|H} : D ( G ) \to D ( H ) $. 

Then for any motivic function $\psi\in \MMM_{\underline{G}_S}$ supported over finitely many points of $G$, 
    \[
    \sum_{h \in H}
    \psi ( g + h )
    = 
   \int_{ D ( G / H ) }
    \four  \psi  ( \chi ) e ( - g, \chi ) 
    \mathrm d
    \chi    \]
in $\MMM_{S}$. In particular, 
taking $g = 0_G$,
one gets
    \[
\sum_{h \in H}
    \psi ( h )
    = 
    \int_{ D ( G / H ) }
    \four \psi ( \chi ) 
    \mathrm d 
    \chi.
    \]
\end{theorem}

Here by $\sum_{h\in H}\psi(h)$ we mean the sum of the values of $\psi$ on the finite number of points in the intersection of $H$ with the support of $\psi$, and $\int_{D(G/H)}:\CharM{G}{S}\to \MMM_{S}$  is the operator defined in \cref{def:restricted-local-integration}.
\begin{proof}
By \cref{lemma-motivic-functions-are-defined-on-points} again
we can assume $S = \Spec ( k ) $. 
We can modify the proof of \cref{proposition-motivic-inversion-formula}
by 
replacing the functional \[
\psi \mapsto \psi ( x ) = \sum_{g \in G} \mathbf 1_{\{ 0 \}} ( g ) \psi ( x + g )
\]
by
\[
\psi \mapsto 
\sum_{h \in H} \psi ( x + h )
= 
\sum_{g \in G}
 \mathbf 1_H ( g ) \psi ( x + g ) 
 .
 \]
Up to replacing $\psi$ by a translate, it is sufficient to prove the second identity.  Moreover, by linearity we may assume that $\psi$ is supported on on a single point $g_0$, so that in fact $\psi = \psi(g_0)\mathbf{1}_{\{g_0\}}$. Then the left-hand side is equal to $\psi(g_0)$ if $g_0\in H$, and to 0 otherwise. 
Using  (\ref{equation:orthogonality_with_quotient}),
we get  
\[
 \int_{ D ( G / H ) }
    \four \psi ( \chi )
    \mathrm{d}
    \chi
= \psi(g_0) 
\int_{ D ( G / H ) }
\ev(g_0,\chi)
\mathrm{d}
\chi 
=
\begin{cases} 
\psi(g_0) & \text{if}\ g_0\in H\\
0 & \text{otherwise},
\end{cases}
\]
hence the result. 

\end{proof}


\subsection{Application to the abstract group of divisors}
\label{remark-application-to-Div_S(C)}

The (abstract) group $\Div_\mathsf S ( \mathscr C ) $ of divisors supported on a finite set $\mathsf S$ of points of $\CCC$ may be identified with $\ZZ^{\mathsf S}$. If a function $\psi$ is defined on $\Div_\mathsf S ( \mathscr C ) $, 
    it can be seen, via extension by zero, as a function on $\Div_{\mathsf S'} ( \mathscr C )$ 
    for every $\mathsf S ' $ containing $ \mathsf S$. 
    Then for every such $\mathsf S'$, the associated Fourier transform $\mathscr F \psi ( \chi ) $ will not depend on $\chi_v \in D ( \underline{\mathbf Z}_{\kappa ( v )} ) $ for $v\notin \mathsf S$.  
    
    Under the same hypothesis, 
    if $\psi$ is seen as a function on $\Div ( \CCC ) = \varinjlim_\mathsf S \Div_\mathsf S ( \CCC ) $,
    then for any subgroup $H \subset \Div ( \CCC ) $,
    $\psi$ is $H$-equivariant if and only if $\psi$ is $( H \cap \Div_\mathsf S ( \CCC ) )$-equivariant
    for any $\mathsf S$ such that $\psi$ is supported on $\Div_\mathsf S ( \mathscr C ) $.

Typically, 
in our application we will take 
\[
H
= 
\PDiv_\mathsf S ( \mathscr C ) 
\]
to be the group of principal divisors of $\CCC$ having support contained in a finite set $\mathsf S \subset | \CCC |$ of closed points of the curve $\mathscr C$ (or of $\CCC \otimes k ' $ for some extension $k ' / k $). 
In that case, $H$ is contained in the subgroup 
\[
H_0 = \Div_\mathsf S^0 ( \CCC ) = ( 1 , ... , 1 )^\perp \cap \Div_\mathsf S ( \CCC )
\]
of degree-zero divisors having support in $\mathsf S$. 
Then 
the commutative diagram of exact sequences
\[
\begin{tikzcd}
  & 0  \dar   & 0  \dar  &     & \\
 & H    \dar  & H  \dar  &   & \\
0 \rar & H_0 \dar \rar & G  \rar \dar  & G/H_0 \rar \arrow[d,equal]  & 0  \\
0 \rar & H_0/H \dar  \rar & G/H \dar \rar & (G/H)/(H_0 / H ) \rar & 0 \\
  & 0     & 0   &     & 
\end{tikzcd}
\]
induces by duality a diagram of exact sequences of diagonalisable algebraic groups. 

By choosing once and for all a $k$-divisor $\mathfrak{D}_1$ of degree one on the curve and enlarging $\mathsf S$ if necessary, 
we fix a section of the quotient morphism
\[
\deg : G \to G / H_0 
\]
and a splitting of the first horizontal sequence of $\mathbf Z$-modules
by writing any divisor $D$ as \[(D - \deg ( D ) \mathfrak D_1) + \deg ( D ) \mathfrak D_1.\]
This splitting induces a splitting of the second horizontal sequence
and by duality, splittings of $D ( G )$ and $D ( G / H )$.


\section{Toric harmonic analysis}

\label{section:toric-harmonic-analysis}


\subsection{Local invariant functions and their Fourier transforms} \label{sect:local_functions}
We will now specialize a bit more to the concrete setup that we need for our purposes. For this, we place ourselves in the setting described in the Notations section, with a curve $\CCC$ over a base field $k$. Let $U$ be a split  algebraic torus of dimension $n$ defined over $k$. We denote by $\mathcal{X}^{*}(U) = \Hom(U,\GG_m)$ its group of characters and $\mathcal{X}_{*}(U) = \Hom(\GG_m,U)$ its group of cocharacters.  There is a natural pairing
\[
\langle \cdot,\cdot \rangle: \mathcal{X}^{*}(U)\times \mathcal{X}_{*}(U) \to \ZZ.
\]
The aim of this section is to define motivic incarnations of families of local functions $(f_v:U(F_v)\to \CC)_{v\in \CCC}$ which are $U(\OO_v)$-invariant. The starting point here is that the local degree mapping
$$\deg_{F,v}:U(F_v)\to \mathcal{X}_*(U)$$
induces an isomorphism $$U(F_v)/U(\OO_v)\simeq \mathcal{X}_{*}(U),$$
so that such a local function may be viewed as a function on the free $\ZZ$-module $\mathcal{X}_{*}(U)$ of rank~$n$. 
Moreover, in our setting, we only need to consider such functions with finite support in $\mathcal{X}_{*}(U)$. 
This motivates the following definition: 

\begin{definition}\label{def:local_invariant_functions}
    A family of local motivic $U(\OO)$-invariant functions is an element of the ring $\MMM_{\underline{\mathcal{X}_{*}(U)}_{\CCC}}$ supported over finitely many points of $\mathcal{X}_*(U)$. These elements form a subring of  $\MMM_{\underline{\mathcal{X}_{*}(U)}_{\CCC}}$ denoted by 
    $\Fun(U(F_v)).$
\end{definition}

\begin{notation}\label{notation:characteristic_functions} Let $\aA \in \mathcal{X}_{*}(U)$ be a cocharacter of $U$. We denote by 
\[
\1_{ \tT^{\aA} U(\OO) }
\] 
the family of local motivic $U(\OO)$-invariant functions  given by the class of the subscheme
\[
\{\aA\}\times \CCC 
\hookrightarrow \mathcal{X}_{*}(U) \times_\kb \CCC  = \underline{\mathcal{X}_{*}(U)}_\CCC
\]
viewed in $\Fun(U(F_v)).$ When picking identifications $U\simeq \GG_m^n$ and $\mathcal{X}_{*}(U)\simeq \ZZ^n$, so that $\aA $ is viewed as a tuple of integers $(a_1,\ldots,a_n)$, it is the motivic incarnation of the family of characteristic functions of the subsets
\[
t_1^{a_1}\OO^{\times}_v \times \ldots \times t_n^{a_n}\OO^{\times}_v
    \subset 
(F_v^{\times})^n .
\]
as $v$ runs over points of $\CCC$.

\end{notation}

\begin{remark}
    The forgetful morphism
     $$\Fun(U(F_v))\to\MMM_{\CCC}$$
     may be interpreted as integration over $U(F_v)$ for each $v$. 
\end{remark}

By the theory of 
\cref{subsect:general_local_fourier_transforms} applied to the constant group scheme $\underline{\mathcal{X}_{*}(U)}_{\CCC}$ over $\CCC$, we then have access to a notion of Fourier transformation for such a family of functions, which naturally lives in $\CharM{\mathcal{X}_{*}(U)}{\CCC}$.

\begin{definition}
    A family of local motivic functions with characters on $U$ is
    an element of $\CharM{\mathcal{X}_{*}(U)}{\CCC}$. 
\end{definition}

Thus, given $f = (f_v)_{v}\in \Fun(U(F_v))$ 
a family of local motivic $U(\OO)$-invariant functions, 
its Fourier transform $\FFF f$ in the sense of \cref{def:motivic-Fourier-transform}
is a family of local motivic functions with characters on $U$. 

\begin{example}\label{example:fourier_of_char_function}
Let $\aA \in \mathcal{X}_*(U)$. 
By example \ref{example_fourier_of_char_function}, we have
\[ 
\FFF\1_{\tT^{\aA}U(\OO)} 
= \ev(\aA,\cdot) \in \Char_{\mathcal{X}_{*}(U)}\MMM_{\CCC}. 
\]
\end{example}


\subsection{Symmetric powers and divisors on the curve}

Before proceeding to describing the global situation, we insert a quick discussion of the way that symmetric products of the curve $\CCC$ will allow us to do Fourier analysis on the group of divisors of the curve. 

Let $I$ be a set, and $\pi = (n_i)_{i\in I}$ a partition. A schematic point $E\in \Sym^{\pi}_{/\kb}(\CCC)_*$ can be written 
\[ 
\sum_{i\in I} i(v_{i,1} + \ldots + v_{i,n_i}) 
\]
where the $v_{i,j}$ are geometric points of $\CCC$. 
Fixing an algebraic closure $K$ of the residue field $\kappa(E)$, and defining $\mathsf{S}_{i}$ to be the set of $\mathrm{Gal}(K/\kappa(E))$-orbits of elements in $\{v_{i,1},\ldots,v_{i,n_i}\}$, so that $\mathsf{S}_i$ gives a finite set of closed points  on $\CCC_{\kappa(E)}$. 

\subsection{Global $\KK(U)$-invariant functions and global Fourier transforms}

In this paragraph we define motivic incarnations of global adelic functions. More precisely, denoting by $\KK(U)$ the subset $\prod_{v}U(\OO_v)\subset U(\mathbf{A}_{F})$,
we define motivic incarnations of $\KK(U)$-invariant global functions of the form 
\[
f = \prod_{v}f_v,
\]
where the $f_v$ are local $U(\OO_v)$-invariant functions. For this to make sense motivically, we assume that the $f_v$ are chosen among a finite number of different possibilities.

Concretely, let $I$ be a set and let $(f_i)_{i\in I}$ a collection of families of local motivic $U(\OO)$-invariant functions. 
Then the number of factors $n_i$ where $f_v$ is given by $f_{i,v}$ is encoded by a partition $\pi = (n_i)_{i\in I}$. 

\begin{definition}\label{def-globalfunctions}
    Let $I$ be a set and let $\pi = (n_i)_{i\in I}$ be a partition. A family of global motivic $\KK(U)$-invariant functions of level $\pi$ is a finite linear combination, with integral coefficients, of elements of the form
    \[
    \Sym^{\pi}_{\CCC / \kb } 
    ((f_i)_{i\in I})_*
        \in 
    \mathscr M_{\Sym^{\pi}_{\CCC/ \kb} (\underline{\mathcal{X}_*(U)}_{\CCC})_* }
    \] 
    (in the sense of \cref{subsection:symmetric-products-of-functions})
    where for every $i\in I$, $f_i$ is a family of local motivic $U(\OO)$-invariant functions
    (having finite support, by definition). 
\end{definition}

\begin{remark} A family of global motivic $\KK(U)$-invariant functions of level $\pi$ naturally lives above 
\[
\Sym^{\pi}_{\CCC/\kb}(\underline{\mathcal{X}_{*}(U)}_{\CCC})_*
\to \Sym^{\pi}_{/\kb}(\CCC)_*.
\]
The bottom symmetric product $ \Sym^{\pi}_{/\kb}(\CCC)_*$ parametrises supports, and the intermediate symmetric product $\Sym^{\pi}_{\CCC/\kb}(\underline{\mathcal{X}_{*}(U)}_{\CCC})_*$ parametrises domains of definition. 

Concretely, the fibre of $ \Sym^{\pi}_{\CCC / \kb } ((f_i)_{i\in I})$ above $E = \sum i(v_{i,1} + \ldots + v_{i,n_i})\in \Sym^{\pi}_{/\kb}(\CCC)_*$ is interpreted as the adelic function given by $f_i$ at the places $v_{i,1},\ldots,v_{i,n_i}$ for each $i\in I$ (and just equal to 1 at all other places), and thus defined on \[\prod_{i\in I}U(F_{v_{i,1}})\times\ldots \times  U(F_{v_{i,n_i}})\] which, since the functions are $U(\OO)$-invariant, is modeled by the fiber $(\Sym^{\pi}(\underline{\mathcal{X}_{*}(U)}_{\CCC}))_{E} = \prod_{v\in |E|}\mathcal{X}_{*}(U)_v.$
\end{remark}

\begin{example}[Family of characteristic functions] 
\label{example:families-of-characteristic-functions}
Assume that 
there exists $\aA_i$ such that
\[
    f_i = \1_{t^{\aA_i} U(\OO)}
\]
for every $ i \in I$.
Then the fibre of $\Sym^{\pi}_{\CCC/\kb}((f_i)_{i\in I})_*$ above $E = \sum i(v_{i,1} + \ldots + v_{i,n_i})\in \Sym^{\pi}_{/\kb}(\CCC)_*$ is the global function
\[
\prod_{i\in I}\prod_{j=1}^{n_i}\1_{t^{\aA_i}U(\OO_{v_{i,j}})}.
\]
For example, if we are given a map $\gamma^\vee : I \to \mathcal X_* ( U )$, as it will be the case latter, then we can take $\bm a_i = \gamma^\vee ( i )$ for all $i \in I$.
Its provides a natural section
\[
\gamma^\vee : 
\Sym^\pi_{/\kb} ( \mathscr C )_* 
\to \Sym^{\pi}_{\CCC/\kb}(\underline{\mathcal{X}_{*}(U)}_{\CCC})_*
\]
given by 
\[
 \sum_{i\in I}
        i ( v_{i,1}+ \ldots + v_{i,n_i} )
\mapsto 
 \sum_{i\in I}
        i (( \gamma^\vee ( i ) ,v_{i,1})+ \ldots + ( \gamma^\vee ( i ) ,v_{i,n_i})) .
\]
\end{example}


\subsection{Summation over rational points} \label{subsection:summation_rat_points}
In this section, 
we define a motivic analogue of the operation which, to an appropriately integrable function $f:U(\mathbf{A}_F) \to \CC$ on the idèles associates
\[
    \sum_{x\in U(F)} 
        f(x).
\]
More specifically, since in the previous paragraph we have defined our functions in families (parametrised by symmetric products), we will explain how to perform this operation simultaneously for all functions in such a family. 

\begin{proposition}\label{prop:map-to-pic}
There is a morphism
\[
    \Sym^{\pi}_{\CCC/\kb}(\underline{\mathcal{X}_*(U)}_{\CCC})_*
    \to 
    \Hom_{\mathrm{gp}}(\mathcal{X}^{*}(U),\PPic(\CCC))
\]
given by
\[
    \sum_{i\in I}
        i((\bB_{i,1},v_{i,1})+ \ldots + (\bB_{i,n_i},v_{i,n_i}))
    \mapsto 
    \left( \aA\mapsto \left[\sum_{i,j}\langle \aA, \bB_{i,j}\rangle v_{i,j}\right]\right).
\]
\end{proposition}
\begin{proof} 
   To simplify notation, we will fix an isomorphism $\mathcal{X}_*(U)\simeq \ZZ^n$, so that our map becomes 
\[
    \begin{array}{ccc}
        \Sym^{\pi}_{\CCC/\kb}(\underline{\ZZ^{n}}_{\CCC})_* & \to & \PPic(\CCC)^n \\
        \sum_{i\in I}i((\bB_{i,1},v_{i,1})+ \ldots + (\bB_{i,n_i},v_{i,n_i}))& \mapsto & \left[\sum_{i,j}\bB_{i,j}v_{i,j}\right]
    \end{array}
\]
   where the square brackets mean that we take linear equivalence classes in each component. For every $\bB=(b_1,\ldots,b_n)\in \ZZ^n$ we have a morphism
\[
    \CCC\to \PPic^{b_1}(\CCC)\times\ldots\times \PPic^{b_n}(\CCC)
\]
   sending a point $v$ to the tuple of linear equivalence classes $(b_1[v],\ldots,b_n[v])$. Combining these all together gives a morphism
\[
    \underline{\ZZ^n}_{\CCC}\to \PPic(\CCC)^n.
\]
Let now $\pi = (n_i)_{i\in I}$ be a partition. Taking products and restricting to the complement of the diagonal in $\prod_{i\in I}\CCC^{n_i}$ we obtain a morphism
\[
    \left(\prod_{i\in I}\left(\underline{\ZZ^n}_{\CCC}\right)^{n_i}\right)_{*}\to \PPic(\CCC)^n,
\]
given by $(\bB_{i,j},v_{i,j})_{\substack{i\in I\\j\in \{1,\ldots,n_i\}}}\mapsto \sum_{i,j}\bB_{i,j}[v_{i,j}].$ This passes to the quotient with respect to the $\prod_{i\in I} \mathfrak{S}_{n_i}$--action, giving the morphism we want.  
\end{proof}

\begin{notation} 
\label{notation:PDiv}
We denote by $\PDDiv^{\pi}(\CCC)$ the fiber above $0$ of the morphism of \cref{prop:map-to-pic}.
\end{notation}

Recall that we also have a morphism $  \Sym^{\pi}_{\CCC/\kb}(\underline{\mathcal{X}_*(U)}_{\CCC})_* \to   \Sym^{\pi}( \CCC)$.

\begin{definition}\label{def:summation-over-PDiv}
Given $f\in \MMM_{\Sym^{\pi}_{\CCC/\kb}(\underline{\mathcal{X}_*(U)}_{\CCC})_*}$ a family of global motivic functions of level $\pi$, we denote by 
\[
 \sum_{\PDiv^{\pi}} f
\]
the image in $\MMM_{\Sym^{\pi}_{/\kb}(\CCC)_*}$ of the pullback of $f$ to $\PDDiv^{\pi}(\CCC)$. 
\end{definition}

\begin{notation}\label{notation:summ-over-PDiv-Euler-prod} Using the formalism of incidence algebras from \cref{subsection:motivic_euler_products}, we may then combine the operators $\sum_{\PDiv^{\pi}}$ into an operator 
\[
\sum_{\PDiv}: \MMM_{\Sym^{\bullet}_{\CCC/k}(\underline{\mathcal{X}_*(U)}_{\CCC})} \to \MMM_{\Sym^{\bullet}_{/\kb}(\CCC)}.
\]

\end{notation} 


\subsection{Global characters and character functions}

Analogously to  \cref{def-globalfunctions}, we define:
\begin{definition}
\label{def-globalfourierfunctions} 
A family of global  functions with characters of level $\pi$ is a finite linear combination, 
with integral coefficients, 
of elements of the form 
\[
\Sym^{\pi}_{\CCC / k} ((g_i)_{i\in I})_*\in \CharM{\pi, \mathcal{X}_*(U)}{\Sym^{\pi}_{/\kb}(\CCC)_*}
\]
where $(g_i)_{i\in I}$ is a family of local motivic character functions on $U$. 
\end{definition}

A family of global functions with characters of level $\pi$ naturally lives above the symmetric product $\Sym^{\pi}_{\CCC/k}(D ( \underline{\mathcal X_* ( U )}_\CCC ))\to \Sym^{\pi}_{/\kb}(\CCC)_*$, which we may think of as parameterizing families of global characters with support of size $|\pi|$. The fibre of $\Sym^{\pi}_{\CCC/k}((g_i)_{i\in I})$ above  $E = \sum i(v_{i,1} + \ldots + v_{i,n_i})\in \Sym^{\pi}_{/\kb}(\CCC)_*$ represents, through the local isomorphism $D(\underline{\mathcal{X}_{*}(U)}_{\CCC})_{v_{i,j}}\simeq U_{v_{i,j}}$, functions defined on the character domain  $$\prod_{i\in I} U_{v_{i,1}}\times \ldots \times U_{v_{i,n_i}}.$$ 

The natural evaluation pairing $\underline{\mathcal{X}_{*}(U)}_{\CCC}\times_{\CCC} D(\underline{\mathcal{X}_{*}(U)}_{\CCC})\to \GG_m$ 
induces a pairing
\[
\Sym^{\pi}_{\CCC/k}(\ev): \Sym^{\pi}_{\CCC/k}(\underline{\mathcal{X}_{*}(U)}_{\CCC})_* \times_{\Sym^{\pi}_{/\kb}(\CCC)_*}\Sym^{\pi}_{\CCC/k}(D(\underline{\mathcal{X}_{*}(U)}_{\CCC}))_*\to \GG_m,
\]
perfect relatively to $\Sym^{\pi}_{/\kb}(\CCC)_*$, in the sense that for every $E\in \Sym^{\pi}_{/\kb}(\CCC)_*$, it induces a perfect pairing
\[
\Sym^{\pi}_{\CCC/k}(\underline{\mathcal{X}_{*}(U)}_{\CCC})_{*,E} \times_{\kappa(E)}\Sym^{\pi}_{\CCC/k}(D(\underline{\mathcal{X}_{*}(U)}_{\CCC}))_{*,E}\to \GG_m,
\]
between the corresponding fibres.

\begin{remark}[Compatibility between Fourier transformation and symmetric products] \label{remark:compatibility_four_sym}
Let $(f_i)_{i\in I}$ be a collection of families of local motivic $U(\OO)$-invariant functions, i.e. each $f_i\in \MMM_{\underline{\mathcal X_{*}(U)}_{\CCC}}$, supported over a finite number of points of $\mathcal X_{*}(U)$. One may check that
\[
\Sym^{\pi}_{\CCC/\kb}(\four f_i)_{i\in I} = \Sym^{\pi}_{\CCC/\kb}((f_i)_{i\in I})_*\cdot \Sym^{\pi}_{\CCC/k}(\ev)
\]
where the product is taken in $\CharM{\pi,\mathcal X_{*}(U)}{\Sym^{\pi}_{\CCC/\kb}(\underline{\mathcal X_{*}(U)}_{\CCC})_*}$ and then viewed in $\CharM{\pi,\mathcal X_{*}(U)}{\Sym^{\pi}_{/\kb}(\CCC)_*}$.
\end{remark}


\subsection{From global to local} \label{remark:fibre_symproduct_and_pdiv} We explain in this section how, looking at the fibre above a point $E\in \Sym^{\pi}_{/\kb}(\CCC)_*$ we come back to a local situation like the one in \cref{section:multiplicative-motivic-harmonic-analysis}. 

Fix an identification $\mathcal{X}_{*}(U)\simeq \ZZ^n.$ Let $E\in \Sym^{\pi}_{/\kb}(\CCC)_*$, written $E = \sum_{i}i(v_{i,1} + \ldots + v_{i,n_i})$ and let, for every $i\in I$, $\mathsf{S}_i$ be the set of Galois orbits in $\{v_{i,1},\ldots,v_{i,n_i}\}$. Then the map
\[
\begin{array}{ccc} \Sym^{\pi}(\underline{\ZZ^n}_{\CCC}) & \to & \prod_{i\in I} \Div^n(\CCC)\\
 \sum_{i\in I}i((\bB_{i,1},v_{i,1})+ \ldots + (\bB_{i,n_i},v_{i,n_i}))& \mapsto & (\sum_{j}\bB_{i,j}v_{i,j})_{i\in I}
    \end{array}
\]
induces group isomorphisms
\[ 
\Sym^{\pi}(\underline{\ZZ^n}_{\CCC})_E \to \prod_{i\in I} \Div^{n}_{\mathsf{S}_i}(\CCC_{\kappa(E)}).
\]
Through this isomorphism, the subgroup $\PDiv^{\pi}(\CCC)_E\subset \Sym^{\pi}(\underline{\ZZ^n}_{\CCC})_E$ corresponds to the subgroup
\[ 
H_{\mathsf{S}} =\left\{ \left(\sum_{j}\bB_{i,j}v_{i,j}\right)_{i\in I} ,\ \sum_{i,j} \bB_{i,j}v_{i,j} \ \text{is principal in each component}\ \right\} 
\]
inside 
\[
G_{\mathsf{S}} = \prod_{i\in I} \Div^{n}_{\mathsf{S}_i}(\CCC_{\kappa(E)})
\]

As for the dual side, the fibre $\Sym^{\pi}_{\CCC/k}(D(\underline{\mathcal{X}_{*}(U)}_{\CCC}))_E$ is identified with the group of characters $D(G_{\mathsf{S}})$.


\subsection{Integration over invariant characters}

Let $A\in \CharM{\pi,\mathcal{X}_{*}(U)}{\Sym^{\pi}_{/\kb}(\CCC)_*}$. Then via \cref{remark:fibre_symproduct_and_pdiv}, for every $E\in \Sym^{\pi}_{/\kb}(\CCC)_*$, the fibre $A_E$ naturally defines an element of $\CharM{G_{\mathsf{S}}}{\kappa(E)}$. 

We then define
\[
\int_{E}A_E := \int_{D(G_{\mathsf{S}}/H_{\mathsf{S}})} A_E.
\]

\begin{definition}
    We say that $A\in \CharM{\pi,\mathcal{X}_{*}(U)}{\Sym^{\pi}_{/\kb}(\CCC)_*}$ is \emph{integrable} if there exists $B\in \MMM_{\Sym^{\pi}_{/\kb}(\CCC)_*}$ such that for all $E\in \Sym^{\pi}_{/\kb}(\CCC)_*$
    \[
    \int_{E}A_E = B_E.
    \]
    In this case we write
    \[ 
    \int_{(\PDiv^\pi)^{\perp}} A = B.
    \]
\end{definition}

\subsection{The motivic multiplicative Poisson formula}

Let $f = \Sym^{\pi}((f_i)_{i\in I})_*$ be a family of global motivic $\KK(U)$-invariant functions of level $\pi$, with  corresponding family of Fourier transforms 
\[
\four f := \Sym^{\pi}((\four f_i)_{i\in I})_* \in \CharM{\pi, \mathcal{X}_*(U)}{\Sym^{\pi}_{/\kb}(\CCC)_*}.
\]

\begin{theorem} 
\label{thm:mot-mult-Poisson-formula} 
Let $f$ be a family of global motivic $\KK(U)$-invariant functions of level $\pi$. Then $\four f$ is integrable and we have the equality
\begin{equation}\label{equation:Poisson_formula}
\sum_{\PDiv^{\pi}}f = \int_{(\PDiv^{\pi})^{\perp}}\four f
\end{equation}
in $\MMM_{\Sym^{\pi}_{/\kb}(\CCC)_*}$.
\end{theorem}

\begin{proof} First of all, we may assume that $f = \Sym^{\pi}_{\CCC/k}(f_i)_{i\in I}$ where $f_i$ are families of local motivic $U(\OOO)$-invariant functions. 
We fix a support divisor
$E= \sum_{i\in I}i
(v_{i,1} + \ldots + v_{i,n_i})\in \Sym^{\pi}_{/\kb}(\CCC)_* $ and we define $\mathsf{S}$, as well as the groups $H_{\mathsf{S}}\subset G_{\mathsf{S}}$ as in \cref{remark:fibre_symproduct_and_pdiv}. We note that $f_E\in \MMM_{G}$ is supported over finitely many points of $G$, so we may apply \cref{thm:local-Poisson-formula}, getting:
\[
\sum_{h\in H_{\mathsf{S}}} f_E(h) = \int_{D(G_{\mathsf{S}}/H_{\mathsf{S}})}\four f_E.
\]
On the other hand, we have, by definition, 
\[
\int_{E}(\four f)_{E} =  \int_{D(G_{\mathsf{S}}/H_{\mathsf{S}})} (\four f)_E
\]
Since $\four f_{E} = (\four f)_{E}$ by \cref{remark:compatibility_four_sym}, this shows integrability and, using \cref{lemma-motivic-functions-are-defined-on-points}, also equation (\ref{equation:Poisson_formula}). 

\end{proof}

\section{Application to the motivic height zeta function}

\label{section:application-motivic-height-zeta-function}


\subsection{An introductory example: the projective line}
We explain without any proof what are the objects coming into play in the very specific case of 
$X = \PP^1_{F}$ viewed as a toric variety. 
Define for every non-zero $m\in \ZZ$
\[
\1_{t^m\OO^{\times}} = \{m\}\times \CCC
\]
as a subscheme of the constant group scheme $ \underline \ZZ_\CCC $. Its fibre above $v\in \CCC$ is denoted 
$\1_{t^m\OO_v^{\times}}$ 
and is the motivic incarnation of the $\OO_v^\times$-invariant function 
$\1_{t^m\OO_v^{\times}}$ 
on the completion $F_v$ of $F$ at $v$. 
Thus we get a family 
\[
( \1_{t^m\OO^\times} )_{m\in \ZZ} 
\]
of subschemes of $ \underline \ZZ_\CCC $ which can be used to define a generalised motivic Euler product like those that are described above. 

In this notation, denoting by $T_-, T_+$ our indeterminates, the local height is the motivic function on $\underline \ZZ_\CCC$
\[
H_v(\cdot,\mathbf{T}) 
    = 
1 
+ \sum_{m >0} 
\1_{t^{-m}\OO_v^\times} T_-^{m} 
+ \sum_{m > 0}
\1_{t^m\OO_v^\times}T_+^m
\in \KVar {\underline \ZZ_\CCC} [ \mathbf{T} ] 
\]
fitting into a motivic Euler product
\[
H(\mathbf{T},\cdot ) 
= \prod_{v\in \CCC}
        \left(
            1 + \sum_{m >0} \1_{t^{-m}\OO_v^\times} T_-^{m} + \sum_{m>0} \1_{t^m\OO_v^\times}T_+^m
        \right)
        . 
\]
The expansion of this motivic Euler product is indexed by pairs of non-negative integers $(n_+,n_-)$, and the coefficient corresponding to $(n_+,n_-)$ lives in the Grothendieck ring above the variety $\Sym^{(n_+,n_-)}_{/ \kb} ( \CCC )_*$ 
parametrising zero-cycles $D$ on $\CCC$ of the form 
\[
D_+ - D_- 
\]
where $D_+,D_-$ are effective with disjoint supports (hence the subscript $*$), of respective degrees $n_+$ and $n_-$. Each coefficient may be viewed as a family of characteristic functions of balls inside the idèles $\mathbb{A}^{\times}_F$, parametrised by $\Sym^{(n_+,n_-)}(\CCC)_*$, the fibre above $D = \sum_{v}m_vv$ being the (motivic incarnation of the) function $\1_{D}:= \prod_{v} \1_{t^{m_v}\OO^{*}_v}$.

In \cref{subsection:summation_rat_points} we defined a summation operator \[
\sum_{\PDiv}
\]
as a pullback to principal divisors, so that applying this operator to $H(\cdot,\mathbf{T})$ kills the coefficients with $n_+ \neq n_-$, and for those with $n_+ = n_- = n$, only the divisors $D$ which are principal contribute. This way 
\[
(\mathbf{L}-1)
\sum_{\PDiv}H(\cdot,\mathbf{T}) = \sum_{n\geq 0}
\left [ 
    \Hom^n(\CCC,\mathbf{P}^1)_{\GG_m} 
\right ] 
\mathbf{T}^{\langle \cdot, n\rangle_{\Sigma}} . 
\]


The relative Cartier dual of $ \underline \ZZ_\CCC $ is $\GG_{m,\CCC}$
and the local Fourier transforms are 
\[
\four \1_{t^m\OO^{\times}}(\cdot) = \mathrm{ev}_{\CCC}(m,\cdot), 
\]
where 
\[
\mathrm{ev}_{\CCC}: \ZZ_{\CCC}\times_{\CCC} \GG_{m,\CCC} \to \GG_{m,\CCC}
\]
is the evaluation map (relative to $\CCC$). Explicitly,
$\four \1_{t^m\OO^{\times}}( \chi_v ) = \chi_v(m). $
Thus, 
\begin{align*}
\four H_v(T,\cdot)( \chi_v )  &=  1 + \sum_{m< 0} \chi_v(m)T_-^{m} + \sum_{m> 0} \chi_v(m)T_+^{m} \\  
& = \frac{1}{1-\chi_{v}(1)T_+} + \frac{1}{1-\chi_{v}^{-1}(1)T_-} - 1 \\
& = \frac{1-T_-T_+}{(1-\chi_v(1)T_+)(1-\chi_v^{-1}(1)T_-)}. 
\end{align*}
Note that this special case has the peculiarity that the numerator of the local Fourier transform does not depend on the character.
Taking motivic Euler products with characters in the sense of \cref{section:motivic-Euler-products}
provides the motivic Euler product with characters
\begin{align*}
\prod_{v\in \CCC} 
    \four H_v(T,\cdot)( \chi_v  ) 
&  = \prod_{v\in \CCC} \frac{1-T_-T_+}{(1-\chi_v(1)T_+)(1-\chi_v^{-1}(1)T_-)} \\
&  = Z_{\CCC}^{\mathrm{Kap}}(T_-T_+)^{-1} L(\chi, T_+)L(\chi^{-1}, T_-), 
\end{align*}
where $(\chi_v)$ is a family of local characters all living in $\GG_{m}$, and $L$ denotes the motivic $L$-functions from \cref{def:motivic-L-function}. 

Applying the motivic Poisson formula to each of the coefficients of this series leads to the expression
\[
\sum_{d \in \mathbf N}
\left [
\Hom_\kb ( \CCC , \mathbf P^1 )_{U}
\right ] T^d 
= 
\left ( 
\LL - 1 
\right ) 
\int_{\PDiv^\perp } \four H(T, \chi ) \mathrm d \chi . 
\]



\subsection{General geometric setting}

Let $k$ be a field of characteristic zero and let $U$ be an algebraic torus of dimension $n$ defined over $k$. We denote by $\mathcal{X}^{*}(U) = \Hom(U,\GG_m)$ its group of characters and $\mathcal{X}_{*}(U) = \Hom(\GG_m,U)$ its group of cocharacters.  There is a natural pairing
\begin{equation}\label{eq:toric-pairing}\langle \cdot,\cdot \rangle: \mathcal{X}^{*}(U)\times \mathcal{X}_{*}(U) \to \ZZ.\end{equation}

Let $\Sigma$ be a projective and regular fan of the $\ZZ$-module $\mathcal{X}_{*}(U)$. We denote by $\Sigma(1)$ the set of its rays (that is, its one-dimensional faces), and for every cone $\sigma\in \Sigma$ we denote by $\sigma(1)$ the set of rays of $\sigma$. For every $\alpha\in \Sigma(1)$, we denote by $\rho_{\alpha}$ a generator of the ray $\alpha$. We denote by $\mathrm{PL}(\Sigma)$ the group of piecewise linear functions on $\Sigma$. The pairing (\ref{eq:toric-pairing}) 
naturally extends to a pairing 
\begin{equation}\label{eq:PL-pairing}
\langle\cdot,\cdot\rangle_{\Sigma}:\mathrm{PL}(\Sigma)\times \mathcal{X}_{*}(U) \to \ZZ.
\end{equation}
given by $\langle \varphi,\bm n\rangle_{\Sigma} = \varphi(\bm n)$,
in the sense that the diagram 
\[
\begin{tikzcd}
\mathcal X^* ( U ) \times \mathcal X_* ( U ) \rar["{\langle \cdot , \cdot \rangle}"] \dar & \mathbf Z \dar \\
\PL ( \Sigma ) \times \mathcal X_* ( U ) \rar["{\langle \cdot , \cdot \rangle_\Sigma}"] & \mathbf Z
\end{tikzcd}
\]
commutes. 

We now assume that $U= \GG_{m}^n$ is split. Then the fan $\Sigma$ defines a smooth projective split toric variety $X_{\Sigma}$ with open orbit $U$. Each $\rho_{\alpha}$ gives rise to a $U$-invariant divisor $\mathfrak{D}_{\alpha}$ on $X_{\Sigma},$ and the natural map
\[\varphi\mapsto \sum_\alpha \varphi ( \rho_\alpha ) \mathfrak D_\alpha\]
defines an isomorphism $\mathrm{PL}(\Sigma) \simeq \bigoplus_{\alpha}\ZZ \mathfrak{D}_{\alpha}$. We will often use this identification in what follows: in particular, $\mathfrak{D}_{\alpha}$ represents the piecewise linear function which sends $\rho_{\alpha}$ to 1 and all other $\rho_{\beta},\ \beta\neq \alpha$ to 0. 

We then have an exact sequence
\begin{equation}
    \label{equ:fundamental-exact-sequence-toric}
    0
    \longrightarrow
    \mathcal X^* ( U ) 
    \overset{\gamma}{
    \longrightarrow}
    \PL ( \Sigma ) 
    \longrightarrow
    \Pic ( X_\Sigma ) 
    \longrightarrow
    0
\end{equation}
where $\gamma$ sends an element $\bm m \in \mathcal X^* ( U ) $
to the unique $\varphi $ in $ \PL ( \Sigma )$ such that $\varphi ( \rho_\alpha ) = \langle \bm m , \rho_\alpha \rangle $
for every $\alpha \in \Sigma ( 1 ) $,
while the second arrow sends each $\mathfrak{D}_{\alpha}$ to its class in $\Pic ( X_\Sigma ) $.

In what follows, we will work with a family of variables $(T_{\alpha})_{\alpha\in \Sigma(1)}$ indexed by $\Sigma(1)$. We will write
\[\mathbf{T}^{\langle \cdot,\bm n\rangle_{\Sigma} }:= \prod_{\alpha}T_{\alpha}^{\langle \mathfrak{D}_{\alpha},\bm n\rangle_{\Sigma}}. \]
Let $\sigma$ be the cone of $\Sigma$ containing $\bm n$. Then, by definition, there is a unique representation $\bm n = \sum_{\alpha\in \sigma(1)} n_{\alpha}\rho_{\alpha}$ for some positive integers $n_{\alpha}>0$, and $\mathbf{T}^{\langle \cdot,\bm n\rangle_{\Sigma} } = \prod_{\alpha\in\sigma(1)}T_{\alpha}^{n_{\alpha}}.$

For any $\sigma \in \Sigma$, let $I_\sigma$ be the subset of $\mathbf N^{\Sigma ( 1 )}$ given by
\[
I_\sigma = \{ \bm n \in \mathbf N^{\Sigma ( 1 )} \mid n_\alpha > 0 \Leftrightarrow \alpha \in \sigma ( 1 ) \}. 
\]
In particular, $I_{\{ \bm 0 \}} = \{ \bm 0 \}$ and elements of $I_\sigma$ have at most $n$ non-zero coordinates and equality holds if and only if $\sigma$ is maximal. 
Moreover, there is a bijection
\begin{equation}
\label{equation:decomposition-of-cochar-cones-of-Sigma}
\mathcal X_* ( U ) \longrightarrow \sqcup_{\sigma \in \Sigma} I_\sigma
\end{equation}
sending $\bm 0 \in \mathcal X_* ( U )$ to $\bm 0 \in I_{\{ \bm 0 \}}$
and any $\bm m$ to its unique representation $\sum_\alpha n_\alpha \rho_\alpha$ with $(n_\alpha ) \in I_\sigma$ for $\sigma$ containing $\bm m$ in its relative interior. 
   
\begin{notation}\label{notation:I-sigma}
In the rest of the paper, for every $\sigma \in \Sigma$ we identify $I_\sigma$ with its image in $\mathcal X_* ( U )$ 
(as commutative semi-groups)
via \eqref{equation:decomposition-of-cochar-cones-of-Sigma}. Therefore we are in the situation described in \cref{remark:mixed-sym-prod-with-extra-structure}.
\end{notation}

This convention will allows us to make use of the mixed variant of motivic Euler products from \cref{definition-mixed-motivic-euler-product}, in particular when defining our motivic height function.


\subsection{Local and global degrees}\label{subsect:multiheights}

Before introducing the motivic height zeta function of the toric variety associated to $\Sigma$,
we discuss the construction of the height at the level of rational and adelic points. 
This construction guides us towards the parameterisation we use next.

\subsubsection{Degree on $\GG_{m}$}
Let $\CCC$ be a smooth projective geometrically irreducible curve over $k$, and let $F = k(\CCC)$ its function field. For every closed point $v\in \CCC$, the completion $F_v$ is a valued field, and the corresponding valuation $v$ induces a local degree map:
\[
\begin{array}{rcl} \deg_{F,v}:\GG_{m}(F_v)\to \ZZ\\
x_v\mapsto v(x_v)\end{array}
\]
Combining these for all places $v$, we get the global degree
\[
\divv_F :
\left \{ 
\begin{array}{rll}
  \mathbf G_m ( \mathbb A_F ) & \longrightarrow & \Div ( \mathscr C ) \\
    ( x_v ) & \longmapsto & \sum_{v\in | \CCC |} \deg_{F,v} ( x_v ) [ v ] 
\end{array}
\right. 
\]
whose fiber above a $D = \sum_v d_v [v]$ is nothing else but $\prod_v t_v^{d_v} \mathcal O_v^\times$.

\subsubsection{Degree on an algebraic torus}
Let $U$ be a split algebraic torus of dimension $n$. One can compose $\divv_F$ with any $\bm m \in \mathcal X^* ( U ) = \Hom ( U , \mathbf G_m )$. 
This provides a new morphism
\[
\divv_U :
\left \{ 
\begin{array}{rll}
  U ( \mathbb A_F ) & \longrightarrow & 
  \Hom_\mathrm{gp} ( \mathcal X^* ( U ) , \Div ( \mathscr C ) ) = \oplus_{v\in | \mathscr C |} \mathcal X_* ( U )
  \\
    ( x_v ) & \longmapsto & 
    \left [
    \bm m \mapsto 
   \divv_F 
   \left ( 
        \bm m ( ( x_v ) ) 
    \right )  
    \right ] . 
\end{array}
\right. 
\]
By composing with the Abel-Jacobi morphism $\Div ( \CCC ) \to \Pic ( \CCC )$, 
we also get a map
\[
\overline{\divv_U}
:
\left \{ 
\begin{array}{rll}
  U ( \mathbb A_F ) & \longrightarrow & 
  \Hom_\mathrm{gp} ( \mathcal X^* ( U ) , \Pic ( \mathscr C ) ) 
  \\
    ( x_v ) & \longmapsto & 
    \left [
    \bm m \mapsto 
   [\divv_F 
   \left ( 
        \bm m ( ( x_v ) ) 
    \right ) ] 
    \right ] . 
\end{array}
\right. 
\]
The fiber of $\overline{\divv_U}
$ 
above $0\in \Hom_\mathrm{gp} ( \mathcal X^* ( U ) , \Pic ( \mathscr C ) )$ is exactly the set $U(F)$. 
Composing $\divv_U$ with the degree morphism
\[
\left \{
\begin{array}{rll}
    \oplus_{v\in | \mathscr C |} \mathcal X_* ( U ) & \longrightarrow  & \mathcal X_* ( U ) \\
    \sum_{v\in  | \mathscr C |} \bm m_v [v] & \longmapsto & \sum_{v\in | \mathscr C |} \bm m_v \deg ( \kappa ( v ) : k ),
\end{array}
\right. 
\]
we also obtain a degree morphism $\degg_U: U(\mathbb{A}_F)\to \mathcal{X}_*(U).$
The logarithmic multi-height of a point $x \in U ( \mathbb A_F )$
is defined to be 
\[
    h ( x , \cdot ) 
     = 
    \langle \cdot, \degg_U ( x ) \rangle_{\Sigma} \in \PL ( \Sigma )^\vee . 
\]

\subsubsection{Parameterisation of $\oplus_{v\in | \mathscr C |} \mathcal X_* ( U )$}

Using the bijection given by \eqref{equation:decomposition-of-cochar-cones-of-Sigma} and \cref{notation:I-sigma},
we can parameterise elements of 
\[
\Hom ( \mathcal X^* ( U ) , \Div ( \mathscr C ) )
\simeq 
\oplus_{v\in | \mathscr C |} \mathcal X_* ( U )
\]
by the family of mixed symmetric products of the form
\[
\Sym_{/k}^{ ( \pi_\sigma )_{\sigma \neq \bm 0} } ( \mathscr C )_* 
\]
where for all $\sigma \neq \bm 0$, $\pi_\sigma$ is a partition (which can be trivial) indexed by $I_\sigma$. 
Similarly,
the algebro-geometric analogue of
\[
\Hom ( \mathcal X^* ( U ) , \Pic ( \mathscr C ) )
\]
is 
\[
\HHom_{k-\text{gp}}  ( \underline{\mathcal X^* ( U )}_k , \PPic ( \mathscr C ) )
\]
which is isomorphic to a product of $n$ copies of $\PPic ( \mathscr C ) $. 
Using the splitting 
\[
\PPic ( \mathscr C ) \simeq \PPic^0 ( \mathscr C ) \oplus \underline{\mathbf Z}_k 
\]
given by $\mathscr O_\mathscr C ( \mathfrak D_1 )$,
we see that a point of this latter group scheme is determined by an element of $\mathcal X_* ( U )$ (its tuple of degrees) and an $n$-tuple of linear classes in $\PPic^0 ( \mathscr C )$. Finally, the morphism 
\[
\Hom ( \mathcal X^* ( U ) , \Div ( \mathscr C ) )
\longrightarrow
\Hom ( \mathcal X^* ( U ) , \Pic ( \mathscr C ) )
\]
also admits an algebro-geometric incarnation: 
\begin{equation} \label{eq:Abel-Jacobi-morphism-for-n-tuples-of-divisors}
\begin{array}{rll}
  \Sym_{/k}^{ ( \pi_\sigma )_{\sigma \neq \bm 0} } ( \mathscr C )_*
& \longrightarrow & \PPic ( \mathscr C )^n 
\simeq \HHom_{k-\text{gp}}  ( \underline{\mathcal X^* ( U )}_k , \PPic ( \mathscr C ) )
\\
    ( E_\sigma = \sum_v \bm m_{\sigma,v} [v] ) & \longmapsto &  [ \sum_\sigma E_\sigma ]
\end{array} 
\end{equation}
whose fibre above $\bm 0$ encodes $n$-tuples of principal divisors. 

\subsubsection{Effective degrees}
We also have (abstract) group morphisms 
\[
\gamma^\vee : 
\Hom ( \PL ( \Sigma ) , \Div ( \CCC ) )
\simeq \oplus_{v\in | \mathscr C |} \mathcal \PL ( \Sigma )^\vee 
\longrightarrow
\Hom ( \mathcal X^* ( U ) ,  \Div ( \CCC ) )
\simeq \oplus_{v\in | \mathscr C |} \mathcal X_* ( U )
\]
and 
\[
\gamma^\vee : 
\Hom ( \PL ( \Sigma ) , \Pic ( \CCC ) )
\longrightarrow
\Hom ( \mathcal X^* ( U ) ,  \Pic ( \CCC ) )
\]
induced by the composition by the group morphism $\gamma : \mathcal X^* ( U ) \to \PL ( \Sigma ) $ 
from \eqref{equ:fundamental-exact-sequence-toric}. 
In our algebro-geometric setting, elements of $\Hom ( \PL ( \Sigma ) , \Div ( \CCC ) )$ can be parameterised by symmetric products with partitions indexed by $\PL ( \Sigma )^\vee - \{ \bm 0 \}$. In practice these partitions 
will always be supported on 
\[
\PL ( \Sigma )_+^\vee = \PL ( \Sigma )^\vee \cap \mathbf N^{\Sigma ( 1 )}
\]
since the local intersection degrees of a morphism $\CCC \to X$ with the effective divisors $\mathfrak D_\alpha$ will always be non-negative. 
Applying $\gamma$ to the two different index sets $\PL ( \Sigma )^\vee - \{ \bm 0 \}$ and $\mathcal X^* ( U ) - \{ \bm 0 \}$ 
gives us a collection of piecewisely defined morphisms 
\begin{equation}
\label{equation:gamma-vee-for-Sym-definition}
\gamma^\vee :
 \Sym_{/k}^{ \bm d } ( \mathscr C ) 
 := \coprod_{\varpi \vdash \bm d  } \Sym^{\varpi}_{/k} ( \CCC )_*
 \longrightarrow 
  \Sym_{/k}^{ \gamma^\vee ( \bm d ) } ( \mathscr C ) := \coprod_{\pi \vdash \gamma^\vee ( \bm d ) } \Sym^{\pi}_{/k} ( \CCC )_*
\end{equation}
for any $\bm d \in \PL ( \Sigma )^\vee$. Again, for our application our partitions for the left hand side will be indexed by $ \PL ( \Sigma )^\vee_+ - \{ \bm 0 \}$ and therefore in that case $\bm d \in \PL ( \Sigma )^\vee_+$.
Eventually \eqref{equation:gamma-vee-for-Sym-definition} can be refined once more by pulling it back to the mixed products of the form $ \Sym_{/k}^{ ( \pi_\sigma )_{\sigma \neq \bm 0} } ( \mathscr C )_*$. 
Moreover these morphisms are compatible with their counterpart at the level of linear classes 
\begin{equation}
\label{equation:gamma-vee-for-Pic-definition}
\gamma^\vee : 
\HHom_{k-\text{gp}}  ( \underline{\PL ( \Sigma ) }_k , \PPic ( \mathscr C ) )
\longrightarrow 
\HHom_{k-\text{gp}}  ( \underline{\mathcal X^* ( U )}_k , \PPic ( \mathscr C ) ) .
\end{equation}
The fibre of \eqref{equation:gamma-vee-for-Pic-definition} above any point of $\HHom_{k-\text{gp}}  ( \underline{\mathcal X^* ( U )}_k , \PPic ( \mathscr C ) )$
is isomorphic to a product of $r = |\Sigma ( 1 )| - n$ copies of $\PPic^0 ( \CCC )$. 
Indeed, it is given by the subgroup of linear classes $(L_\alpha)_{\alpha \in \Sigma ( 1 )}$
satisfying 
\[
\otimes_\alpha L_\alpha^{n_\alpha} = \mathscr O_\CCC
\]
for every $(n_\alpha )\in \mathcal X^* ( U )$,
which means imposing $\rk \mathcal X^* ( U ) = n$ conditions. 

\begin{remark}
Since the fan $\Sigma$ is complete, the morphism $\gamma$ is already surjective when the index set is $ \PL ( \Sigma )^\vee_+ - \{ \bm 0 \}$.
\end{remark}

\begin{remark}
Two $\bm d $ and $\bm d '$ have same image by $\gamma^\vee$ if and only if there exists $\delta \in \Pic ( X )^\vee$ such that 
$\bm d ' = \bm d + \delta$. 
\end{remark}


\subsection{Motivic height zeta function}

\begin{definition}
    The motivic height zeta function is the series
    \[
     \zeta_H^\mathrm{mot} ( \mathbf T ) 
        =
    \sum_{\delta \in \Pic ( X_\Sigma )^\vee}
    \left [
    \HHom_k^{\delta} ( \mathscr C , X_\Sigma )_{U} 
    \right ]
    \mathbf T^{\delta} . 
    \]
\end{definition}
Using the dual exact sequence
\[
    0
\to 
    \Pic(X_{\Sigma})^{\vee}
\to 
    \mathrm{PL}(\Sigma)^{\vee}
\overset{\gamma^\vee}{\to} 
    \mathcal{X}_{*}(U)
\to 
    0
\]
of \eqref{equ:fundamental-exact-sequence-toric},
we can see any $\delta \in \Pic (X_\Sigma)^\vee$ as a tuple of integers $\bm d \in \PL ( \Sigma )^\vee$ via the identification $\PL ( \Sigma ) \simeq \oplus_{\alpha \in \Sigma ( 1 )} \mathbf Z \mathfrak D_\alpha$.  
Since the $\mathfrak{D}_\alpha$ are effective divisors, the image of $\degg ( f ) \in \Pic (X_\Sigma)^\vee$ lies in  $\PL ( \Sigma )_+^\vee$ for every $f : \CCC \to X$. 
More precisely, for every ${\bm d} \in \Pic (X_\Sigma)^\vee \cap \PL ( \Sigma )^\vee_+ $
there is a constructible morphism 
\[
 \HHom_k^{\bm d} ( \mathscr C , X_\Sigma )_{U} 
 \longrightarrow
 \Sym_{/k}^{\bm d} ( \CCC )
\]
sending $f$ to the tuple of divisors $(f^*\mathfrak{D}_\alpha)_{\alpha \in \Sigma ( 1 )}$. 
We can rewrite $ \zeta_H^\mathrm{mot} ( \mathbf T ) $ as 
\[
 \zeta_H^\mathrm{mot} ( \mathbf T ) 
 =
 \sum_{\bm d \in \PL ( \Sigma )_+^\vee \cap \Pic (X_\Sigma)^\vee}
 \sum_{
   \bm E \in \Sym_{/k}^{\bm d} ( \CCC ) }
 \left [
    f \in  \HHom_k^{\bm d} ( \mathscr C , X_\Sigma )_{U}  \mid 
    (f^*\mathfrak{D}_\alpha)_{\alpha \in \Sigma ( 1 )} = \bm E
 \right ] \mathbf T^{\bm d}. 
\]
Using the decomposition of $\mathcal X_* ( U )$ given by \eqref{equation:decomposition-of-cochar-cones-of-Sigma}, the previous expression can be refined once more thanks to the following remark: for every closed point $v\in \mathscr C$, if the $\Sigma ( 1 )$-tuple of local intersection degrees with the $\mathfrak{D}_\alpha$ is non-trivial,
then it belongs to the relative interior of a unique cone $\sigma \in \Sigma - \{ \bm 0 \}$.
Concretely, we compose with the constructible morphism $\gamma^\vee$ from \eqref{equation:gamma-vee-for-Sym-definition} 
to get 
\[
 \zeta_H^\mathrm{mot} ( \mathbf T ) 
 =
 \sum_{\bm d \in \PL ( \Sigma )_+^\vee \cap \Pic (X_\Sigma)^\vee} 
   \Sym_{/k}^{\gamma^\vee ( \bm d ) } ( \CCC ) 
  \left [
    \HHom_k^{\bm d} ( \mathscr C , X_\Sigma )_{U} \to \Sym_{/k}^{\gamma^\vee ( \bm d ) } ( \CCC )_* 
 \right ] 
 \mathbf T^{\bm d}.  
\]
Of course, since we are summing over $\bm d \in \PL ( \Sigma )_+^\vee \cap \Pic (X_\Sigma)^\vee$, we have $\gamma^\vee ( \bm d ) = \bm 0$ for every term appearing in the sum, but this will not be the case on the Fourier side. It may also be convenient to set $ \HHom_k^{\bm d} ( \mathscr C , X_\Sigma )_{U} := \varnothing $ whenever $\bm d \notin \Pic (X_\Sigma)^\vee$ 
(equivalently we could also have defined these moduli spaces starting from any $\bm d$). 
\begin{lemma}
The morphism $ \HHom_k^{\bm d} ( \mathscr C , X_\Sigma )_{U} \to \Sym_{/k}^{\gamma^\vee ( \bm d ) } ( \CCC )$ is a piecewise trivial fibration in $U$.  
\end{lemma}
\begin{proof}
For any extension $k ' / k $,
the fibre above a $k'$ point of $\Sym_{/k}^{\gamma^\vee ( \bm d ) } ( \CCC )$ it the set of $F'$ points of $U$, where $F' = k' ( \mathscr C \times_k k ')$, that is to say $n$-tuples of non-zero elements of $F'$ that share the same $n$-tuple of divisors of zeros and poles.
Such elements only differ by a $k'$-point of $U$. 
\end{proof}

\begin{definition}
    For any $( \pi_\sigma )_\sigma$ as before, 
    let 
    \[
    \mathbf 1 : 
\Sym^{(\pi_\sigma)_\sigma}_{/k} ( \CCC )_* \longrightarrow \Sym^{(\pi_\sigma)_\sigma}_{/k} ( \underline{\mathcal{X}_{*}(U)}_{\CCC} )_*
\]
be the morphism sending any
\[
\bm E = \sum_v \bm m_v [v]
\]
to 
\[
\bm E = \sum_v \bm m_v ( \bm m_v , [v] ) . 
\]
\end{definition}
Combining this with \eqref{eq:Abel-Jacobi-morphism-for-n-tuples-of-divisors}
and the summation operator from \cref{subsection:summation_rat_points}, we get the following lemma. 
We may sometimes write $\pi$ for $( \pi_\sigma )_\sigma$ to shorten our notations. 
\begin{lemma} \label{lemma:moduli-space-as-sum-on-principal-div}
For any $\bm d \in \PL ( \Sigma )_+^\vee$ 
the relation
\[ 
\left [
 \HHom_k^{\bm d} ( \mathscr C , X_\Sigma )_{U} 
\right ]
= 
[ U ]_k 
\sum_{ (\pi_\sigma)_\sigma \vdash \gamma^\vee ( \bm d ) }
\sum_{\PDiv^\pi} 
\left [
\mathbf 1 : 
\Sym^{(\pi_\sigma)_\sigma}_{/k} ( \CCC )_* \longrightarrow \Sym^{(\pi_\sigma)_\sigma}_{/k} ( \underline{\mathcal{X}_{*}(U)}_{\CCC} )_*
\right ]
\]
holds in $\mathscr M_{\Sym^{\gamma^\vee ( \bm d )}_{/k} ( \CCC )_*}$.
\end{lemma}

\begin{proof}
Recall that we also have a morphism 
\[
\Sym^{(\pi_\sigma)_\sigma}_{/k} ( \underline{\mathcal{X}_{*}(U)}_{\CCC} )_*
\longrightarrow
\Sym^{(\pi_\sigma)_\sigma}_{/k} ( \CCC )_*  
\] 
which can be composed with 
\[
\mathbf 1 : \Sym^{(\pi_\sigma)_\sigma}_{/k} ( \CCC )_* \longrightarrow \Sym^{(\pi_\sigma)_\sigma}_{/k} ( \underline{\mathcal{X}_{*}(U)}_{\CCC} )_*
\]
to get 
\[
\bm \varepsilon : \Sym^{(\pi_\sigma)_\sigma}_{/k} ( \underline{\mathcal{X}_{*}(U)}_{\CCC} )_* \to \Sym^{(\pi_\sigma)_\sigma}_{/k} ( \underline{\mathcal{X}_{*}(U)}_{\CCC} )_* .
\]
Now we can interpret $\Sym^{(\pi_\sigma)_\sigma}_{/k} ( \CCC )_* \longrightarrow \Sym^{(\pi_\sigma)_\sigma}_{/k} ( \underline{\mathcal{X}_{*}(U)}_{\CCC} )_*$ as the motivic function sending $\bm D \in \Sym^{(\pi_\sigma)_\sigma}_{/k} ( \underline{\mathcal{X}_{*}(U)}_{\CCC} )_*$ 
to one if $\bm \varepsilon ( \bm D ) = \bm D $ and zero otherwise. 
Relatively to a divisor $\bm E \in \Sym^{(\pi_\sigma)_\sigma}_{/k} ( \CCC )_*$, this gives the indicator function $\mathbf 1_{\bm E}$ of $\bm E$ on the set of divisors having support contained in the one of $\bm E$, hence the choice of notation:
\[
\mathbf 1 
= 
( \mathbf 1_{\bm E} )_{\bm E \in \Sym^{(\pi_\sigma)_\sigma}_{/k} ( \CCC )_*} . 
\]
Now 
relatively to $\bm E \in \Sym^{(\pi_\sigma)_\sigma}_{/k} ( \CCC )_*$ 
it makes sense to rewrite the fibre above $\bm 0$ of the Abel-Jacobi morphism \eqref{eq:Abel-Jacobi-morphism-for-n-tuples-of-divisors}
as 
\[
\sum_{\PDiv^\pi} \mathbf 1_{\bm E} 
\]
hence the lemma. 
\end{proof}


\subsection{Adelic height zeta function as a motivic Euler product}

Recall the notation for characteristic functions given in \cref{notation:characteristic_functions}.

\begin{definition}
    The motivic adelic height function is the motivic Euler product 
    \[
    H ( \cdot  , \mathbf T )
    =
    \prod_{v\in \mathscr C }
    \left ( 
    \sum_{\bm m_v \in \mathcal X_* ( U ) }
    \mathbf 1_{t^{\bm m_v} U ( \mathcal O )}
    \mathbf T^{\langle \scdot , \bm m_v  \rangle_{\Sigma}}
    \right ) \in \MMM_{\Sym^{\bullet}(\underline{\mathcal{X}_{*}(U)}_{\CCC})}.
    \]
    which, using the definition of the pairing $\langle \cdot,\cdot\rangle_{\Sigma}$ becomes 
    \[
    H ( \cdot  , \mathbf T ) = \left(1+ \sum_{\sigma\in \Sigma\setminus \{\mathbf{0}\}} \sum_{\mathbf{m}_v\in \sigma^{\circ}}\mathbf{1}_{t^\mathbf{m}_vU ( \mathcal O ) } \prod_{\alpha\in \sigma(1)}T_{\alpha}^{\langle\mathfrak{D}_{\alpha},\mathbf{m}_v\rangle_{\Sigma}} \right)
    \]
\end{definition}
 
It is really important for what comes next to understand $H ( \cdot , \mathbf T )$
as a collection of families of motivic functions, as in \cref{example:families-of-characteristic-functions}. 
We can then rewrite \cref{lemma:moduli-space-as-sum-on-principal-div} in a compact way as follows. 

\begin{proposition}
    The motivic height zeta function may be rewritten as 
    \begin{equation}
        \zeta_H^\mathrm{mot} ( \mathbf T ) 
        =
        [U]_k 
        \sum_{\PDiv ^n}
        H^\mathrm{mot} ( \cdot , \mathbf T ) . 
    \end{equation}
\end{proposition}


\subsection{Application of the Poisson formula}

Taking  the Fourier transform of the motivic Euler product $H ( \cdot  , \mathbf T )$ and applying \cref{example:fourier_of_char_function}, we obtain 
\[
( \four H ) ( \cdot , \mathbf T )
=
\prod_{v\in \mathscr C}
\left ( 
    \sum_{\bm m_v \in \mathcal X_* ( U ) }
    \ev (\bm m_v, \cdot )  
        \mathbf T^{\langle \scdot , \bm m_v  \rangle_{\Sigma}}
\right ) \in \CharM{\Sym^{\bullet}(\mathcal X_* (U ) )}{\Sym^{\bullet}(\CCC)}.
\]
in the sense of \cref{def:incidence-algebra-with-characters-localised}.

An application of \cref{thm:mot-mult-Poisson-formula}
to each of the coefficients of  $\zeta_H^{\mathrm{mot}} ( \mathbf T )$, which are
given by \cref{lemma:moduli-space-as-sum-on-principal-div}, 
provides the following compact expression for $\zeta_H^{\mathrm{mot}} ( \mathbf T )$. 

\begin{proposition}The multivariate motivic height zeta function 
$\zeta_H^{\mathrm{mot}} ( \mathbf T )$
can be rewritten
    \begin{equation}
\zeta_H^{\mathrm{mot}} ( \mathbf T )
=
\left (
    \mathbf L - 1 
\right )^n 
\int_{\PDiv^\perp}
( \four H ) ( \bm \chi , \mathbf T ) 
    \mathrm d \bm \chi 
    \in \MMM_{\Sym^{\bullet}(\CCC)}.
\end{equation} 

\end{proposition}


\subsection{The local factors of the Fourier transform}

\begin{definition}
    \label{def-QSigma}
    As in \cite[Definition 3.6]{batyrev-tschinkel1998manin-toric} 
    and \cite{bourqui2011MAMS}
    we define a polynomial 
    $Q_\Sigma ( \mathbf X )$ 
    with ${\mathbf X = ( X_\alpha )_{\alpha \in \Sigma ( 1 )}}$
    via the relation
    \[
    \sum_{\sigma \in \Sigma}
    \prod_{\alpha \in \sigma ( 1 )}
    \frac{X_\alpha}{1 - X_\alpha}
    =
    \frac{
    Q_{\Sigma} ( \mathbf X )
    }
    {
    \prod_{\alpha \in \Sigma ( 1 ) } ( 1 - X_\alpha ) 
    }.
    \]
\end{definition}

The following proposition is uniform in $v$,
hence it holds in our rings of varieties over $\CCC$.
It is the motivic analogue of \cite[Theorem 2.2.6]{batyrev-tschinkel1995anisotropic-toric}
quoted as Théorème 4.17 in \cite{bourqui2011MAMS}.
Remember that algebraic characters on
the group scheme $\ZZ_{\kappa ( v )}^n$
play the role of characters on the quotient $U ( F_v ) / U ( \mathcal O_v )$. 
Recall also that $\rho_\alpha \in \mathcal X_* ( T )$
is the generator of the ray given by $\alpha \in \Sigma ( 1 )$, so the evaluation by $\bm \chi_v$ makes sense.
\begin{proposition}
We have
    \[
( \four 
H_v ( \mathbf T_v ) ) 
( \bm \chi_v ) 
=
\frac{Q_\Sigma ( \bm \chi_v \TT )}{\prod_{\alpha \in \Sigma ( 1 ) } ( 1 - \bm \chi_v ( \rho_ \alpha ) T_\alpha )} 
\]
relatively to $\bm \chi_v \in \widehat{\mathcal X_* ( T )_{\kappa ( v )}} = \GG_{m,\kappa ( v )}^n$
and $v \in \mathscr C$. 
\end{proposition}

\begin{proof}
Indeed, by definition of the local Fourier transform,
\begin{align*}
( \four 
H_v ( \mathbf T_v ) ) 
( \bm \chi_v )
& =
\sum_{\bm m_v \in \mathcal X_* ( T )_{\kappa ( v )}}
H_v ( \mathbf T_v , \bm m_v )
\bm \chi_v ( \bm m_v ) \\
& = 
\sum_{\sigma \in \Sigma}
\sum_{(n_\alpha ) \in \ZZ_{>0}^{\sigma ( 1 )}}
H_v 
    \left ( \mathbf T_v , \sum_\alpha n_\alpha \rho_\alpha \right )
\prod_{\alpha \in \sigma ( 1 )} \bm \chi_v ( \rho_\alpha )^{n_\alpha} \\
& = 
\sum_{\sigma \in \Sigma}
\prod_{\alpha \in \sigma ( 1 )}
\frac{\bm \chi_v ( \rho_\alpha ) T_{v,\alpha}}{1 - \bm \chi_v ( \rho_\alpha ) T_{v,\alpha }}
\end{align*}
hence the claim, applying the definition of $Q_\Sigma$ to $X_\alpha = \bm \chi_v ( \rho_\alpha ) T_{v,\alpha}$. 
\end{proof}

\begin{proposition}The multivariate motivic height zeta function 
$\zeta_H^{\mathrm{mot}} ( \mathbf T )$
can be rewritten
    \begin{equation}
\zeta_H^{\mathrm{mot}} ( \mathbf T )
=
\left (
    \mathbf L - 1 
\right )^n 
\int_{\PDiv^\perp  }
 \prod_{v \in \CCC } \frac{Q_\Sigma ( \bm \chi_v \TT )}{\prod_{\alpha \in \Sigma ( 1 ) } ( 1 - \bm \chi_v ( \rho_ \alpha ) T_\alpha )} 
    \mathrm d \bm \chi  . 
\end{equation} 
\end{proposition}


\subsection{Splitting of $\PDiv^\perp$}

\label{subsection:splitting-abstract-PDiv}

We separate the part of characters depending on the degree from the part depending only on the degree-zero linear class. This manipulation is very classical but we have to explain how it fits into our motivic setting. 

Recall that we fixed once and for all a $k$-divisor $\mathfrak D_1 = \sum \mathfrak d_v [v] $ of degree one on $\mathscr C$. 
Any support $\mathsf S$ being given as in \cref{remark-application-to-Div_S(C)}, we can always enlarge it to include the support of $\mathfrak D_1$. 
Characters $\bm z : \bm m \mapsto \bm z^{\bm m} $ 
on $M = \mathcal X_* ( U ) \simeq \mathbf Z^n$
will be identified with 
\[
\bm m \mapsto \bm \chi ( \bm m \mathfrak D_1 ) = 
\left ( \prod_{v \in | \mathfrak D_1 |} \bm \chi_v  ( \mathfrak d_v )\right )^{\bm m}
\]
via 
\[
\bm z = \prod_{v \in | \mathfrak D_1 |} \bm \chi_v ( \mathfrak d_v  \cdot ) .
\]
We are able to write any divisor $D$ supported on $\mathsf S$ as
\[
D = 
( D - \deg ( D ) \mathfrak D_1 ) 
+
\deg ( D ) \mathfrak D_1 
\]
so that for any family of characters $\bm \chi $ we have
\[
\bm \chi ( D ) 
=
\bm \chi ( D - \deg ( D ) \mathfrak D_1 ) 
\bm z^{\deg ( D ) }. 
\]
In other words we perform the change of variables 
\[
\bm \chi \in \mathrm D ( \Div_\mathsf S ( \CCC ) ) 
\mapsto \left ( \underbrace{\bm \chi ( \scdot - \deg ( \scdot ) \mathfrak D_1  ))}_{=:\bm \chi ' } , \underbrace{\bm \chi ( \deg ( \scdot ) \mathfrak D_1  )}_{=: \bm z } \right ) 
\in \mathrm D ( \Div_\mathsf S^0 ( \CCC ) ) \times \mathrm D ( M ) .
\]
Moreover,
\cref{example:familie-of-char-is-just-a-torsor}
allows us to move $\bm z $ outside the motivic Euler product.

This way we are also able to use \cref{sect:local_integration_operator} and write:
\[
\int_{\PDiv^\perp  }
=
\int_{ D ( ( \Div ( \CCC ) / \Div^0 ( \CCC ) )^n }
\circ 
\int_{(\PDiv^0)^\perp  } . 
\]

\begin{lemma}
Under the previous splitting, 
\[
( \four 
H_v ( \mathbf T_v ) ) 
( \bm \chi_v ' , \bm z )
= 
\frac{Q_\Sigma ( \bm \chi_v ' \prod_\alpha \langle \bm z , \rho_\alpha \rangle T_\alpha )}{\prod_{\alpha \in \Sigma ( 1 ) } ( 1 - \bm \chi_v'  ( \rho_ \alpha ) \prod_\alpha \langle \bm z , \rho_\alpha \rangle T_\alpha )} .
\]
\end{lemma}

In the sequel we write $\bm \chi$ for $\bm \chi '$. The previous discussion leads to the following global statement. 

\begin{proposition}The multivariate motivic height zeta function 
$\zeta_H^{\mathrm{mot}} ( \mathbf T )$
can be rewritten
    \begin{equation}\label{equ:zeta_function_after_splitting}
\zeta_H^{\mathrm{mot}} ( \mathbf T )
=
\left (
    \mathbf L - 1 
\right )^n 
\int_{D ( \mathcal X_* ( U ) ) }
\int_{(\PDiv^0)^\perp}
 \prod_{v \in \CCC } \frac{Q_\Sigma ( \bm \chi_v \prod_\alpha \langle \bm z , \rho_\alpha \rangle T_\alpha )}{\prod_{\alpha \in \Sigma ( 1 ) } ( 1 - \bm \chi_v ( \rho_ \alpha ) \langle \bm z , \rho_\alpha \rangle T_\alpha )} 
    \mathrm d \bm \chi \mathrm d \bm z   . 
\end{equation} 
\end{proposition}


\section{Analysis of the integral over the Fourier domain}

\label{section:final-analysis}


\subsection{Motivic $\mathfrak L$-functions of free $\mathbf Z$-modules and indicator functions of cones}
We develop an elementary motivic analogue of the theory from \cite[Chap. 6]{bourqui2011MAMS}.
We keep most of Bourqui's notations here. 
Some of the proofs are very similar to the ones in the classical setting; some of them are included for the reader's convenience. 

Conventions are the following: in the whole section, $N$ is a free $\ZZ$-module, and we denote by 
\[\langle\cdot,\cdot\rangle:N^{\vee}\times N \to \ZZ\]
the pairing with $N^{\vee}$. For a cone $\Lambda\subset N_{\RR}$, we denote by $\overset{\circ}{\Lambda}$ its relative interior.  In what follows,~$\Upsilon$ will always denote a rational polyhedral cone in $N_{\RR}$ (i.e. generated by a finite number of elements of $N$). Such a cone may be written as the support of a regular fan. 

\subsubsection{$\mathfrak L$-function associated to a cone}

\begin{definition}
Let $S$ be a scheme, $N$ a free $\mathbf Z$-module of finite rank
and $N_S$ its associated constant group scheme over $S$. 

For every strictly convex rational polyhedral cone $\Upsilon$ in $N_\mathbf R$,
we define the 
$\mathfrak L$-function associated to $N$ relative to $S$  
to be the formal power series
\[
\mathfrak L_{N,\Upsilon}
( \mathbf T ) 
=
\sum_{y \in N \cap \Upsilon} \mathbf T^y 
\in \KVar S [[ N \cap \Upsilon ]] = \KVar S [[ \mathbf T ]] . 
\]
It is the formal power series associated to the motivic function $\bm 1_{N \cap \Upsilon} \in  \KVar {N_S}$. 

This definition naturally extends to all the variants and localisations of the Grothendieck ring of varieties
we use in the present work. 
In particular,
\[
\mathfrak L_{N,\Upsilon}
( \bm \chi \mathbf T ) 
=
\sum_{y \in N \cap \Upsilon} \bm \chi ( y ) \mathbf T^y 
\in \KCharVar{N}{S} [[ N \cap \Upsilon ]] . 
\]
More generally, 
if 
\[
\mathfrak a 
\in 
\KCharVar{N}{N_S}
\]
is a motivic function with characters on $N_S$
and $A$ is a subset of $N_\mathbf R$,
we set 
\[
\mathfrak L_{N, A , \mathfrak a } ( \mathbf T )
= 
\sum_{y \in N \cap A}
\mathfrak a ( y ) \mathbf T^y . 
\]
If $\mathfrak a$ is identically $1$,
we simply write $\mathfrak L_{N, A }
$.
\end{definition}

\begin{lemma}
    If $\lambda_0$ lies in the relative interior of $\Upsilon^\vee$,
    the series 
    \[
    \mathfrak L_{N,\Upsilon} \left ( T ^{\lambda_0} \right ) 
    =
    \sum_{y \in N \cap \Upsilon} T ^{ \langle \lambda_0 , y \rangle } 
    \]
    converges for $| T | < 1$ and 
    admits a pole at $T=1$ of order at most $\dim ( \Upsilon )$. 
    Moreover,
    if $\dim ( \Upsilon ) = \rk ( N )$,
    there exists a positive integer $a_{\lambda_0}$ such that 
    \[
    \left ( 
        ( 
        1 - T^{a_{\lambda_0}}
        )^{\rk ( N )}
    \mathfrak L_{N,\Upsilon} \left ( T ^{\lambda_0} \right )
    \right )_{T = 1}
    =
     a_{\lambda_0}^{\rk ( N ) }
    \mathfrak X_{N, \Upsilon^\vee } ( \lambda_0 ) \in \mathbf N
    \]
    where 
    \[
    \mathfrak X_{N, \Upsilon^\vee } ( \lambda )
    =
    \int_{\Upsilon} \exp (- \langle \lambda , y \rangle ) \mathrm d y
    \]
    with $\mathrm d y$ being the Lebesgue measure on $N^\vee_\mathbf R$ normalised by $N^\vee$. 
\end{lemma}

\begin{proof}
    We can write $\Upsilon$ as the support of a regular fan $\Delta$ to get 
    \[
    \mathfrak L_{N,\Upsilon} \left ( T ^{\lambda_0} \right ) 
        =
    \sum_{\delta \in \Delta} 
        \prod_{\ell \in \delta ( 1 )}    
            \left ( 
                \frac{1}{1 - T^{ \langle \lambda_0 , \rho_\ell \rangle }} 
                - 1 
            \right )
    \]
    where $\rho_\ell$ is the generator corresponding to $\ell \in \Delta ( 1 )$. 
    Furthermore, one can show 
    \cite[Remarque 5.3]{bourqui2011MAMS}
    that 
    \[
     \mathfrak X_{N, \Upsilon^\vee } ( \lambda_0 )
    =
     \sum_{
    \substack{
        \delta \in \Delta \\
        \dim ( \delta ) = \rk ( N ) 
    }
    }
        \prod_{\ell \in \delta ( 1 )}    
    \frac{1}{\langle \lambda_0 , \rho_\ell \rangle } .
    \]
    Therefore we take 
    \[
    a_{\lambda_0}
    = 
    \LCM_{\ell \in \delta ( 1 )} \langle \lambda_0 , \rho_\ell \rangle
    \]
    and check that 
    \[
       ( 
        1 - T^{a_{\lambda_0}} 
        )^{\rk ( N )}
    \sum_{
    \substack{
        \delta \in \Delta \\
        \dim ( \delta ) = \rk ( N ) 
    }
    }
        \prod_{\ell \in \delta ( 1 )}    
            \left ( 
                \frac{1}{1 - T^{\langle \lambda_0 , \rho_\ell \rangle} } 
                - 1 
            \right )
    \]
    takes the value 
    \[
           a_{\lambda_0}^{\rk ( N ) }
    \mathfrak X_{N, \Upsilon^\vee } ( \lambda_0 )  
    \]
    at $T = 1$.
    Moreover it is clear from the definitions that $ a_{\lambda_0}^{\rk ( N ) }
    \mathfrak X_{N, \Upsilon^\vee } ( \lambda_0 ) $ is a non-negative integer. 
\end{proof}

We assume that we are given 
an exact sequence of $\mathbf Z$-modules,
all free of finite rank,
\[
0 \longrightarrow M 
\overset{i}{\longrightarrow}
N
\overset{j}{\longrightarrow}
\Gamma 
\longrightarrow 
0 . 
\]
This exact sequence induces an exact sequence of dual modules
\[
0 
\longrightarrow 
\Gamma^\vee 
\overset{j^\vee}{\longrightarrow}
N^\vee 
\overset{i^\vee}{\longrightarrow}
M^\vee  
\longrightarrow 
0 
\]
and an exact sequence of algebraic tori
\[
0 \longrightarrow 
D ( \underline{\Gamma}_S )
\overset{D ( j_S )}{\longrightarrow}
D ( \underline N_S )  
\overset{D ( i_S )}{\longrightarrow}
D ( \underline M_S )  
\longrightarrow 
0 .
\]
Given a cone $\Lambda\subset N_{\mathbf{R}}$, this allows us to give a meaning to:
\[
 \mathfrak L_{N,\Lambda \cap M_\mathbf R} ( \mathbf T )
 =
  \sum_{y \in \Lambda \cap i(M) \cap N} \mathbf T^{ y }
= 
 \sum_{y \in \Lambda \cap M} \mathbf T^{i ( y )}
 =
 \sum_{y \in \Lambda \cap M} i^\vee ( \mathbf T^y ) 
 = 
 \mathfrak L_{M,\Lambda \cap M_\mathbf R}  ( i^\vee ( \mathbf T )  ) . 
\]

The following is a 
motivic version of 
Bourqui's ``lemme technique : forme jouet''
\cite[Lemme 6.14]{bourqui2011MAMS}.
\begin{lemma}
Let $S$ be a scheme.
We have 
    \[
    \int_{D ( \underline{\Gamma}_S )} 
    \mathfrak L_{N,\Lambda} ( D ( j_S ) ( \bm \chi ) \mathbf T ) \mathrm d \bm \chi 
    = 
    \mathfrak L_{N,\Lambda \cap i (  M_\mathbf R )} ( \mathbf T ) 
    \]
    in $\KVar S [[ N \cap \Upsilon ]]$. 
\end{lemma}

\begin{proof}
    The proof is completely formal and goes exactly as the one of \cite[Lemma 6.14]{bourqui2011MAMS}.
    We can directly compute the left hand side to get 
    \begin{align*}
       \int_{D ( \underline{\Gamma}_S )}
    \mathfrak L_{N,\Lambda} ( D ( j_S ) ( \bm \chi ) \mathbf T ) \mathrm d \bm \chi 
    & = 
    \int_{D ( \underline{\Gamma}_S )}
    \sum_{y \in \Lambda \cap N}
     \bm \chi (  j ( y ) ) \mathbf T^y \mathrm d \bm \chi \\
    & = 
    \sum_{y \in \Lambda \cap N}
      \mathbf T^y 
          \int_{D ( \underline{\Gamma}_S )}
    \bm \chi ( j ( y ) ) \mathrm d \bm \chi \\
    & = 
     \sum_{
     \substack{ y \in \Lambda \cap N \\ 
     j ( y ) = 0}
     }
    \mathbf T^y 
        \tag{by def. of $\int_{D ( \underline{\Gamma}_S )}$}
    \\
    & = \mathfrak L_{N,\Lambda \cap i (  M_\mathbf R )} ( \mathbf T )
    \end{align*}
    which is what we wanted. 
\end{proof}

\begin{notation}
From now on, we fix once and for all a basis 
$( \lambda_i )_{i\in I}$
of $N$
and take $\Lambda$ to be the cone generated by this basis. 
\end{notation} 

\begin{definition}
    An \emph{elementary admissible motivic function of non-negative multiplicity}
    is a motivic formal series 
    \[
    \mathfrak L_{N,\Lambda , \mathfrak a } ( \mathbf T )
        = 
    \sum_{y \in N \cap \Lambda}
    \mathfrak a ( y ) \mathbf T^y 
    \]
    converging for $| \mathbf T| < \mathbf L^{\varepsilon}$ for some $\varepsilon >0$. 

    Whenever $r$ is a non-negative integer, 
    an \emph{elementary admissible motivic function of multiplicity at least $-r$} is 
    a motivic function $f$ on $\Lambda$ 
    admitting a decomposition
    \[
    f ( \mathbf T ) 
    =
    g ( \mathbf T ) 
    \mathfrak L_{N ', \overset{\circ}{\Upsilon} , \mathfrak a } ( \mathbf T )
    \]
    where $g$ is an elementary admissible function of non-negative multiplicity,
    $\Upsilon$ is a rational polyhedral subcone of $N_\mathbf R$ of dimension at most $r$
    and $N'$ is a subgroup of $N$. 

    \emph{Admissible motivic functions of multiplicity at least $-r$}
    are finite linear combinations of 
    such elementary functions. 
\end{definition}

\subsubsection{Decomposition lemmas and bounds on degrees}
Recall that $\Lambda\subset N_{\RR}$ is the simplicial cone generated by $(\lambda_i)_{i\in I}.$ Let $( \lambda_i^\vee )_{i\in I}$ be the basis of $N^\vee$
induced by $( \lambda_i )_{i\in I}$
and define ${\lambda^\vee = \sum_{i\in I} \lambda_i^\vee}$. We also fix a rational polyhedral cone $\Upsilon\subset N_{\RR}$ contained in $\Lambda$, and write $\Upsilon$ as the support of a regular fan $\Delta$. 
Following Bourqui \cite[\S 6.4.1]{bourqui2011MAMS}, we define for any subset $K\subset I$, 
any cone $\delta \in \Delta$ and any element $z\in \Lambda \cap N$,
\[
\delta ( K , z )
=
\{ 
y \in \overset{\circ}{\delta} 
\mid 
\forall \, i \in K ,  \quad 
\langle \lambda_i^\vee , y \rangle < \langle \lambda_i^\vee , z \rangle 
\}  
\]
as well as 
\[
\delta ( 1 )_K 
=
\{
\ell \in \delta ( 1 ) 
\mid 
\forall \, i \in K , \quad 
\langle \lambda_i^\vee , \rho_\ell \rangle = 0
\}.
\]
We denote by $\delta_K$ the cone spanned by rays in $\delta ( 1 )_K$
and by $\delta^K$ the cone spanned by rays of $\delta$ not in $\delta ( 1 )_K$. 

The following lemma is an immediate extension
of \cite[Lemme 6.5]{bourqui2011MAMS}
to our setting.

\begin{lemma}
\label{lemma:decomposition-of-L-function-cone}
For any $z \in \Lambda \cap N$
we have 
\begin{equation}
   \mathfrak L_{N , \Upsilon \cap ( z + \Lambda )} ( \mathbf T ) 
   =
   \sum_{\substack{
        \delta \in \Delta \\
        K \subset I }
    }
    ( - 1 )^{| K |}
    \mathfrak L_{N , \delta ( K , z )} ( \mathbf T )  . 
\end{equation} 
\end{lemma}

We will also reemploy the following (see \cite[p.~123]{bourqui2011MAMS} for a proof).

\begin{lemma}[{\cite[Lemme 6.5]{bourqui2011MAMS}}]
\label{lemma:cones-decomposition-and-bounds}
    Let $\delta $ be a cone of $\Delta$ and $K \subset I$. 
    \begin{enumerate}
        \item There is a decomposition 
        \[
            \delta ( K , z ) = 
                \left ( 
                    \overset{\circ}{\delta_K} \cap N
                \right )
                \oplus 
                \left ( 
                    \delta^K ( K , z ) \cap N 
                \right ) .
        \]
        \item The set $  \delta^K ( K , z ) \cap N $ is finite, of cardinality bounded by
        \[
            \langle z , \lambda^\vee \rangle^{\rk ( N )} .
        \]
        \item Moreover, for any $y \in \delta^K ( K , z ) \cap N$ we have 
        \[
        \langle y , \lambda^\vee \rangle 
            \leqslant
        | I | \cdot \langle z , \lambda^\vee \rangle 
        \cdot 
        \mathrm{supp}_{\ell \in \Delta ( 1 )} \langle \rho_\ell , \lambda^\vee \rangle .
        \]
        \item If $K$ is empty, then $\delta ( 1 )_K = \delta ( 1 )$
        and $\delta^K ( K , z ) \cap N = \{ 0 \}$. 
    \end{enumerate}
\end{lemma}

\begin{lemma}[{\cite[Lemme 6.6]{bourqui2011MAMS}}]
\label{lemma:proper-subcones}
Let $\delta $ be a cone of $\Delta$. Assume that
\begin{enumerate}
    \item $\delta$ has maximal dimension,
    \item $K$ is a non-empty subset of $I$,
    \item and $( N / \langle \Upsilon \rangle )^\vee \cap \lambda^\vee = \{ 0 \}$.
\end{enumerate}
Then $\delta ( 1 )_K$ is a \emph{proper} subset of $\delta ( 1 )$. 
\end{lemma}

\subsubsection{Convergence lemmas}

\begin{lemma}
    \label{lemma-control-L-functions-by-integration}
    Let $\mathfrak a$ be a motivic function on $\Lambda \cap N$
    relatively to a scheme $S$
    and $\varepsilon >0$ be 
    such that 
    $\mathfrak L_{N, \Lambda , \mathfrak a }$
    converges for $| \mathbf T| < \mathbf L^\varepsilon$. 
    In particular, 
    it implies that 
    the sum 
    \[
    \mathfrak L_{N,\Lambda , \mathfrak a } ( \bm 1 )
    =
    \sum_{y \in \Lambda \cap N} 
    \mathfrak a ( y ) 
    \]
    is well-defined
    in the dimensional completion. 
    
    Let 
    \[
    f_1 ( \mathbf T ) 
    = 
    \int_{D ( \underline{\Gamma}_S )}
    \mathfrak L_{N, \Lambda , \mathfrak a} 
    ( \hat j ( \bm \chi ) \mathbf T )
    \mathfrak L_{N , \Lambda} 
    ( \hat j ( \bm \chi ) \mathbf T )
    \mathrm d \bm \chi  
    \in \KVar S [[ \mathbf T ]] 
    . 
    \]
    This motivic formal series
    converges for $| \mathbf T | < 1$. 
    
    Assume that $\Gamma^\vee \cap \Lambda^\vee = \{ 0 \} $. 
    Then the series
    \[
        f_1 ( \mathbf T ) 
    -   
        \mathfrak L_{N,\Lambda , \mathfrak a } ( \bm 1 )
        \mathfrak L_{N, \Lambda \cap M_\mathbf R} ( \mathbf T ) 
    \]
    is an admissible motivic function of multiplicity at least 
    $- \rk ( M ) + 1 $. 
\end{lemma}

\begin{proof}
    First, 
    one shows that 
    \[
    f_1 ( \mathbf T )
    =
    \sum_{y_1 \in \Lambda \cap N}
    \mathfrak a ( y_1 ) 
    \mathfrak L_{N, \Lambda \cap M \cap ( y_1 + \Lambda )} ( \mathbf T ) .
    \]
    Indeed, exactly as in the proof of 
    \cite[Proposition 6.15]{bourqui2011MAMS}, we can write 
    \begin{align*}
        f_1 ( \mathbf T ) 
        & =
        \int_{D ( \underline{\Gamma}_S )}
     \mathfrak L_{N, \Lambda , \mathfrak a} 
     ( \hat j ( \bm \chi ) \mathbf T )
    \mathfrak L_{N , \Lambda} 
    ( \hat j ( \bm \chi ) \mathbf T )
    \mathrm d \bm \chi  
        \\
        & = 
        \sum_{(y_0,y_1)\in ( \Lambda \cap N )^2} 
        \mathfrak a ( y_1 )
        \mathbf T^{y_0 + y_1} 
        \int_{D ( \underline{\Gamma}_S )}
            \bm \chi ( j ( y_0 + y_1 ) )
            \mathrm d \bm \chi 
        \\
        & = 
        \sum_{y_1\in ( \Lambda \cap N )} 
        \mathfrak a ( y_1 )
        \sum_{
        \substack{
        y_0 \in \Lambda \cap N \\ j ( y_0 + y_1 ) = 0 }
        } 
        \mathbf T^{y_0 + y_1} \\
        & = 
        \sum_{y_1\in ( \Lambda \cap N )} 
        \mathfrak a ( y_1 )
        \sum_{
        \substack{
        y \in  \Lambda \cap M \\ y \in y_1 + \Lambda \cap N  }
        } 
        \mathbf T^{y} \\
        & =  \sum_{y_1 \in \Lambda \cap N}
         \mathfrak a ( y_1 ) 
    \mathfrak L_{N, \Lambda \cap M \cap ( y_1 + \Lambda )} ( \mathbf T ) . 
    \end{align*}
    One can write $\Upsilon = \Lambda \cap M_\mathbf R$
    as the support of a regular fan $\Delta$
    and use \cref{lemma:decomposition-of-L-function-cone}
    to get 
    \[
    f_1 ( \mathbf T ) 
    =
    \sum_{\delta \in \Delta}
    \sum_{J \subset I}
    ( - 1 )^{|J|}
    \mathfrak L_{N , \overset{\circ}{\delta_J}} ( \mathbf T )
    \sum_{y_1\in \Lambda \cap N}
    \mathfrak a ( y_1 ) 
    \mathfrak L_{N, \delta^J ( J , y_1 )} ( \mathbf T ) . 
    \]
    The main term of this series is given by $J = \varnothing$
    and equals 
    \[
    \mathfrak L_{N , \Lambda , \mathfrak a} ( \bm 1 ) 
    \mathfrak L_{M , \Lambda \cap M_\mathbf R} ( \mathbf T ) . 
    \]

    If $\delta$ and $J \neq 0$ are given,
    we know by \cref{lemma:cones-decomposition-and-bounds}
    that 
    for any $y_1\in \Lambda \cap N$
    there are at most 
    $ \langle y_1 , \lambda^\vee \rangle^{ \rk ( N ) }$
    elements in 
    $\delta^J ( J , y_1 )$
    and that for any such element $y \in \delta^J ( J , y_1 )$,
    the scalar product $\langle y , \lambda^\vee \rangle$
    is bounded by $
        | I | \cdot \langle y_1 , \lambda^\vee \rangle 
             \cdot 
        \supp_{\ell \in \Delta ( 1 )} \langle \rho_\ell , \lambda^\vee \rangle 
    $. 
    It means in particular that the series 
    \[
        \sum_{y_1 \in \Lambda \cap N}
        a_{y_1} \mathfrak L_{N , \delta^J ( J , y_1 )} ( \mathbf T )
    \]
    is an elementary admissible function of positive multiplicity. Finally,
    since we assume that $\Gamma^\vee \cap \Lambda^\vee = \{ 0 \}$,
    \cref{lemma:proper-subcones}
    ensures that $\delta ( 1 )_J$ is a proper subset of 
    $\delta ( 1 )$.
    This implies that poles of $\mathfrak L_{N , \overset{\circ}{\delta_K}} ( T )$
    are controlled up to order $- ( \rk ( M ) - 1)$
    and finally that 
    \[
        \mathfrak L_{N , \overset{\circ}{\delta_J}} ( \mathbf T )
    \sum_{y_1\in \Lambda \cap N}
    \mathfrak a ( y_1 ) 
    \mathfrak L_{N, \delta^J ( J , y_1 )} ( \mathbf T )
    \]
    is admissible of multiplicity at least $- ( \rk ( M ) - 1)$. 
\end{proof}

Since the order of the pole of 
$ \mathfrak L_{M , \Lambda \cap M_\mathbf R} ( \mathbf T ) $
is exactly $\rk ( M )$, we get:

\begin{cor}
    \label{cor-control-L-functions-by-integration--residue}
    Let $\lambda_0$ be in the interior of $\Lambda^\vee$. Then 
    \[
     \left ( 
     ( 
        1 - T^{a_{\lambda_0}} 
    )^{\rk ( M )}
    f_1 \left ( T^{\lambda_0} \right ) 
    \right )_{T = 1} 
    = 
    a_{\lambda_0}^{\rk ( M )}
    \mathfrak L_{N,\Lambda , \mathfrak a } ( \bm 1 )
    \mathfrak X_{M^\vee , i^\vee ( \Lambda^\vee ) } ( i^\vee ( \lambda_0 ) ) . 
    \]
\end{cor}


\subsection{Application to the motivic height zeta function}
In our application, 
the short exact sequence 
\[
0 \longrightarrow M 
\overset{i}{\longrightarrow}
N
\overset{j}{\longrightarrow}
\Gamma 
\longrightarrow 
0  
\]
is given by the short exact sequence of $\ZZ$-modules
\[ 
0 \longrightarrow\Pic ( X_\Sigma )^{\vee} \overset{\pi^{\vee}}{\longrightarrow} \PL ( \Sigma )^{\vee} \overset{\gamma^{\vee}}{\longrightarrow} \mathcal X_*(U)\longrightarrow 0 
\]
dual to \eqref{equ:fundamental-exact-sequence-toric},
and $\Lambda = ( \PL ( \Sigma )^\vee_+)_\mathbf R$.  
We also get for every scheme $S$ an exact sequence of diagonalisable algebraic groups (here split tori)
\[
1
\longrightarrow 
D ( \underline{\mathcal X_*(U)}_S ) 
\overset{D ( \gamma^\vee ) }{\longrightarrow} 
D ( \underline{ \PL ( \Sigma )^\vee}_S ) 
\overset{D ( \pi ^\vee )}{\longrightarrow} 
D ( \underline{\Pic ( X_\Sigma )^\vee}_S ) 
\longrightarrow 
1 .
\] Our cone $\Lambda\subset N_\mathbf R$ will be the cone generated by the dual basis $\mathfrak{D}_{\alpha}^{\vee}$, i.e. $\Lambda = \CEff(X_{\Sigma})^{\vee}$.

Recall  from (\ref{equ:zeta_function_after_splitting}) that we have the expression 
  \[
\zeta_H^{\mathrm{mot}} ( \mathbf T )
=
\left (
    \mathbf L - 1 
\right )^n 
\int_{D ( \mathcal X_* ( U ) ) }
\int_{(\PDiv^0 ( \mathscr C ) ^n)^\perp  }
 \prod_{v \in \CCC } 
    \frac{Q_\Sigma ( ( \bm \chi_v ( \rho_\alpha ) \langle \bm z , \rho_\alpha \rangle T_\alpha )_{\alpha \in \Sigma ( 1 )} ) }{\prod_{\alpha \in \Sigma ( 1 ) } ( 1 - \bm \chi_v ( \rho_ \alpha ) \langle \bm z , \rho_\alpha \rangle T_\alpha )} 
    \mathrm d \bm \chi \mathrm d \bm z   
\]
where we write $\bm \chi_v ( \rho_\alpha )$ for the fibre above $v\in \CCC$ of the algebraic cocharacter corresponding to~$\rho_\alpha$, 
that is to say
\[
\bm \chi ( \rho_\alpha ) 
=
\ev ( \rho_\alpha , \scdot )
= 
\left [
    \bm \chi \in 
    D ( \underline{\mathcal X_* ( U )}_\CCC ) 
    \mapsto 
    \bm \chi( \rho_\alpha ) \in \mathbf G_m 
\right ]
\in \KCharVar{\mathcal X_* ( U )}{\CCC}
,
\]
to which the motivic Euler product is applied,
and similarly $\langle \bm z , \rho_\alpha \rangle $ for the motivic class
of the algebraic cocharacter 
\[
\langle \bm z , \rho_\alpha \rangle 
=
\ev ( \rho_\alpha , \scdot ) 
=
\left [
    \bm z \in D ( \underline{\mathcal X_* ( U )} ) \simeq D ( \underline{ ( \mathbf Z \mathfrak D_1 )^n } )
    \mapsto
    \bm z ( \rho_\alpha )
\right ]
\in \KCharVar{\mathcal X_* ( U )}{k} .
\]
Using multiplicativity of motivic Euler products, we have 
\begin{align*}
   &  \prod_{v \in \CCC } 
    \frac{Q_\Sigma ( ( \bm \chi_v ( \rho_\alpha ) \langle \bm z , \rho_\alpha \rangle T_\alpha )_{\alpha \in \Sigma ( 1 )} )}{\prod_{\alpha \in \Sigma ( 1 ) } ( 1 - \bm \chi_v ( \rho_ \alpha ) \prod_\alpha \langle \bm z , \rho_\alpha \rangle T_\alpha )} \\
& =
\left ( \prod_{\alpha \in \Sigma ( 1 ) } L ( \bm \chi ( \rho_\alpha ) \langle \bm z , \rho_\alpha \rangle  , T_\alpha ) \right )
\prod_{v\in \CCC} Q_\Sigma ( (  \bm \chi_v ( \rho_\alpha ) \langle \bm z , \rho_\alpha \rangle T_{\alpha} )_{\alpha \in \Sigma ( 1 )} ) 
\end{align*}
which is our motivic analogue of expression (4.3.22) in \cite[Lemme 4.44]{bourqui2011MAMS}.

To apply the results from the previous section,
let us formally define $P ( \mathbf T \langle \bm z , \rho \rangle )$ as 
\[
\left ( 
\prod_{\alpha \in \Sigma ( 1 )}
( 1 - T_\alpha )
\right )
\int_{( \PDiv^0 ( \CCC )^n ) ^\perp}
\left ( \prod_{\alpha \in \Sigma ( 1 ) } L ( \bm \chi ( \rho_\alpha ) \langle \bm z , \rho_\alpha \rangle  , \mathbf L^{-1} T_\alpha ) \right )
\left ( 
\prod_{v\in \mathscr C}
Q_{\Sigma} 
( \bm \chi_v ( \rho_\alpha ) \langle \bm z , \rho_\alpha \rangle \mathbf L^{-1} T_{\alpha , v} )
\right )
\mathrm d \bm \chi 
\]
where we normalise by $\LL$ each indeterminate $T_\alpha$
to later apply \cref{lemma-control-L-functions-by-integration} more easily, 
so that 
\[
\zeta_H^{\mathrm{mot}} ( \mathbf T )
=
( \mathbf L - 1 )^n
\int_{D ( \underline{\mathcal X_* ( U )}_k ) } 
\mathfrak L_{ \mathbf Z^{\Sigma ( 1 )} , \CEff ( X )^\vee } ( ( \langle \bm z , \rho_\alpha \rangle T_\alpha )_{\alpha \in \Sigma ( 1 ) } )  
    P (  ( \langle \bm z , \rho_\alpha \rangle T_{\alpha} )_{\alpha \in \Sigma ( 1 ) } )  
\mathrm d \bm z  
\]
where we 
recall that 
    \[
\mathfrak L_{ \mathbf Z^{\Sigma ( 1 )} , \CEff ( X )^\vee } ( \mathbf T ) 
=
\frac{1}{\prod_{\alpha \in \Sigma ( 1 )} ( 1 - T_\alpha )}  . 
\]

So far we have been working with motivic Euler product defined by (mixed) symmetric products taking (the subdivision \eqref{equation:decomposition-of-cochar-cones-of-Sigma} of) 
$\mathcal X_* ( U ) - \{ \bm 0 \}$ as a set of index.
It is now time to replace them by a motivic Euler product 
with coefficients indexed by $\mathbf N^{\Sigma ( 1 )}$, which are better understood and reappear naturally. 
Thanks to the discussion in \cref{sect:local_integration_operator} applied to the situation of \cref{subsection:splitting-abstract-PDiv}, we can compare the effect of
$\int_{
    ( \PDiv ( \CCC ) ^n )^\perp
    }$
    and 
$\int_{
    ( \PDiv^0 ( \CCC ) ^n )^\perp
    }$ above distinct 
    points of 
    \[
    \coprod_{(\pi_\sigma)_\sigma} \Sym_{/k}^{(\pi_\sigma)_\sigma} ( \CCC )_* 
    \]
    which may not share the same support. 
In particular, one can easily show that the output only depends on the linear class in the group scheme
\[
\HHom_{k-\text{gp}}  ( \underline{\mathcal X^* ( U )}_k , \PPic ( \mathscr C ) ) \simeq \PPic ( \CCC )^n 
\]
by the Abel-Jacobi morphism \eqref{eq:Abel-Jacobi-morphism-for-n-tuples-of-divisors}. 
Moreover, this remark is fully compatible with both the algebraic splitting $\PPic ( \mathscr C ) \simeq \PPic^0 ( \mathscr C ) \oplus \underline{\mathbf Z}_k$
and its abstract counterpart $\Div ( \CCC ) \simeq \Div^0 ( \CCC ) \oplus \mathbf Z \mathfrak D_1 $ that is used fibrewise in \cref{subsection:splitting-abstract-PDiv}. 

Using the definitions from \cref{subsect:multiheights}, 
we have for any $\bm d \in \mathbf N^{\Sigma ( 1 )}$ a commutative diagram 
\[
\begin{tikzcd}
\Sym_{/k}^{\bm d} ( \CCC )_* \rar["{\gamma^\vee}"] \dar & \Sym^{\gamma^\vee ( \bm d )} ( \CCC )_* \dar \\
\HHom_{k-\text{gp}}  ( \underline{\PL ( \Sigma ) }_k , \PPic ( \mathscr C ) )
\rar["{\gamma^\vee}"] &
\HHom_{k-\text{gp}}  ( \underline{\mathcal X^* ( U )}_k , \PPic ( \mathscr C ) )
\end{tikzcd}
\]
and we already described the fibre of the bottom morphism \eqref{equation:gamma-vee-for-Pic-definition} : it is isomorphic to $r$ copies of $\PPic^0 ( \CCC )$. 

Putting all this together, we obtain the following lemma. Note that the class of $\PPic^0 ( \CCC )$ is invertible in $\widehat{\mathscr M_k}^{\dim}$,
since it can be obtained as a special value of $\frac{\prod_{v\in \CCC} ( 1 - T_v )}{(1-T)(1-\mathbf LT)}$.  

\begin{lemma}
\label{lemma:partial-computation-of-P(T)}
For all but finitely many $\bm d  = ( d_\alpha) \in \PL ( \Sigma )^{\vee}_+$,
the coefficient of $\mathbf T^{\langle \cdot, \gamma^\vee ( \bm d ) \rangle }$ in 
\[
\int_{
    ( \PDiv^0 ( \CCC ) ^n )^\perp
    }
\left ( \prod_{\alpha \in \Sigma ( 1 ) } L ( \bm \chi ( \rho_\alpha ) \langle \bm z , \rho_\alpha \rangle  , \mathbf L^{-1} T_\alpha ) \right )
\left ( 
\prod_{v\in \mathscr C}
Q_{\Sigma} 
( \bm \chi_v ( \rho_\alpha ) \langle \bm z , \rho_\alpha \rangle \mathbf L^{-1} T_{\alpha , v} )
\right )
\mathrm d \bm \chi 
\]
equals the coefficient of $\mathbf T^{\bm d}$ in 
\[
\frac{1}{[\PPic^0 ( \CCC )]^n}  
\left ( \prod_{\alpha \in \Sigma ( 1 ) } Z^\Kapr_\CCC ( \langle \bm z , \rho_\alpha \rangle   \mathbf L^{-1} T_\alpha ) \right )
\left ( 
\prod_{v\in \mathscr C}
Q_{\Sigma} 
\right )
( ( \langle \bm z , \rho_\alpha \rangle \mathbf L^{-1} T_\alpha )_\alpha )
\in \CharM{\mathcal X_* ( U )}{k} [[ ( T_\alpha )_\alpha ]]
\]
where in the second expression the motivic Euler products are taken with $\mathbf N^{\Sigma ( 1 )} \setminus \{ \bm 0 \}$ as index set. 
\end{lemma}

\begin{remark}
Note that the relation 
\begin{equation}
\label{equation:def-of-Q-Sigma}
\sum_{\sigma \in \Sigma}
    \prod_{\alpha \in \sigma ( 1 )}
    \frac{X_\alpha}{1 - X_\alpha}
    =
    \frac{
    Q_{\Sigma} ( \mathbf X )
    }
    {
    \prod_{\alpha \in \Sigma ( 1 ) } ( 1 - X_\alpha ) 
    } 
\end{equation}
forces the coefficients of this latter motivic Euler product to be actually mixed symmetric product with respect to the collection of subsets $I_\sigma \subset \mathbf N^{\Sigma ( 1 )}$. 
\end{remark}

\begin{lemma} \label{lemma:convergence-Euler-product-Q-Sigma}
The motivic Euler product
\[ \prod_{v\in \CCC}Q_{\Sigma}((T_{\alpha})_{\alpha\in \Sigma(1)}) \]
converges for $|\mathbf{T}| < \LL^{-\frac{1}{2}}.$

Moreover,
its value at $T_\alpha = \mathbf L^{-1}$ is given by the convergent motivic Euler product
\[
\prod_{v\in \mathscr C}
\left (
( 1 - \mathbf L^{-1} )^{\rg ( \Pic ( X_\Sigma ))}
\frac{\left [X_\Sigma \right ]}{\mathbf L^{\dim ( X_\Sigma )} }
\right ) 
\in \widehat{\mathscr M_k}^{\dim}. 
\]
\end{lemma}

\begin{proof}
Recall that $Q_\Sigma$
was defined 
by \eqref{equation:def-of-Q-Sigma}. 
By \cite[Proposition 2.2.3]{batyrev-tschinkel1995anisotropic-toric}, the polynomial $Q_\Sigma -1$ contains only monomials of degree $\geq 2$, from which we can deduce the convergence statement. 
Moreover, 
\begin{align*}
 Q_{\Sigma} ( \mathbf L^{-1} )
 & =
 \sum_{\sigma \in \Sigma}
  \prod_{\alpha \in \sigma ( 1 )} 
  \mathbf L^{-1}
   \prod_{\alpha \notin \sigma ( 1 )}
   ( 1 - \mathbf L^{-1} )\\
& =
\mathbf L^{-| \Sigma ( 1 )|}
 \sum_{\sigma \in \Sigma}
 ( \mathbf L - 1 )^{| \Sigma ( 1 ) | - | \sigma ( 1 ) |} \\
& =
\mathbf L^{-\rg ( \Pic ( X_\Sigma ) ) - \dim ( X_\Sigma ) }
( \mathbf L - 1 )^{\rg ( \Pic ( X_\Sigma )}
\underbrace{\sum_{\sigma \in \Sigma}
( \mathbf L - 1 )^{\dim ( X_\Sigma ) - | \sigma ( 1 ) |}}_{=[X_\Sigma]}
\\ 
& =
\left ( 1 - \mathbf L^{-1} \right )^{\rg ( \Pic ( X_\Sigma ))}
\frac{\left [X_\Sigma \right ]}{\mathbf L^{\dim ( X_\Sigma )} } 
\end{align*}
where we used the relation 
$| \Sigma ( 1 ) | = \rg ( \Pic ( X_\Sigma ) ) + \dim ( X_\Sigma )$.
\end{proof}

By \cref{lemma:convergence-Euler-product-Q-Sigma},
we know that $\prod_{v\in \CCC}Q_{\Sigma}(( \mathbf L^{-1} T_{\alpha})_{\alpha\in \Sigma(1)}) $
converges for $|\mathbf{T}| < \LL^{\frac{1}{2}}$. 
It is also well-known that $(1-T)(1- \mathbf L T ) Z^\Kapr_\CCC ( \mathbf T ) $ is a polynomial,
therefore $(1-T) Z^\Kapr_\CCC ( \mathbf L^{-1} T ) $ converges for $| T | < \mathbf L$. 
Together with
\cref{lemma:partial-computation-of-P(T)},
these facts ensures that $\mathfrak L_{N,\Lambda , \mathfrak a } = P $ satisfies the assumptions of \cref{lemma-control-L-functions-by-integration} for all $0 < \varepsilon \leqslant \frac{1}{2}$. 
In particular, thanks to \cref{cor-control-L-functions-by-integration--residue} we get the following,
where $\alpha^* ( X_\Sigma )$ is the value of $ \mathfrak X_{N, \Upsilon^\vee } ( \lambda_0 )  $ in our setting. 

\begin{theorem}
    \label{theorem:final-residue-type-result}
    There exists an $\eta > 0$ such that the series 
    \[
    ( 
        1 - ( \mathbf L T )^{a_{\lambda_0}} 
    )^{\rk ( \Pic ( X_\Sigma )}
    \zeta_H^{\mathrm{mot}}  \left ( T ^{\omega_{X_\Sigma}^{-1}} \right ) 
    \]
    converges for $| T | < \mathbf L^{-1 + \eta}$.
    Moreover, 
    \[
     \left ( 
     ( 
        1 - ( \mathbf L T )^{a_{\lambda_0}} 
    )^{\rk ( \Pic (X_\Sigma) )}
    \zeta_H^{\mathrm{mot}}  \left ( T ^{\omega_{X_\Sigma}^{-1}} \right ) 
    \right )_{T = \mathbf L^{-1}} 
    =
    a_{\lambda_0}^{\rk ( \Pic ( X_\Sigma ) )}
    \alpha^* ( X_\Sigma ) 
    \gamma^{\mathrm{mot}}_H ( X_\Sigma ) 
    \]
    in $\widehat{\mathscr M_k}^{\dim}$
where 
\[
 \gamma^{\mathrm{mot}}_H ( X_\Sigma )
 =
 \mathbf L^{n(1-g)} 
 \left ( 
 \frac{[\PPic^0 ( \CCC ) ] \mathbf L^{-g}}{1 - \mathbf L^{-1}} 
 \right )^{\rk ( \Pic ( X_\Sigma ) )}
\prod_{v\in \mathscr C}
\left (
( 1 - \mathbf L^{-1} )^{\rg ( \Pic ( X_\Sigma ))}
\frac{\left [X_\Sigma \right ]}{\mathbf L^{\dim ( X_\Sigma )} }
\right ) 
\in \widehat{\mathscr M_k}^{\dim}. 
\]
\end{theorem}
To deduce \cref{thm-intro-stabilisation-multi-height} from the previous discussion, it suffices to apply \cref{lemma-from-convergent-series-to-stabilisation-of-coeffs}
to 
\[
F ( \mathbf T )
=
\frac{1}{[\PPic^0 ( \CCC )]^n}  
\left ( \prod_{\alpha \in \Sigma ( 1 ) } ( 1 - \mathbf L T_\alpha ) Z^\Kapr_\CCC (  T_\alpha ) \right )
\left ( 
\prod_{v\in \mathscr C}
Q_{\Sigma} (  \mathbf T )
\right ) 
\]
with $\bm \rho = ( 1 , ... , 1 )$.


\bibliography{references}

@article{batyrev-tschinkel1995anisotropic-toric,
  title={Rational points of bounded height on compactifications of anisotropic tori},
  author={Batyrev, Victor V and Tschinkel, Yuri},
  journal={International Mathematics Research Notices},
  volume={1995},
  number={12},
  pages={591--635},
  year={1995},
  publisher={Oxford University Press}
}

@article{batyrev-tschinkel1998manin-toric,
  title={Manin's conjecture for toric varieties},
  author={Batyrev, Victor V and Tschinkel, Yuri},
  journal={Journal of Algebraic Geometry},
  volume={7},
  number={1},
  pages={15--53},
  year={1998},
  publisher={American Mathematical Society}
}

@article{bilu2023MAMS,
	author = {Margaret Bilu},
	journal = {Memoirs of the American Mathematical Society},
	number = {1396},
	title = {Motivic {E}uler products and motivic height zeta functions},
	volume = {282},
	year = {2023},
    publisher={American Mathematical Society}}

@article{bourqui2011MAMS,
  title={Fonction z{\^e}ta des hauteurs des vari{\'e}t{\'e}s toriques non d{\'e}ploy{\'e}es},
  author={Bourqui, David},
  journal = {Memoirs of the American Mathematical Society},
  volume={211},
  number={994},
  year={2011},
  publisher={American Mathematical Society}
}

@article{chambert-loir-loeser2016motivic,
  title={Motivic height zeta functions},
  author={Chambert-Loir, Antoine and Loeser, Fran{\c{c}}ois},
  journal={American Journal of Mathematics},
  pages={1--59},
  year={2016},
  publisher={JSTOR}
}

@book{chambert-loir-nicaise-sebag2018motivic,
  title={Motivic integration},
  author={Chambert-Loir, Antoine and Nicaise, Johannes and Sebag, Julien},
  year={2018},
  publisher={Springer}
}

@article{cluckers-motivic-mellin,
  title={Motivic {M}ellin transforms},
  author={Cluckers, Raf and Loeser, Fran{\c{c}}ois and Nguyen, Kien Huu and Vermeulen, Floris},
  journal={arXiv preprint arXiv:2412.17764},
  year={2024}
}

@article {faisant-additive-groups,
    AUTHOR = {Faisant, Lo\"is},
     TITLE = {Geometric {B}atyrev-{M}anin-{P}eyre for equivariant
              compactifications of additive groups},
   JOURNAL = {Beitr. Algebra Geom.},
  FJOURNAL = {Beitr\"age zur Algebra und Geometrie. Contributions to Algebra
              and Geometry},
    VOLUME = {64},
      YEAR = {2023},
    NUMBER = {3},
     PAGES = {783--850},
}

@article{faisant2025motivic-distribution,
  title={Motivic distribution of rational curves and twisted products of toric varieties},
  author={Faisant, Lo{\"\i}s},
  journal={Algebra \& Number Theory},
  volume={19},
  number={5},
  pages={883--965},
  year={2025},
  publisher={Mathematical Sciences Publishers}
}

@BOOK{SGA1,
    AUTHOR = "Grothendieck, Alexander",
    TITLE = "Rev\^etements \'etales et groupe fondamental ({SGA} 1)",
    PUBLISHER = "Springer-Verlag",
    YEAR = "1971",
    SERIES = "Lecture notes in mathematics",
    VOLUME = "224"
}

@misc{SGA3-II,
  title={S{\'e}minaire de g{\'e}om{\'e}trie alg{\'e}brique du {B}ois {M}arie 1962-64. {S}ch{\'e}mas en groupes ({SGA} 3). {T}ome {II}: Groupes de Type Multiplicatif, et Structure des Schemas en Groupes Generaux },
  author={Demazure, M and Grothendieck, A and Artin, M and Bertin, JE and Gabriel, P and Raynaud, M and Serre, JP},
  year={1962--1964}
}

@article {hrushovski-kazhdan,
    AUTHOR = {Hrushovski, Ehud and Kazhdan, David},
     TITLE = {Motivic {P}oisson summation},
   JOURNAL = {Mosc. Math. J.},
  FJOURNAL = {Moscow Mathematical Journal},
    VOLUME = {9},
      YEAR = {2009},
    NUMBER = {3},
     PAGES = {569--623, back matter},
      ISSN = {1609-3321,1609-4514},
   MRCLASS = {11U09 (03C60 11R56 14E18)},
  MRNUMBER = {2562794},
MRREVIEWER = {H.\ Dugald\ Macpherson},
       DOI = {10.17323/1609-4514-2009-9-3-569-623},
       URL = {https://doi.org/10.17323/1609-4514-2009-9-3-569-623},
}

@article{bourqui2011quadriquesintrinseques,
  title={La conjecture de {M}anin g{\'e}om{\'e}trique pour une famille de quadriques intrins{\`e}ques},
  author={Bourqui, David},
  journal={manuscripta mathematica},
  volume={135},
  number={1},
  pages={1--41},
  year={2011},
  publisher={Springer}
}

@article{bourqui2009produit,
  title={Produit eul{\'e}rien motivique et courbes rationnelles sur les vari{\'e}t{\'e}s toriques},
  author={Bourqui, David},
  journal={Compositio Mathematica},
  volume={145},
  number={6},
  pages={1360--1400},
  year={2009},
  publisher={London Mathematical Society}
}

@article{cluckers-halupczok2022evaluation,
	author = {Cluckers, Raf and Halupczok, Immanuel},
	journal = {Advances in Mathematics},
	pages = {108635},
	publisher = {Elsevier},
	title = {Evaluation of motivic functions, non-nullity, and integrability in fibers},
	volume = {409},
	year = {2022}}

@article{bilu-howe2021motivic-statistics,
	author = {Bilu, Margaret and Howe, Sean},
	journal = {Algebra \& Number Theory},
	number = {9},
	pages = {2195--2259},
	publisher = {Mathematical Sciences Publishers},
	title = {Motivic {E}uler products in motivic statistics},
	volume = {15},
	year = {2021}}

@article{bilu-das-howe2022zeta,
  title={Zeta statistics and {H}adamard functions},
  author={Bilu, Margaret and Das, Ronno and Howe, Sean},
  journal={Advances in Mathematics},
  volume={407},
  pages={108556},
  year={2022},
  publisher={Elsevier}
}

@article{faisant2025universaltorsors,
	author = {Faisant, Lo{\"\i}s},
	journal = {arXiv preprint arXiv:2502.11704},
	title = {Motivic counting of rational curves with tangency conditions via universal torsors},
	year = {2025}}

@article{milne,
        author = {Milne, James S.},
        title = {Lectures on Etale Cohomology},
        url = {https://www.jmilne.org/math/CourseNotes/LEC.pdf},
        year = {2013}}

\end{document}